\documentclass[reqno,11pt]{article}
\usepackage{graphicx, amsmath, amsthm,amssymb}

\def\cemetery{\Delta}
\def\QQ{Z}
\def\KK{Q}
\def\Kbar{\overbar{Q}}

\input{stochbif.sty}

\begin{document}


\title{Mixed-mode oscillations and interspike interval statistics in the
stochastic FitzHugh--Nagumo model}
\author{Nils Berglund\thanks{MAPMO,
CNRS -- UMR 6628, Universit\'{e} d'Orl\'{e}ans, F\'{e}d\'{e}ration Denis
Poisson -- FR 2964, B.P. 6759, 45067 Orl\'{e}ans Cedex 2, France.} 
\thanks{Supported by ANR project MANDy, Mathematical Analysis of Neuronal
Dynamics, ANR-09-BLAN-0008-01.}
~and Damien Landon$^{*\dagger}$}
\date{}   

\maketitle

\begin{abstract}
We study the stochastic FitzHugh--Nagumo equations, modelling the
dynamics of neuronal action potentials, in parameter regimes characterised by
mixed-mode oscillations. The interspike time interval is related to the random
number of small-amplitude oscillations separating consecutive spikes. We prove
that this number has an asymptotically geometric distribution, whose parameter
is related to the principal eigenvalue of a substochastic Markov chain. We
provide rigorous bounds on this eigenvalue in the small-noise regime, and derive
an approximation of its dependence on the system's parameters for a large range
of noise intensities. This yields a precise description of the probability
distribution of observed mixed-mode patterns and interspike intervals. 
\end{abstract}

\noindent
{\it Date.\/}
May 6, 2011. Revised version, April 5, 2012. 




\section{Introduction}
\label{sec_intro}

Deterministic conduction-based models for action-potential generation in neuron
axons have been much studied for over half a century. In particular, the
four-dimensional Hodgkin--Huxley equations~\cite{HodgkinHuxley52} have been
extremely successful in reproducing the observed behaviour. Of particular
interest is the so-called excitable regime, when the neuron is at rest, but
reacts sensitively and reliably to small external perturbations, by emitting a
so-called spike. Much research efforts have been concerned with
the effect of deterministic perturbations, though the inclusion of random
perturbations in the form of Gaussian noise goes back at least
to~\cite{GersteinMandelbrot64}. A detailed account of different
models for stochastic perturbations and their effect on single
neurons can be found in~\cite{Tuckwell}. Characterising the influence of noise
on the spiking behaviour amounts to solving a stochastic first-exit
problem~\cite{Tuckwell75}.
Such problems are relatively well understood in dimension one, in particular
for the Ornstein--Uhlenbeck
process~\cite{CapocelliRicciardi71,Tuckwell77,RicciardiSacerdote80}. 
In higher dimensions, however, the situation is much more involved, and
complicated patterns of spikes can appear. See for
instance~\cite{TanabePakdaman2001,TakahataTanabePakdaman2002,Rowat_2007} for
numerical studies of the effect of noise on the interspike interval
distribution in the Hodgkin--Huxley equations. 

Being four-dimensional, the Hodgkin--Huxley equations are notoriously difficult
to study already in the deterministic case. For this reason, several simplified
models have been introduced. In particular, the two-dimensional FitzHugh--Nagumo
equations~\cite{Fitzhugh,Fitzhugh61,Nagumo62}, which generalise the Van der Pol
equations, are able to reproduce one type of excitability, which is associated
with a Hopf bifurcation (excitability of type II~\cite{Izhikevich00}). 

The effect of noise on the FitzHugh--Nagumo equations or similar excitable
systems has been studied
numerically~\cite{Longtin,KosmidisPakdaman,KosmidisPakdaman2006,Turcotte_2008,
Borowski_Kuske_etal_2011} and using approximations based on the Fokker--Planck
equations~\cite{Lindner_Schimansky_1999,SimpsonKuske_2011}, moment
methods~\cite{TanabePakdaman_PRE2001,Tuckwell_etal_2003}, and the Kramers
rate~\cite{Longtin2000}. Rigorous results on the oscillatory (as opposed to
excitable) regime have been obtained using the theory of large
deviations~\cite{MuratovVanden-EijndenE,DossThieullen2009} and by a detailed
description of sample paths near so-called canard solutions~\cite{Sowers08}. 

An interesting connection between excitability and mixed-mode oscillations
(MMOs) was observed by Kosmidis and
Pakdaman~\cite{KosmidisPakdaman,KosmidisPakdaman2006}, and further analysed by
Muratov and Vanden-Eijnden~\cite{MuratovVandeneijnden2007}. MMOs are patterns of
alternating large-
and small-amplitude oscillations (SAOs), which occur in a variety of chemical
and biological systems 
\cite{DegnOlsenPerram,HudsonHartMarinko,PetrovScottShowalter,Dicksonetal1}. In
the deterministic case, at least three variables are
necessary to reproduce such a behaviour (see~\cite{KuehnMMO} for a recent review
of deterministic mechanisms responsible for MMOs). As observed 
in~\cite{KosmidisPakdaman,KosmidisPakdaman2006,MuratovVandeneijnden2007}, 
in the presence of noise, already the two-dimensional
FitzHugh--Nagumo equations can display MMOs. In fact, depending on the three
parameters noise intensity $\sigma$, timescale separation $\eps$ and distance to
the Hopf bifurcation $\delta$, a large variety of behaviours can be observed,
including sporadic single spikes, clusters of spikes, bursting relaxation
oscillations and coherence resonance. \figref{fig_bif_diagram} shows a
simplified version of the phase diagram proposed
in~\cite{MuratovVandeneijnden2007}. 

\begin{figure}
\centerline{\includegraphics*[clip=true,width=100mm]
{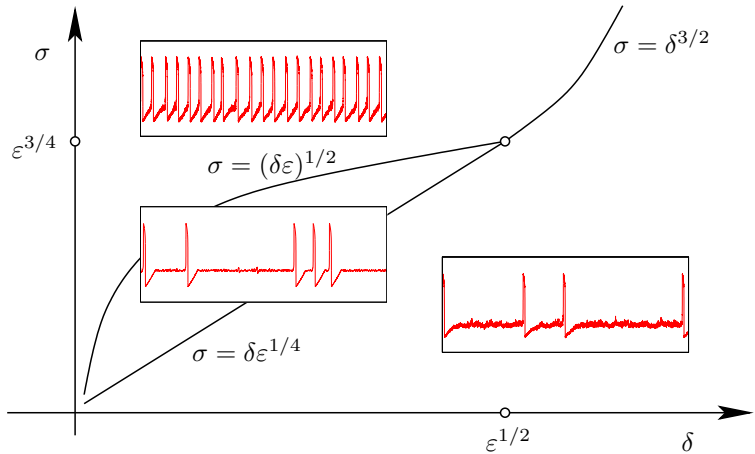}
}
 \figtext{
 	\writefig	11.3	0.3	$\delta$
 	\writefig	8.7	0.3	$\eps^{1/2}$
 	\writefig	2.4	4.2	$\eps^{3/4}$
 	\writefig	2.7	5.5	$\sigma$
 	\writefig	5.0	4.0	$\sigma=(\delta\eps)^{1/2}$
 	\writefig	4.8	1.5	$\sigma=\delta\eps^{1/4}$
 	\writefig	10.4	5.6	$\sigma=\delta^{3/2}$
 } 
\vspace{2mm}
\caption[]{Schematic phase diagram of the stochastic FitzHugh--Nagumo
equations. The parameter $\sigma$ measures the noise intensity, $\delta$
measures the distance to the singular Hopf bifurcation, and $\eps$ is the
timescale separation. The three main regimes are characterised be rare isolated
spikes, clusters of spikes, and repeated spikes.}
\label{fig_bif_diagram}
\end{figure}

In the present work, we build on ideas of~\cite{MuratovVandeneijnden2007} to
study in more detail the transition from rare individual spikes, through
clusters of spikes and all the way to bursting relaxation oscillations. We begin
by giving a precise mathematical definition of a random variable $N$ counting
the number of SAOs between successive spikes. It is related to a substochastic
continuous-space Markov chain, keeping track of the amplitude of each SAO. We
use this Markov process to prove that the distribution of $N$ is asymptotically
geometric, with a parameter directly related to the principal eigenvalue of the
Markov chain (Theorem~\ref{thm_geometric}). A similar behaviour has been
obtained for the length of bursting relaxation oscillations in a
three-dimensional system~\cite{HiczenkoMedvedev2009}. In the weak-noise regime,
we derive rigorous bounds on the principal eigenvalue and on the expected number
of SAOs (Theorem~\ref{thm_weak}). Finally, we derive an approximate expression
for the distribution of $N$ for all noise intensities up to the regime of
repeated spiking (Proposition~\ref{prop_trans}). 

The remainder of this paper is organised as follows. Section~\ref{sec_results}
contains the precise definition of the model. In Section~\ref{sec_N}, we define
the random variable $N$ and derive its general properties.
Section~\ref{sec_weak} discusses the weak-noise regime, and
Section~\ref{sec_trans} the transition from weak to strong noise. We present
some numerical simulations in Section~\ref{sec_sim}, and give concluding
remarks in Section~\ref{sec_conc}. A number of more technical computations are
contained in the appendix.

\subsubsection*{Acknowledgements}
It's a pleasure to thank Barbara Gentz, Simona Mancini and Khashayar Pakdaman
for numerous inspiring discussions, Athanasios Batakis for advice on
harmonic measures, and Christian Kuehn for sharing his deep knowledge on
mixed-mode oscillations. We also thank the two anonymous referees for providing
constructive remarks which helped to improve the manuscript. 
NB was partly supported by the International Graduate
College \lq\lq Stochastics and real world models\rq\rq\ at University of
Bielefeld. NB and DL thank the CRC 701 at University of Bielefeld for
hospitality.


\section{Model}
\label{sec_results}


We will consider random perturbations of the deterministic FitzHugh--Nagumo
equations given by 
\begin{equation}
 \label{res01}
\begin{split}
\eps \dot{x} &= x - x^3 + y \\
\dot{y} &= a - b x - c y\;,
\end{split} 
\end{equation} 
where $a, b, c\in\R$ and $\eps>0$ is a small parameter. The smallness of $\eps$
implies that $x$ changes rapidly, unless the state $(x,y)$ is close to the
nullcline $\set{y=x^3-x}$. Thus System~\eqref{res01} is called a fast-slow
system, $x$ being the fast variable and $y$ the slow one. 

We will assume that $b\neq0$. Scaling time by a factor $b$ and redefining the
constants $a$, $c$ and $\eps$, we can and will replace $b$ by $1$
in~\eqref{res01}. If $c\geqs0$ and $c$ is not too large, the nullclines
$\set{y=x^3-x}$ and $\set{a=x+cy}$ intersect in a unique stationary point $P$.
If $c<0$, the nullclines intersect in $3$ aligned points, and we let $P$ be the
point in the middle.  It can be written $P=(\alpha,\alpha^3-\alpha)$, where
$\alpha$ satisfies the relation
\begin{equation}
 \label{res01a}
\alpha + c(\alpha^3-\alpha) = a\;. 
\end{equation} 
The Jacobian matrix of the vector field at $P$ is given by 
\begin{equation}
 \label{res01b} 
J = 
\begin{pmatrix}
\dfrac{1-3\alpha^2}{\eps} & \dfrac{1}{\eps} \\
\vrule height 18pt depth 6pt width 0pt
-1 & -c
\end{pmatrix}\;.
\end{equation} 
It has determinant $(1-c(1-3\alpha^2))/\eps$
and trace 
\begin{equation}
 \label{res01c} 
\Tr J = \frac{3(\alpha_*^2-\alpha^2)}{\eps}\;, 
\qquad\text{where } 
\alpha_* = \sqrt{\frac{1-c\eps}{3}}\;.
\end{equation} 
Thus if
$\abs{c}<1/\sqrt{\eps}$, $J$ admits a pair of conjugate
imaginary eigenvalues when $\alpha=\pm\alpha_*$. 
Furthermore, the eigenvalues' real parts are of order $(\alpha_*-\alpha)/\eps$
near $\alpha_*$. The system undergoes so-called
singular Hopf bifurcations~\cite{BaerErneuxI,BaerErneuxII,Braaksma}  at
$\alpha=\pm\alpha_*$.

\begin{figure}
\centerline{\includegraphics*[clip=true,width=100mm]
{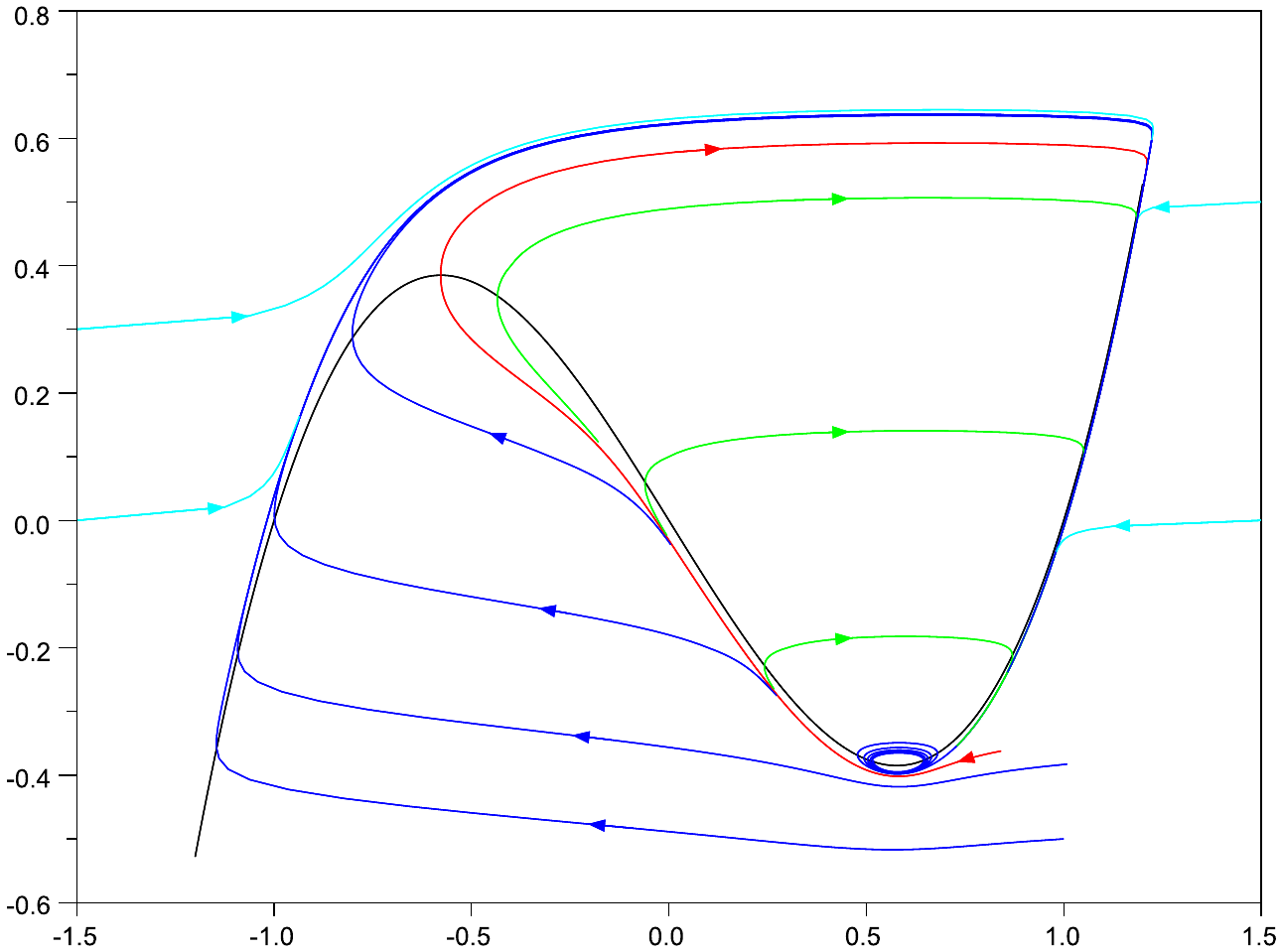}
}
 \figtext{ 
 	\writefig	12.7	0.6	$x$
 	\writefig	2.0	7.5	$y$
 }
\caption[]{Some orbits of the deterministic FitzHugh--Nagumo
equations~\eqref{res01} for parameter values $\eps=0.05$, $a=0.58$, $b=1$ and
$c=0$. The black curve is the nullcline, and the red orbit is the separatrix.
}
\label{fig_FHN_det}
\end{figure}

We are interested in the excitable regime, when $\alpha-\alpha_*$ is small and
positive. In this situation, $P$ is a stable stationary point, corresponding to
a quiescent neuron. However, a small perturbation of the initial condition,
e.g.\ a slight decrease of the $y$-coordinate, causes the system to make a large
excursion to the region of negative $x$, before returning to $P$
(\figref{fig_FHN_det}). This
behaviour corresponds to a spike in the neuron's membrane potential, followed
by a return to the quiescent state. One can check from the expression of the
Jacobian matrix that $P$ is a focus for $\alpha-\alpha_*$  of order
$\sqrt{\eps}$. Then return to rest involves small-amplitude oscillations (SAOs),
of exponentially decaying amplitude.

For later use, let us fix a particular orbit delimiting the spiking and
quiescent regimes, called \defwd{separatrix}. An arbitrary but convenient choice
for the separatrix is the negative-time orbit of the local maximum
$(-1/\sqrt{3},2/(3\sqrt{3}))$ of the nullcline (\figref{fig_FHN_det}). The main
results will not depend on the detailed choice of the separatrix. 

In this work we consider random perturbations of the deterministic
system~\eqref{res01} by Gaussian white noise. They are described by the system
of It\^o stochastic differential equations (SDEs) 
\begin{equation}
 \label{res02}
\begin{split}
\6 x_t &= \frac{1}{\eps} (x_t-x_t^3+y_t) \6t 
+ \frac{\sigma_1}{\sqrt{\eps}} \6W_t^{(1)} \\
\6 y_t &= (a - x_t - cy_t) \6t 
+ \sigma_2 \6W_t^{(2)}\;,
\end{split}
\end{equation} 
where $W_t^{(1)}$ and $W_t^{(2)}$ are independent, standard Wiener processes,
and $\sigma_1,\sigma_2 >0$. The parameter $a$ will be our bifurcation
parameter, while $c$ is assumed to be fixed, and small enough for the system to
operate in the excitable regime. The scaling in $1/\sqrt{\eps}$ of the noise
intensity in the first equation is chosen because the variance of the noise
term then grows like $\sigma_1^2 t/\eps$, so that $\sigma_1^2$ measures the
ratio of diffusion and drift for the $x$-variable, while $\sigma_2^2$ plays the
same r\^ole for the $y$-variable. 

\figref{fig_timeseries} shows a selection of time series for the stochastic
FitzHugh--Nagumo equations~\eqref{res02}. For the chosen parameter values, one
can clearly see large-amplitude spikes, separated by a random number of SAOs. 
Note the rather large variability of the SAOs' amplitude.  

\begin{figure}
\centerline{\includegraphics*[clip=true,width=70mm]
{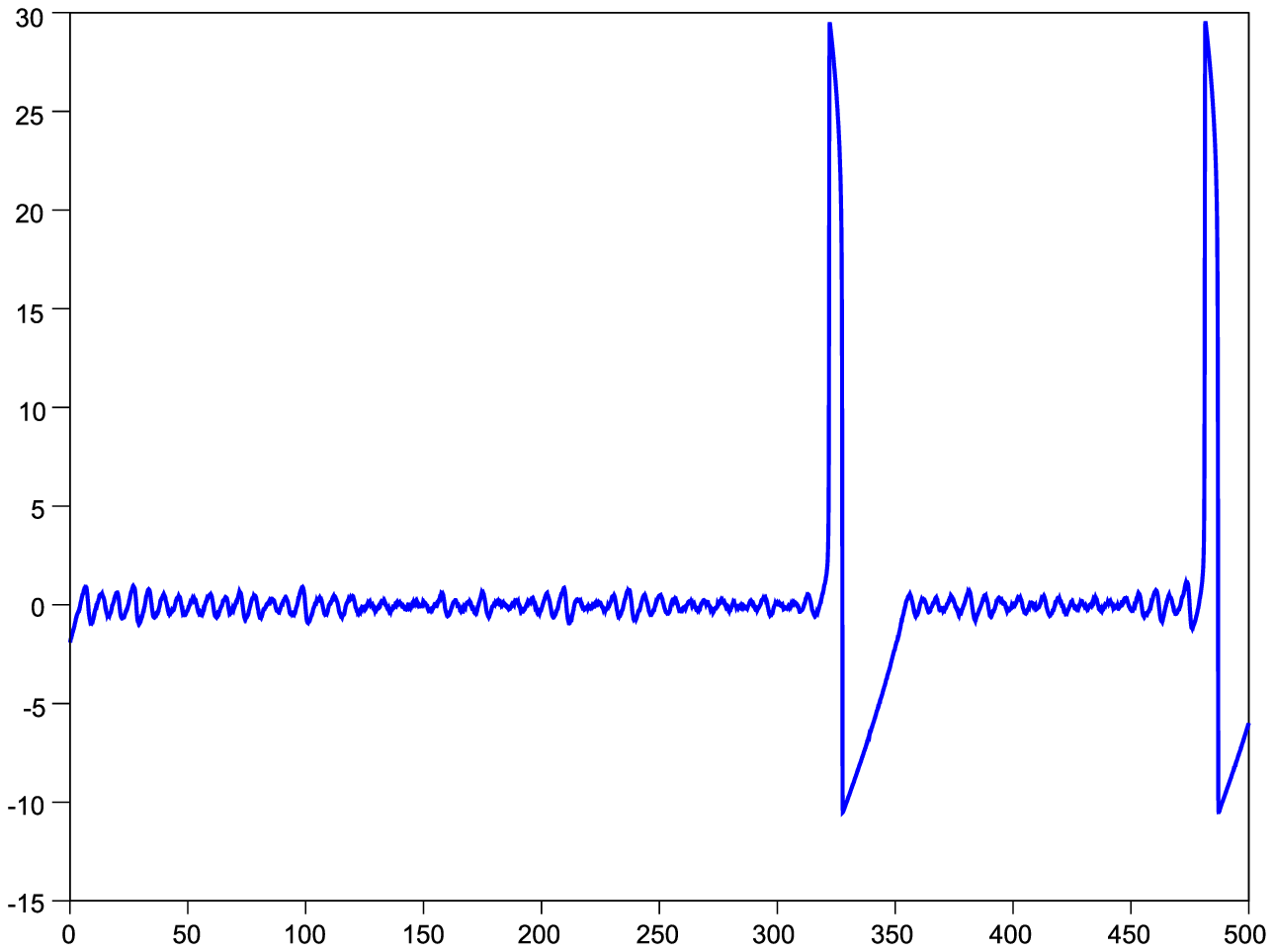}
\includegraphics*[clip=true,width=70mm]{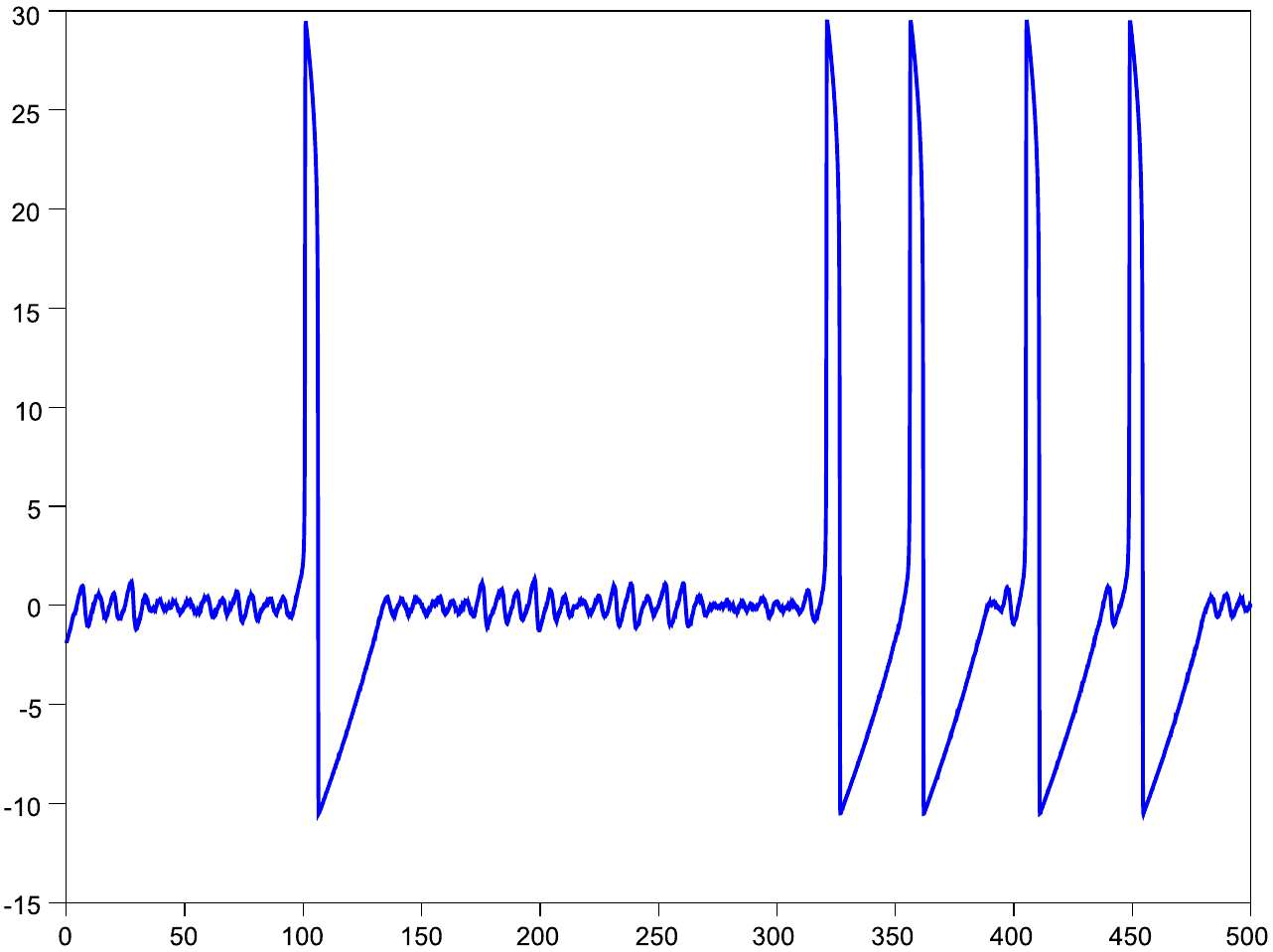}
}
\vspace{2mm}
\centerline{\includegraphics*[clip=true,width=70mm]
{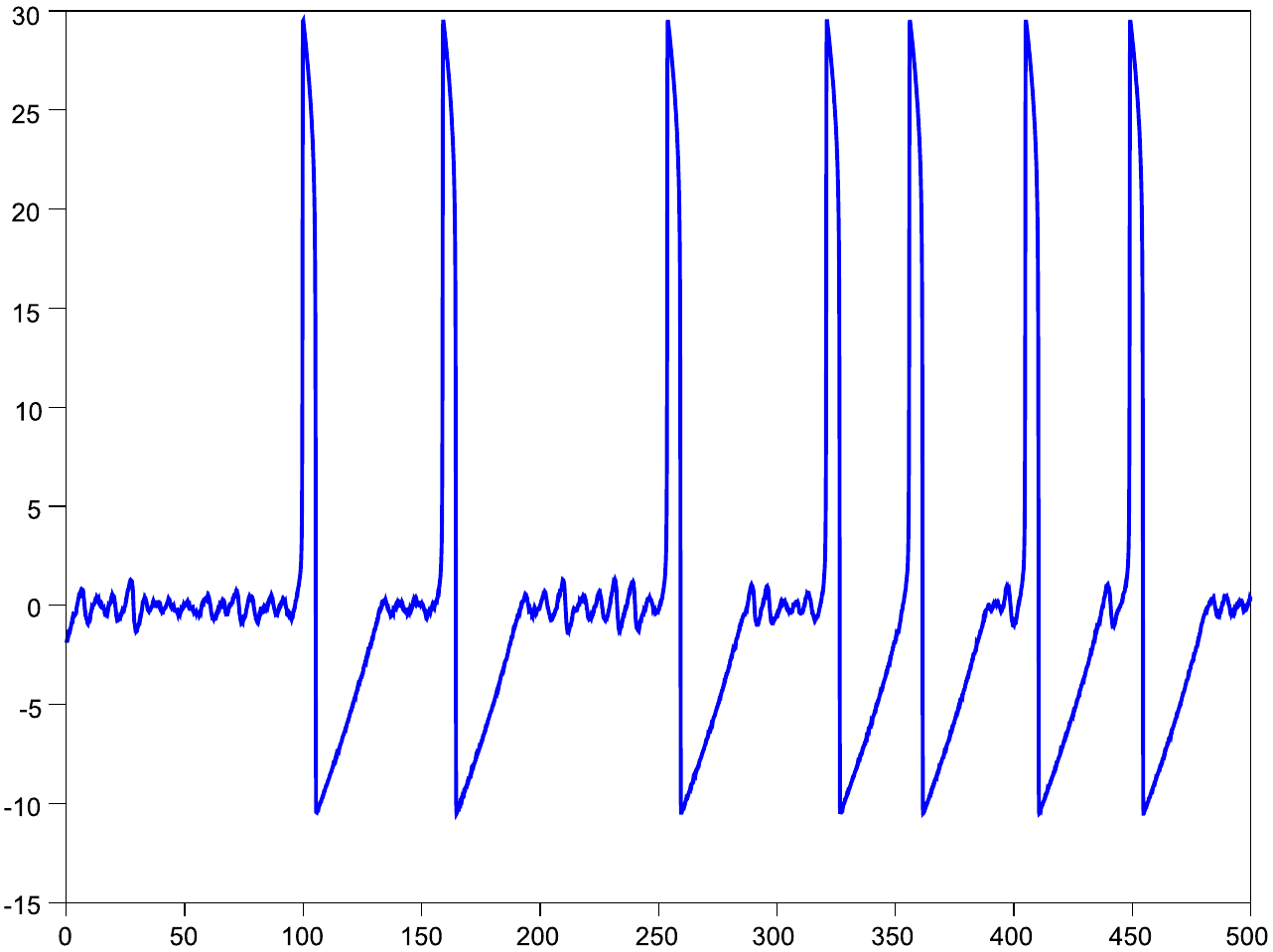}
\includegraphics*[clip=true,width=70mm]{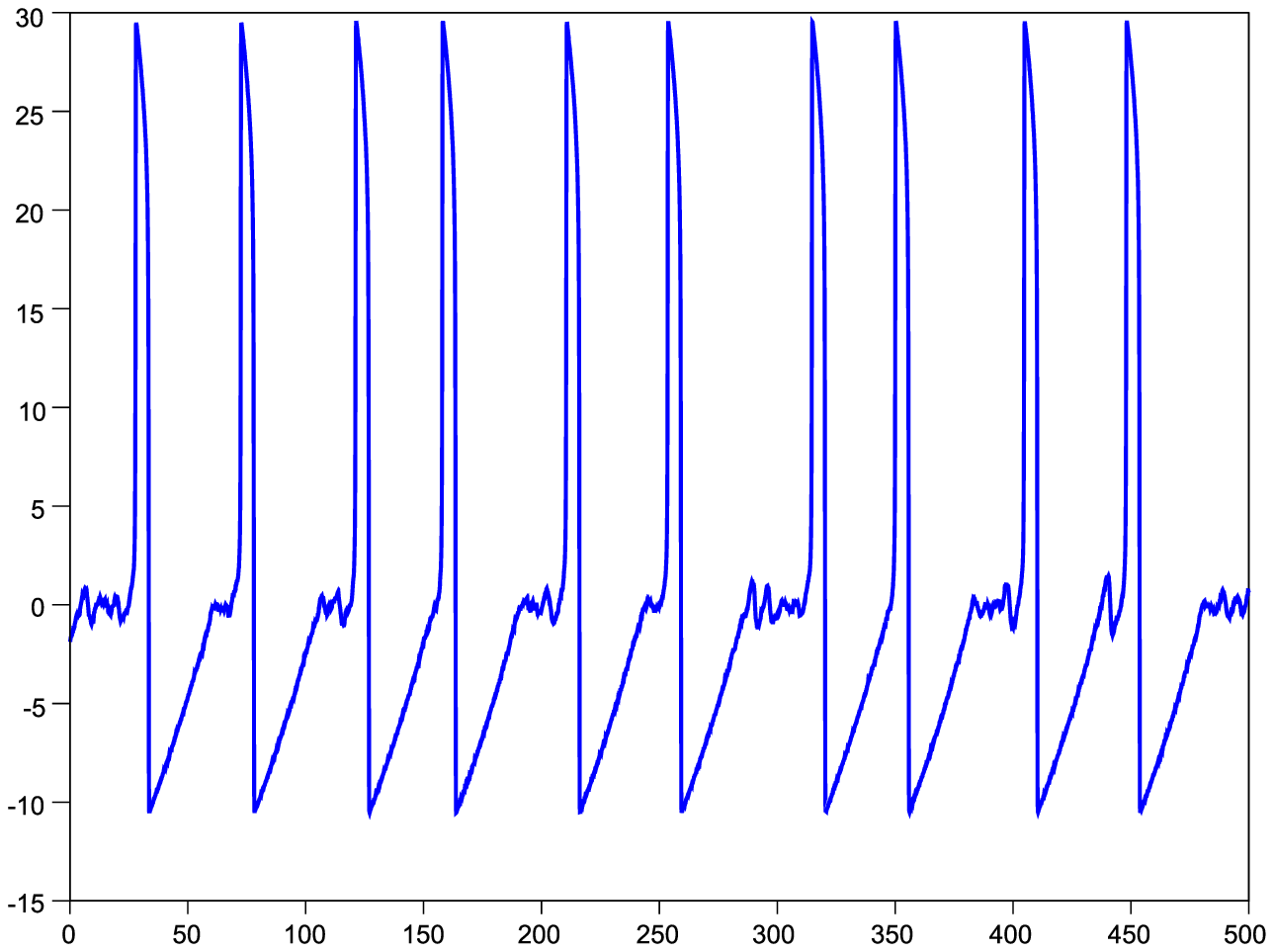}
}
 \figtext{ 
 	\writefig	6.9	0.9	$t$
 	\writefig	0.7	5.3	$\xi$
 	\writefig	14.0	0.9	$t$
 	\writefig	7.8	5.3	$\xi$
 	\writefig	6.9	6.35	$t$
 	\writefig	0.7	10.75	$\xi$
 	\writefig	14.0	6.35	$t$
 	\writefig	7.8	10.75	$\xi$
 }
 \vspace{2mm}
\caption[]{Examples of time series of the stochastic FitzHugh--Nagumo
equations~\eqref{res02}. The plots show the functions $t\mapsto\xi_t$, where
the variable $\xi$ is defined in Section~\ref{sec_weak}.
Parameter values are $\eps=0.01$ and $\delta=3\cdot 10^{-3}$ for the top row,
$\delta=5\cdot 10^{-3}$ for the bottom row. The noise intensities are 
given by $\sigma_1=\sigma_2=1.46\cdot 10^{-4}$, $1.82\cdot 10^{-4}$, 
$2.73\cdot 10^{-4}$ and $3.65\cdot 10^{-4}$. 
}
\label{fig_timeseries}
\end{figure}


\section{The distribution of small-amplitude oscillations}
\label{sec_N} 


In this section we define and analyse general properties of an integer-valued
random variable $N$, counting the number of
small-amplitude oscillations the stochastic system performs between two
consecutive spikes. The definition is going to be topological, making our
results robust to changes in details of the definition. We start by fixing a
bounded set $\cD\subset\R^2$, with smooth boundary $\partial\cD$, containing the
stationary point $P$ and a piece of the separatrix (\figref{fig_definition_N}).
Any excursion of the sample path $(x_t,y_t)_t$ outside $\cD$ will be considered

\begin{figure}
\centerline{\includegraphics*[clip=true,width=130mm]{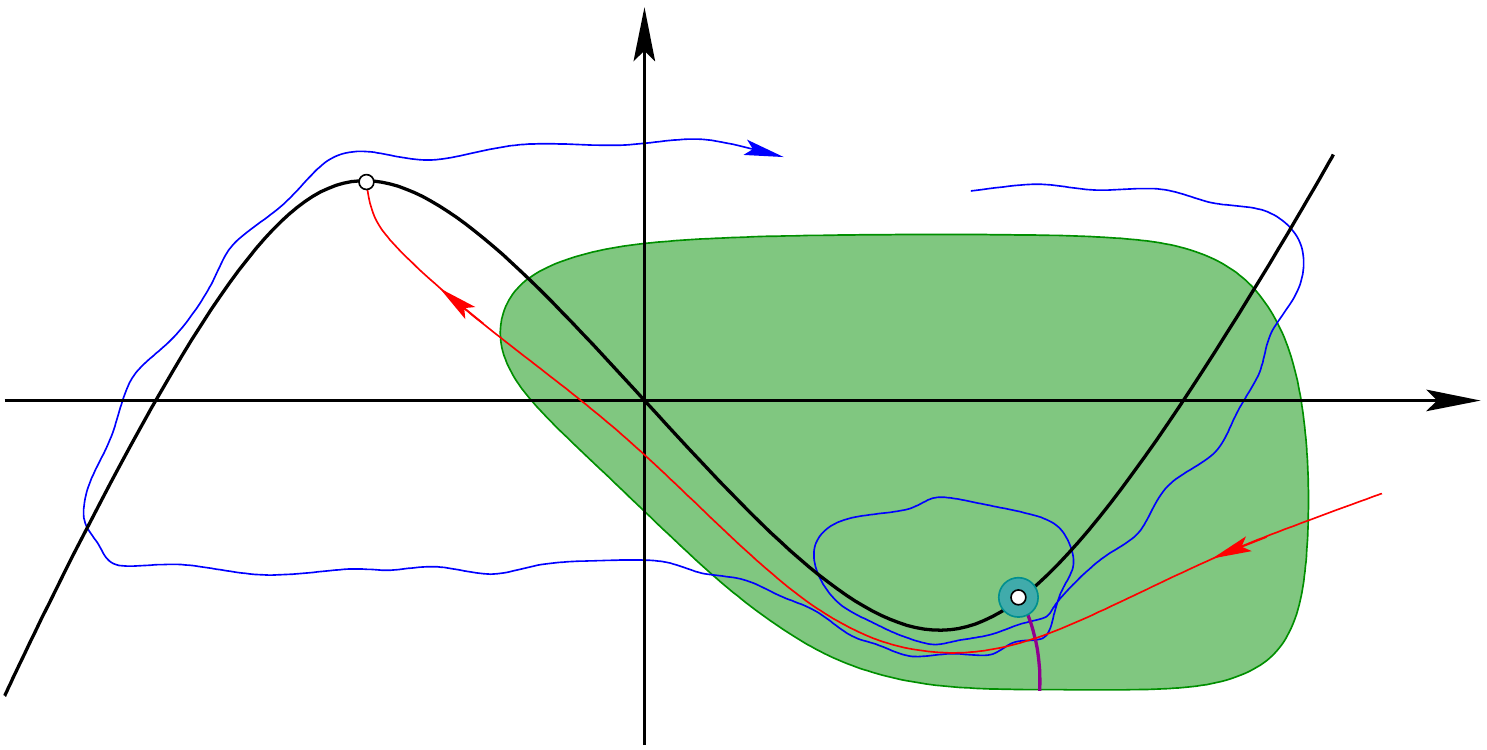}}
 \figtext{ 
	\writefig	12.9	3.7	$x$
	\writefig	6.2	6.1	$y$
	\writefig	7.5	4.4	$\cD$
	\writefig	9.9	0.6	$\cF$
	\writefig	11.9	0.9	$\partial\cD$
	\writefig	9.45	1.75	$P$
	\writefig	9.7	2.05	$\cB$
	\writefig	12.7	5.5	$y=x^3-x$
	\writefig	12.7	5.9	{nullcline}
	\writefig	13.1	2.7	{separatrix}
 }
 \vspace{2mm}
\caption[]{Definition of the number $N$ of SAOs. The sample path (blue) enters
the region $\cD$, and intersects twice the line $\cF$ before leaving $\cD$,
making another spike. Thus $N=2$ in this example. The separatrix is
represented in red.  
}
\label{fig_definition_N}
\end{figure}

To define $N$ precisely, we let $\cB$ be a small ball of radius $\rho>0$
centred in $P$. Then we draw a smooth curve $\cF$ from $\cB$ to
the boundary $\partial\cD$, which we parametrise by a variable
$r\in[0,1]$ proportional to arclength (the results will be independent,
however, of the choice of $\cF$ and of $r$). We extend the
parametrisation of $\cF$ to a polar-like parametrisation of all
$\cD\setminus\cB$, i.e.\ we choose a diffeomorphism $T:[0,1]\times\fS^1\to\cD$,
$(r,\varphi)\mapsto(x,y)$, where $T^{-1}(\cF)=\set{\varphi=0}$,
$T^{-1}(\partial\cD)=\set{r=0}$ and $T^{-1}(\partial\cB)=\set{r=1}$. We also
arrange that $\dot{\varphi}>0$ near $P$ for the deterministic flow. 

Consider the process $(r_t,\varphi_t)_t$ (where the angle $\varphi$ has been
lifted from $\fS^1$ to $\R$). Given an initial condition $(r_0,0)\in
T^{-1}(\cF)$ and an integer $M\geqs1$, we define the stopping time 
\begin{equation}
 \label{res03}
\tau = \inf\bigsetsuch{t>0}{\varphi_t\in\set{2\pi,-2M\pi} \text{ or } r_t \in
\set{0,1}}\;. 
\end{equation} 
There are four cases to consider:
\begin{itemiz}
\item	The case $r_\tau=0$ corresponds to the sample path $(x_t,y_t)$ leaving
$\cD$, and thus to a spike. This happens with strictly positive probability, by
ellipticity of the diffusion process~\eqref{res02}. In this situation, we set by
convention $N=1$. 
\item	In the case $\varphi_\tau=2\pi$ and  
$r_\tau\in(0,1)$, the sample path has returned to $\cF$ after
performing a complete revolution around $P$, staying all the while in
$\cD\setminus\cB$.
This corresponds to an SAO, and thus $N\geqs 2$. 
\item	The case $r_\tau=1$ corresponds to the sample path entering $\cB$,
which we consider as the neuron reaching the quiescent state. In that case
we simply wait until the state leaves $\cB$ again and
either hits $\cF$ or leaves $\cD$. 
\item	The case $\varphi_\tau=-2M\pi$ and $r_\tau\in(0,1)$ represents the
(unlikely) event that the sample path winds $M$ time around $P$ in the wrong
direction. We introduce this case for technical reasons only, as we will need
$\tau$ to be the first-exit time of a bounded set. For simplicity, we also
consider this situation as one SAO. 
\end{itemiz}
As long as $r_\tau\in(0,1)$, we repeat the above procedure, incrementing
$N$ at each iteration. This yields a sequence $(R_0,R_1,\dots,R_{N-1})$ of
random variables, describing the position of the successive intersections of the
path with $\cF$, separated by rotations around $P$, and up to the first exit
from $\cD$. 

\begin{remark}
\label{rem_defN} 
The above definition of $N$ is the simplest one to analyse mathematically. 
There are several possible alternatives. One can, for instance, introduce a 
quiescent state $(x,y)\in\cB$, and define $N$ as the number of SAOs until the
path either leaves $\cD$ or enters $\cB$. This would allow to keep track of the
number of SAOs between successive spikes and/or quiescent phases. 
Another possibility would be to count rotations in both the positive and
negative directions. For simplicity, we stick here to the above simplest
definition of $N$, but we plan to make a more refined study in a future
work.
\end{remark}

The sequence $(R_n)_n$ forms a substochastic Markov chain on $E=(0,1)$, with
kernel 
\begin{multline}
 \label{res04}
K(R,A) = \bigpcond{\varphi_\tau=\in\set{2\pi,-2M\pi},r_\tau\in
A}{\varphi_0=0,r_0=R}\;, 
\\ 
R\in E, A\subset E \text{ a Borel set\;.}
\end{multline} 
The Markov chain is substochastic because $K(R,E)<1$, due to the positive
probability of sample paths leaving $\cD$. We can make it
stochastic in the usual way by adding a cemetery state $\cemetery$ to $E$ (the
spike), and setting $K(R,\cemetery)=1-K(R,E)$, $K(\cemetery,\cemetery)=1$
(see~\cite{OreyBook,Nummelin84} for the general theory of such processes). 

The number of SAOs is given by 
\begin{equation}
 \label{res05}
N = \inf\bigsetsuch{n\geqs 1}{R_n=\cemetery} \in\N\cup\set{\infty} 
\end{equation} 
(we set $\inf\emptyset=\infty$). 
A suitable extension of the well-known Perron--Frobenius theorem 
(see \cite{Jentzsch1912,KreinRutman1950,Birkhoff1957}) shows that $K$ admits a
maximal eigenvalue $\lambda_0$, which is real and simple. It is called the
\defwd{principal eigenvalue}\/ of $K$. If there exists a probability measure
$\pi_0$ such that $\pi_0K=\lambda_0\pi_0$, it is called
the \defwd{quasi-stationary distribution (QSD)}\/ of
the kernel~$K$~\cite{Seneta_VereJones_1966}.


Our first main result gives qualitative properties of the distribution of $N$
valid in all parameter regimes with nonzero noise. 

\begin{theorem}[General properties of $N$]
\label{thm_geometric} 
Assume that $\sigma_1,\sigma_2>0$. Then for any initial distribution $\mu_0$ of
$R_0$ on the curve $\cF$,
\begin{itemiz}
\item	the kernel $K$ admits a quasi-stationary distribution $\pi_0$;
\item	the associated principal eigenvalue
$\lambda_0=\lambda_0(\eps,a,c,\sigma_1,\sigma_2)$ is strictly smaller than
$1$;
\item	the random variable $N$ is almost surely finite;
\item	the distribution of $N$ is \lq\lq asymptotically geometric\rq\rq, that
is, 
\begin{equation}
 \label{res06}
\lim_{n\to\infty} \bigpcondin{\mu_0}{N=n+1}{N>n} = 1 - \lambda_0\;; 
\end{equation}  
\item	$\expecin{\mu_0}{r^N} < \infty$ for $r<1/\lambda_0$ and thus all
moments $\expecin{\mu_0}{N^k}$ of $N$ are finite.
\end{itemiz}
\end{theorem}
\begin{proof}
Let us denote by $K(x,\6y)$ the kernel defined in~\eqref{res04}. 
We consider $K$ as a bounded linear
operator on $L^\infty(E)$, acting on bounded measurable functions by 
\begin{equation}
 \label{markov12A}
 f(x)\mapsto (Kf)(x) = \int_E K(x,\6y) f(y) = \expecin{x}{f(R_1)} \;,
\end{equation} 
and as a bounded linear operator on $L^1(E)$, acting on finite measures $\mu$ by
\begin{equation}
 \label{markov12B}
\mu(A)\mapsto (\mu K)(A) = \int_E \mu(\6x) K(x,A) = \probin{\mu}{R_1\in A}\;.
\end{equation} 
To prove existence of a QSD $\pi_0$, we first have to establish a
uniform positivity condition on the kernel. 
Note that in $K(x,\6y)$,  $y$ represents the first-exit location from the domain
$\cG =
(0,1)\times(-2M\pi,2\pi)$, for an initial condition $(x,0)$, in case the exit
occurs through one of the lines $\varphi=-2M\pi$ or $\varphi=2\pi$. In harmonic
analysis, $K(x,\6y)$ is called the harmonic measure for the generator of the
diffusion in $\cG$ based at $(x,0)$. In the case of Brownian motion, it has been
proved in~\cite{Dahlberg1977} that sets of positive Hausdorff measure have
positive harmonic measure. This result has been substantially extended
in~\cite{BenArous_Kusuoka_Stroock_1984}, where the authors prove that for a
general class of hypoelliptic diffusions, the harmonic measure admits a smooth
density $k(x,y)$ with respect to Lebesgue measure $\6y$. Our diffusion process
being uniformly elliptic for $\sigma_1,\sigma_2>0$, it enters into the class of
processes studied in that work. 
Specifically,~\cite[Corollary~2.11]{BenArous_Kusuoka_Stroock_1984} shows that
$k(x,y)$ is smooth, and its derivatives are bounded by a function of the
distance from $x$ to $y$. This distance being uniformly bounded below by a
positive constant in our setting, there exists a constant $L\in\R_+$ such that 
\begin{equation}
 \label{geom_01}
\frac{\displaystyle\sup_{y\in E} k(x,y)}{\displaystyle\inf_{y\in E} k(x,y)}
\leqs L
\qquad\forall x\in E\;.
\end{equation} 
We set 
\begin{equation}
 \label{geom_02}
s(x) = \inf_{y\in E} k(x,y)\;. 
\end{equation} 
Then it follows that 
\begin{equation}
 \label{geom_03}
s(x) \leqs k(x,y) \leqs L s(x)  
\qquad\forall x, y\in E\;.
\end{equation} 
Thus the kernel $K$ fulfils the uniform positivity condition
\begin{equation}
 \label{geom_04} 
s(x)\nu(A) \leqs K(x,A) \leqs L s(x)\nu(A)
\qquad\forall x\in E\;,\forall A\subset E
\end{equation} 
for $\nu$ given by the Lebesgue measure. It follows by 
\cite[Theorem~3]{Birkhoff1957} that $K$ admits unique
positive left and right unit eigenvectors, and that the corresponding
eigenvalue $\lambda_0$ is real and positive. In other words, there is a measure
$\pi_0$ and a positive function $h_0$ such that $\pi_0K=\lambda_0\pi_0$ and 
$Kh_0=\lambda_0h_0$.  We normalise the eigenvectors in such a way that 
\begin{equation}
 \label{markov14}
\pi_0(E) = \int_E \pi_0(\6x) = 1\;, 
\qquad
\pi_0h_0 = \int_E \pi_0(\6x) h_0(x) = 1\;. 
\end{equation} 
Thus $\pi_0$ is indeed the quasistationary distribution of the Markov
chain. Notice that 
\begin{equation}
 \label{markov14B}
\lambda_0 = \lambda_0 \pi_0(E)
= \int_E \pi_0(\6x) K(x,E) \leqs \pi_0(E) = 1\;,  
\end{equation} 
with equality holding if and only if $K(x,E)=1$ for $\pi_0$-almost all $x\in E$.
In our case, $K(x,E)<1$ since $E$ has strictly smaller Lebesgue measure than
$\partial\cG$, and the density of the harmonic measure is bounded below. This
proves that $\lambda_0<1$. 

We denote by $K^n$ the $n$-step transition kernel, defined inductively by 
$K^1=K$ and 
\begin{equation}
 \label{markov14C}
 K^{n+1}(x,A) = \int_E K^n(x,\6y) K(y,A)\;. 
\end{equation} 
Lemma~3 in \cite{Birkhoff1957} shows that for any bounded
measurable function $f:E\to\R$, there exists
a finite constant $M(f)$ such that the spectral-gap estimate 
\begin{equation}
 \label{markov100}
\abs{(K^n f)(x) - \lambda_0^n (\pi_0 f) h_0(x)} \leqs M(f)
(\lambda_0\rho)^n h_0(x) 
\end{equation}
holds for some $\rho<1$ (note that this confirms that $\lambda_0$
is indeed the leading eigenvalue of $K$).
In order to prove that $N$ is almost surely finite, we first note that 
\begin{equation}
 \label{markov101}
\bigprobin{\mu_0}{N>n} = \bigprobin{\mu_0}{R_n\in E} 
= \int_E \mu_0(\6x) K^n(x,E) \;. 
\end{equation} 
Applying~\eqref{markov100} with $f=\vone$, the function identically equal to
$1$, we obtain 
\begin{equation}
 \label{markov102}
\lambda_0^n h_0(x) - M(\vone)(\lambda_0\rho)^n h_0(x) 
\leqs K^n(x,E) \leqs 
\lambda_0^n h_0(x) + M(\vone)(\lambda_0\rho)^n h_0(x)\;. 
\end{equation} 
Integrating against $\mu_0$, we get 
\begin{equation}
 \label{markov103}
\lambda_0^n \bigpar{1 - M(\vone)\rho^n} \leqs 
\bigprobin{\mu_0}{N>n}
\leqs \lambda_0^n \bigpar{1 + M(\vone)\rho^n}\;.
\end{equation} 
Since $\lambda_0<1$, it follows that $\lim_{n\to\infty}\probin{\mu_0}{N>n}=0$,
i.e., $N$ is almost surely finite. 

In order to prove that $N$ is asymptotically geometric, we have to control 
\begin{equation}
 \label{markov104}
\bigprobin{\mu_0}{N=n+1} = 
\int_E\int_E  \mu_0(\6x) K^n(x,\6y) \bigbrak{1-K(y,E)}\;. 
\end{equation} 
Applying~\eqref{markov100} with $f(y)=1-K(y,E)$, and using the fact
that 
\begin{equation}
 \label{markov105}
\pi_0 f = 1 - \int_E \pi_0(\6y)K(y,E) = 1 - \lambda_0 
\end{equation} 
yields
\begin{equation}
 \label{markov106}
\lambda_0^n \bigpar{1 - \lambda_0 - M(f)\rho^n} \leqs 
\bigprobin{\mu_0}{N=n+1} \leqs
\lambda_0^n \bigpar{ 1 - \lambda_0 + M(f)\rho^n}\;.
\end{equation} 
Hence~\eqref{res06} follows upon dividing~\eqref{markov106}
by~\eqref{markov103} and taking the limit $n\to\infty$.

Finally, the moment generating function $\expecin{\mu_0}{r^N}$ can be
represented as follows:
\begin{align}
\nonumber
\bigexpecin{\mu_0}{r^N} 
&= \sum_{n\geqs0}r^n\bigprobin{\mu_0}{N=n}
= \sum_{n\geqs0} \biggbrak{1+(r-1)\sum_{m=0}^{n-1} r^m} \bigprobin{\mu_0}{N=n}
\\
&= 1 + (r-1) \sum_{m\geqs0} r^m \bigprobin{\mu_0}{N>m}\;,
\label{markov107} 
\end{align}
which converges for $\abs{r\lambda_0}<1$ as a consequence of~\eqref{markov103}.
\end{proof}

Note that in the particular case where the initial distribution $\mu_0$ is
equal to the QSD $\pi_0$, the random variable $R_n$ has the law
$\mu_n=\lambda_0^n\pi_0$, and $N$ follows an exponential law of parameter
$1-\lambda_0$~:
\begin{equation}
 \label{res07}
\bigprobin{\pi_0}{N=n} = \lambda_0^{n-1}(1-\lambda_0)
\qquad\text{and}\qquad
\bigexpecin{\pi_0}{N} = \frac{1}{1-\lambda_0}\;. 
\end{equation} 
In general, however, the initial distribution $\mu_0$ after a spike will be far
from the QSD $\pi_0$, and thus the distribution of $N$ will only be
asymptotically geometric. 

Theorem~\ref{thm_geometric} allows to quantify the clusters of spikes observed
in~\cite{MuratovVandeneijnden2007}. To this end, we have to agree on a
definition of clusters of spikes. One may decide that a cluster is a sequence of
successive spikes between which there is no complete SAO, i.e. $N=1$ between
consecutive spikes. If the time resolution is not very good, however, one may
also fix a threshold SAO number $n_0\geqs 1$, and consider as a cluster a
succession of spikes separated by at most $n_0$ SAOs. Let $\mu_0^{(n)}$ be the
arrival distribution on $\cF$ of sample paths after the $n$th spike. Then the
probability to observe a cluster of length $k$ is given by 
\begin{equation}
 \label{res08}
\bigprobin{\mu_0^{(0)}}{N\leqs n_0} 
\bigprobin{\mu_0^{(1)}}{N\leqs n_0} \dots
\bigprobin{\mu_0^{(k-1)}}{N\leqs n_0} 
\bigprobin{\mu_0^{(k)}}{N> n_0}\;. 
\end{equation} 
In general, the consecutive spikes will not be independent, and thus the
distributions $\mu_0^{(n)}$ will be different. For small noise, however, after a
spike sample paths strongly concentrate near the stable branch of the
nullcline (see the discussion in~\cite[Section~3.5.2]{BG_neuro09}), and thus
we expect all $\mu_0^{(n)}$ to be very close to some constant
distribution $\mu_0$. This implies that the lengths of clusters of spikes also
follow an approximately geometric distribution~:
\begin{equation}
 \label{res09}
\bigprob{\text{cluster of length $k$}} \simeq p^k(1-p)
\qquad
\text{where }
p = \bigprobin{\mu_0}{N\leqs n_0}\;. 
\end{equation}


\section{The weak-noise regime}
\label{sec_weak}

In order to obtain more quantitative results, we start by transforming the
FitzHugh-Nagumo equations to a more
suitable form. The important part of dynamics occurs near the singular Hopf
bifurcation point. We carry out the transformation in four steps, the first
two of which have already been used in~\cite{BaerErneuxI,BaerErneuxII}~:
\begin{enum}
\item	An affine transformation $x=\alpha_* + u$, $y= \alpha_*^3 -
\alpha_*+ v$ translates the origin to the bifurcation point, and
yields, in the deterministic case~\eqref{res01}, the system
\begin{equation}
 \label{weak01}
\begin{split}
\eps\dot{u} &= v + c\eps u - 3\alpha_* u^2 - u^3\;,\\
\dot{v} &= \delta - u - cv\;,
\end{split} 
\end{equation} 
where $\delta = a - \alpha_* - c(\alpha_*^3-\alpha_*)$ is small and
positive. Note that~\eqref{res01a} implies that $\delta$ is of order
$\alpha-\alpha_*$ near the bifurcation point, and thus measures the distance to
the Hopf bifurcation. In particular, by~\eqref{res01c} the eigenvalues of the
Jacobian matrix $J$ have real parts of order $-\delta/\eps$.
\item	The scaling of space and time given by $u=\sqrt{\eps}\xi$, $v=\eps\eta$
and $t=\sqrt{\eps}t'$ yields 
\begin{equation}
 \label{weak02}
\begin{split}
\dot{\xi} &= \eta - 3\alpha_*  \xi^2 
+ \sqrt{\eps} \bigpar{c\xi -  \xi^3}\;,\\
\dot{\eta} &= \frac{\delta}{\sqrt{\eps}} - \xi - \sqrt{\eps}\,c\eta\;,
\end{split} 
\end{equation} 
where dots now indicate derivation with respect to $t'$. On this scale, the
nullcline $\dot\xi=0$ is close to the parabola $\eta=3\alpha_* \xi^2$. 

\item	The nonlinear transformation $\eta = 3\alpha_* \xi^2+z-1/(6\alpha_* )$
has the effect of straightening out the nullcline, and transforms~\eqref{weak02}
into 
\begin{equation}
 \label{weak03}
\begin{split}
\dot{\xi} &= z - \frac{1}{6\alpha_* } 
+ \sqrt{\eps} \bigpar{c\xi - \xi^3}\;,\\
\dot{z} &= \frac{\delta}{\sqrt{\eps}} - 6\alpha_* \xi z
+ \sqrt{\eps} \biggpar{6\alpha_*
\xi^4 + c\Bigpar{\frac1{6\alpha_*}-9\alpha_*\xi^2-z}}\;.
\end{split} 
\end{equation} 
\item	Finally, we apply the scaling $\xi\mapsto-\xi/3\alpha_* $,
$z\mapsto z/3\alpha_* $, which yields 
\begin{equation}
 \label{weak004}
\begin{split}
\dot{\xi} &= \frac{1}{2} - z +  
\sqrt{\eps} \biggpar{c\xi - \frac{1}{9\alpha_*^2} \xi^3}\;,\\
\dot{z} &= \mu + 2\xi z
+ \sqrt{\eps} \biggpar{\frac{2}{9\alpha_*^2} \xi^4 + c
\Bigpar{\frac12-3\xi^2-z}}\;,
\end{split} 
\end{equation} 
where the distance to the Hopf bifurcation is now measured by the parameter
\begin{equation}
 \label{weak04b}
\mu = \frac{3\alpha_*\delta}{\sqrt{\eps}}\;.
\end{equation}
\end{enum}

\begin{figure}
\centerline{\includegraphics*[clip=true,width=75mm]
{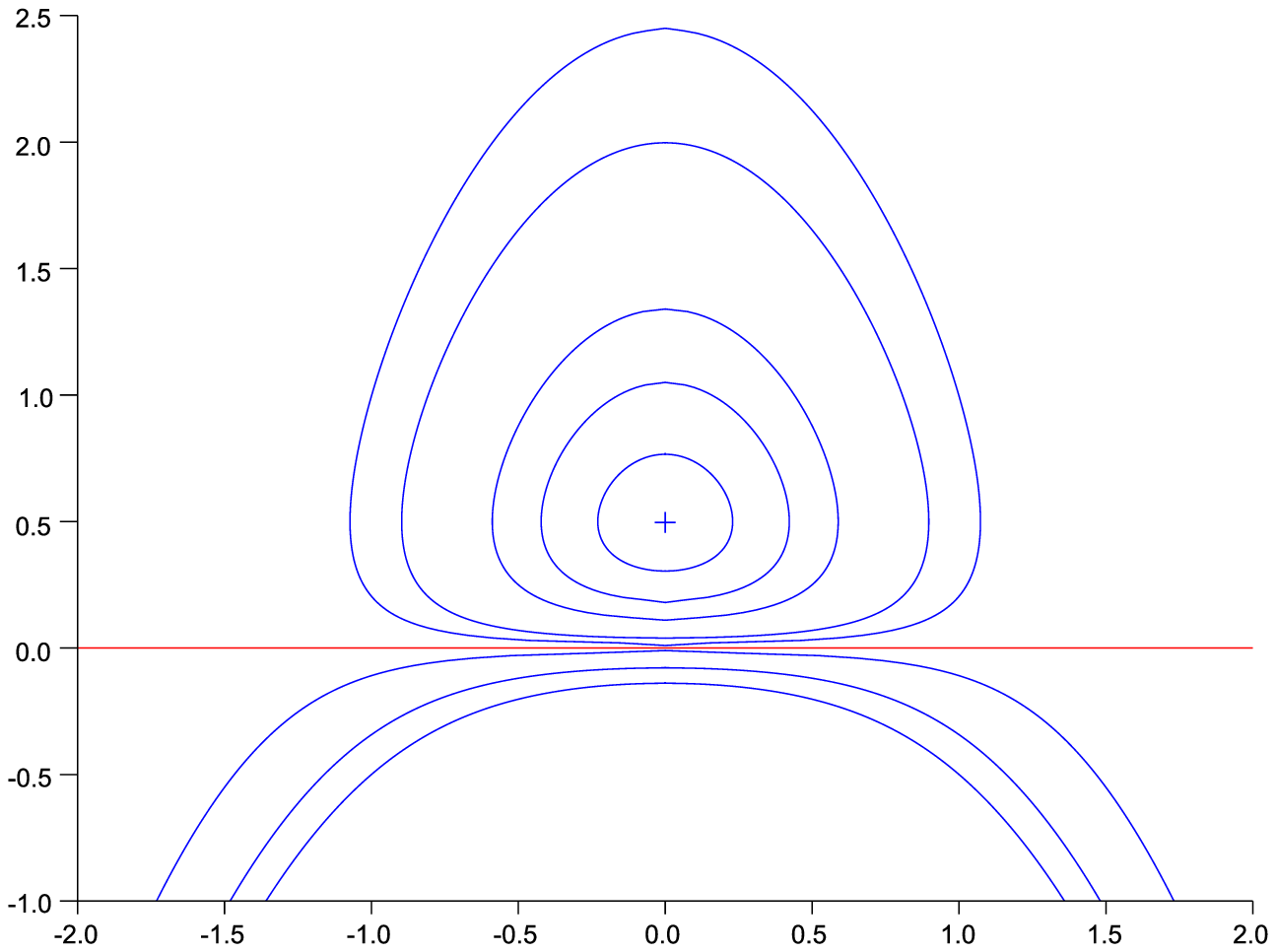}
\vspace{2mm}
\includegraphics*[clip=true,width=75mm]
{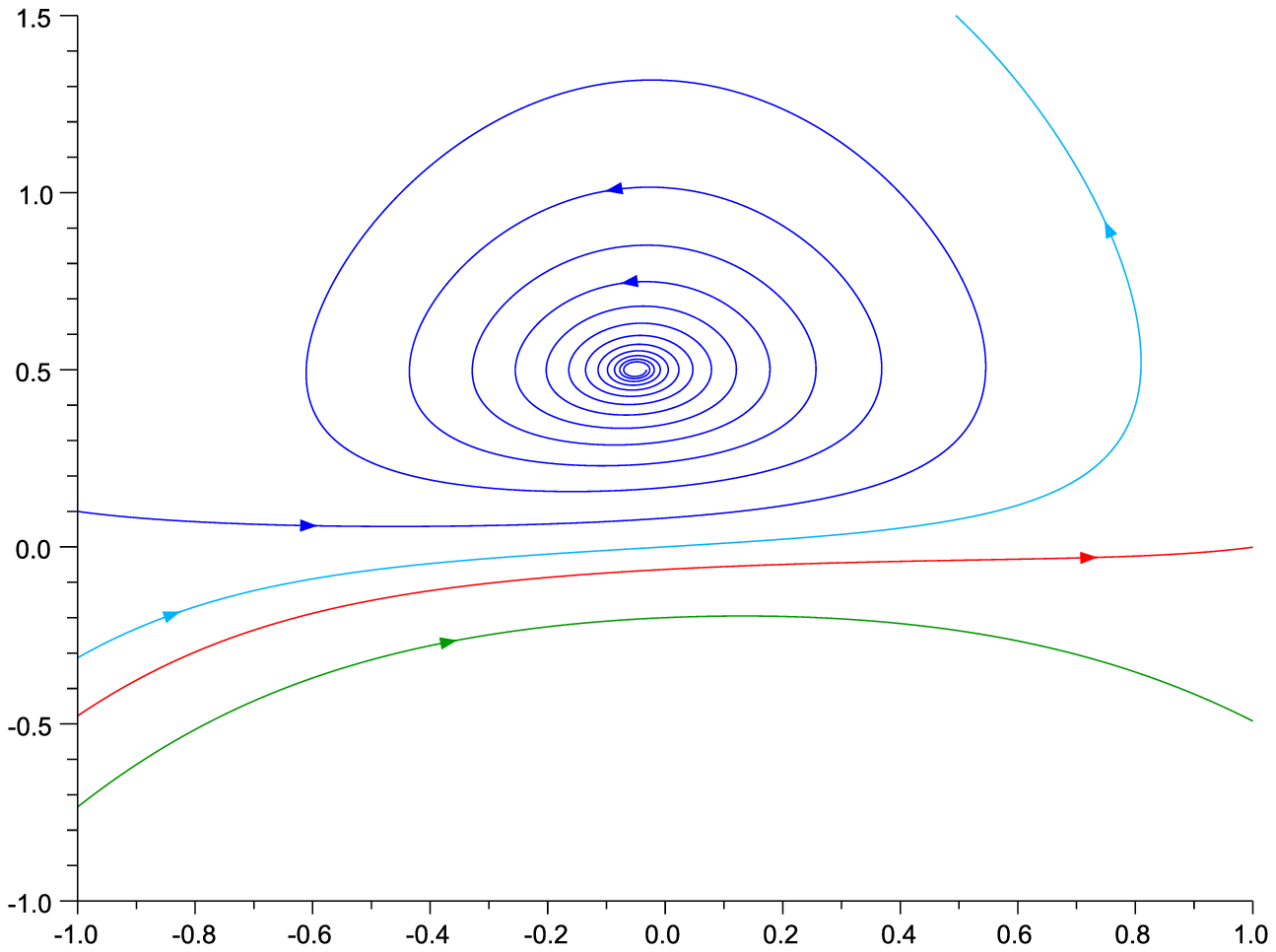}
}
 \figtext{ 
 	\writefig	-0.7	6.0	{\bf (a)}
 	\writefig	6.8	6.0	{\bf (b)}
 	\writefig	6.8	1.2	$\xi$
 	\writefig	0.3	5.9	$z$
 	\writefig	14.5	1.2	$\xi$
 	\writefig	8.0	5.9	$z$
 }
\caption[]{
{\bf (a)} Level curves of the first integral $\KK=2z \e^{-2z - 2\xi^2+1}$. 
{\bf (b)} Some orbits of the deterministic equations~\eqref{weak004} in
$(\xi,z)$-coordinates, for parameter values $\eps=0.01$, $\mu=0.05$ and $c=0$
(i.e.\ $\alpha_*=1/\sqrt{3}$). 
}
\label{fig_FHN_xiz}
\end{figure}

Let us first consider some special cases of the system~\eqref{weak004}. 

\begin{itemiz}
\item
If $\eps=\mu=0$, we obtain 
\begin{equation}
\label{weak0005}
\begin{split}
\dot{\xi } &= \dfrac{1}{2}-z \\
\dot{z} &= 2 \xi z\;. \\
\end{split}
\end{equation}
This system admits a first integral
\begin{equation}
\label{weak0006}
\KK = 2z \e^{-2z - 2\xi^2+1}\;,
\end{equation}
which is equivalent to the first integral found in~\cite{BaerErneuxII} (we
have chosen the normalisation in such a way that $\KK \in [0,1]$ for $z\geqs0$).
\figref{fig_FHN_xiz}a shows level curves of $\KK$. 
The level curve $\KK=0$ corresponds to the horizontal $z=0$. Values of
$\KK\in(0,1)$ yield periodic orbits, contained in the upper half plane
and encircling the stationary point $P=(0,1/2)$ which corresponds to
$\KK=1$. Negative values of $\KK$ yield unbounded orbits. Hence the horizontal
$z=0$ acts as the separatrix in these coordinates. 

\item	If $\mu>0$ and $\eps=0$, the stationary point $P$ moves to $(-\mu,1/2)$
and becomes a focus. The separatrix is deformed and now lies in the negative
half plane. It delimits the basin of attraction of $P$.  

\item	If $\mu>0$ and $0<\eps\ll1$, the dynamics does not change
\emph{locally}  (\figref{fig_FHN_xiz}b). The global dynamics, however, is
topologically the same as in
original variables. Therefore, orbits below the separatrix are no longer
unbounded, but get \lq\lq reinjected\rq\rq\ from the left after making a large
excursion in the plane, which corresponds to a spike. Orbits above the
separatrix still converge to $P$, in an oscillatory way. 
\end{itemiz}

Carrying out the same transformations for the stochastic system~\eqref{res02}
yields the following result (we omit the proof, which is a straightforward
application of It\^o's formula).

\begin{prop}
\label{prop_xiz}
In the new variables $(\xi,z)$, and on the new timescale $t/\sqrt{\eps}$, the
stochastic FitzHugh--Nagumo equations~\eqref{res02} take the form 
\begin{equation}
\label{Rxiz02}
\begin{split}
\6\xi_t &= \biggbrak{ \dfrac{1}{2} - z_t + \sqrt{\eps} \Bigpar{c\xi_t -
\frac{1}{9\alpha_*^2} \xi_t^3} }
\6t + \tilde{\sigma}_1  \6W_t^{(1)} \;,\\
\6z_t &= \biggbrak{ \tilde{ \mu}  + 2 \xi_t  z_t + \sqrt{\eps}
\biggpar{\frac{2}{9\alpha_*^2} \xi_t^4 + c
\Bigpar{\frac12-3\xi_t^2-z_t}}}  \6t - 2 \tilde{\sigma}_1  \xi_t \6W_t^{(1)}+
\tilde{\sigma}_2  
\6W_t^{(2)}\;,
\end{split}
\end{equation}
where
\begin{align}
\nonumber
\tilde{\sigma}_1 &= - 3\alpha_* \eps^{-3/4}  \sigma_1 \;,\\
\label{Rxiz19}
\tilde{\sigma}_2 &= 3\alpha_* \eps^{-3 /4}  \sigma_2 \;,\\
\nonumber
\tilde{\mu} &= \mu-\tilde{\sigma}_1^2
= \frac{3\alpha_*(\delta-3\alpha_*\sigma_1^2/\eps)}
{\sqrt{\eps}}\;.
\end{align} 
\end{prop}

Note the It\^o-to-Stratonovich correction term $-\tilde{\sigma}_1^2$ in
$\tilde\mu$, which implies that this parameter can be positive or negative,
depending on the value of $\delta$ and the noise intensity.

\begin{figure}
\centerline{\includegraphics*[clip=true,width=70mm]
{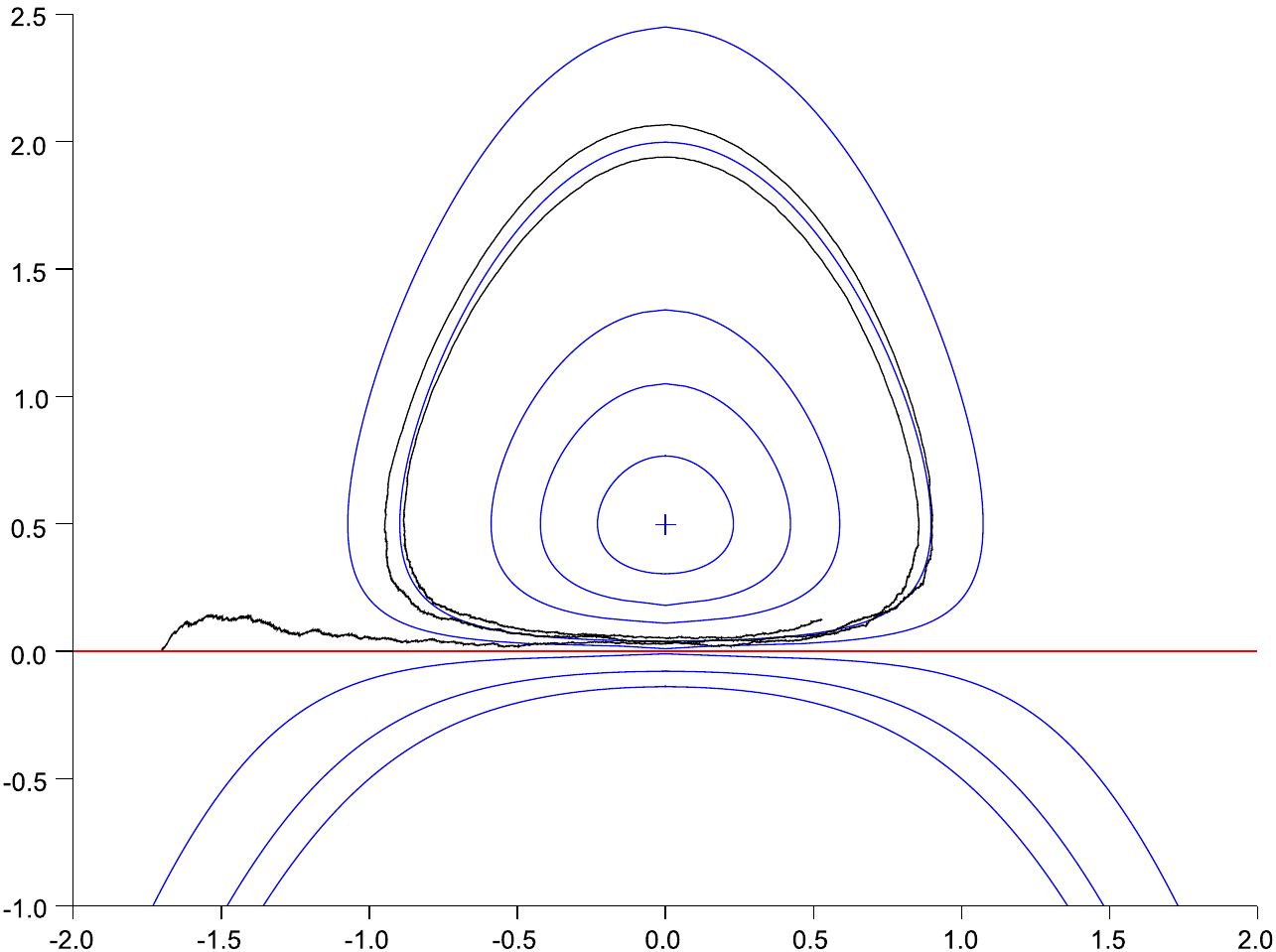}
\hspace{3mm}
\includegraphics*[clip=true,width=70mm]
{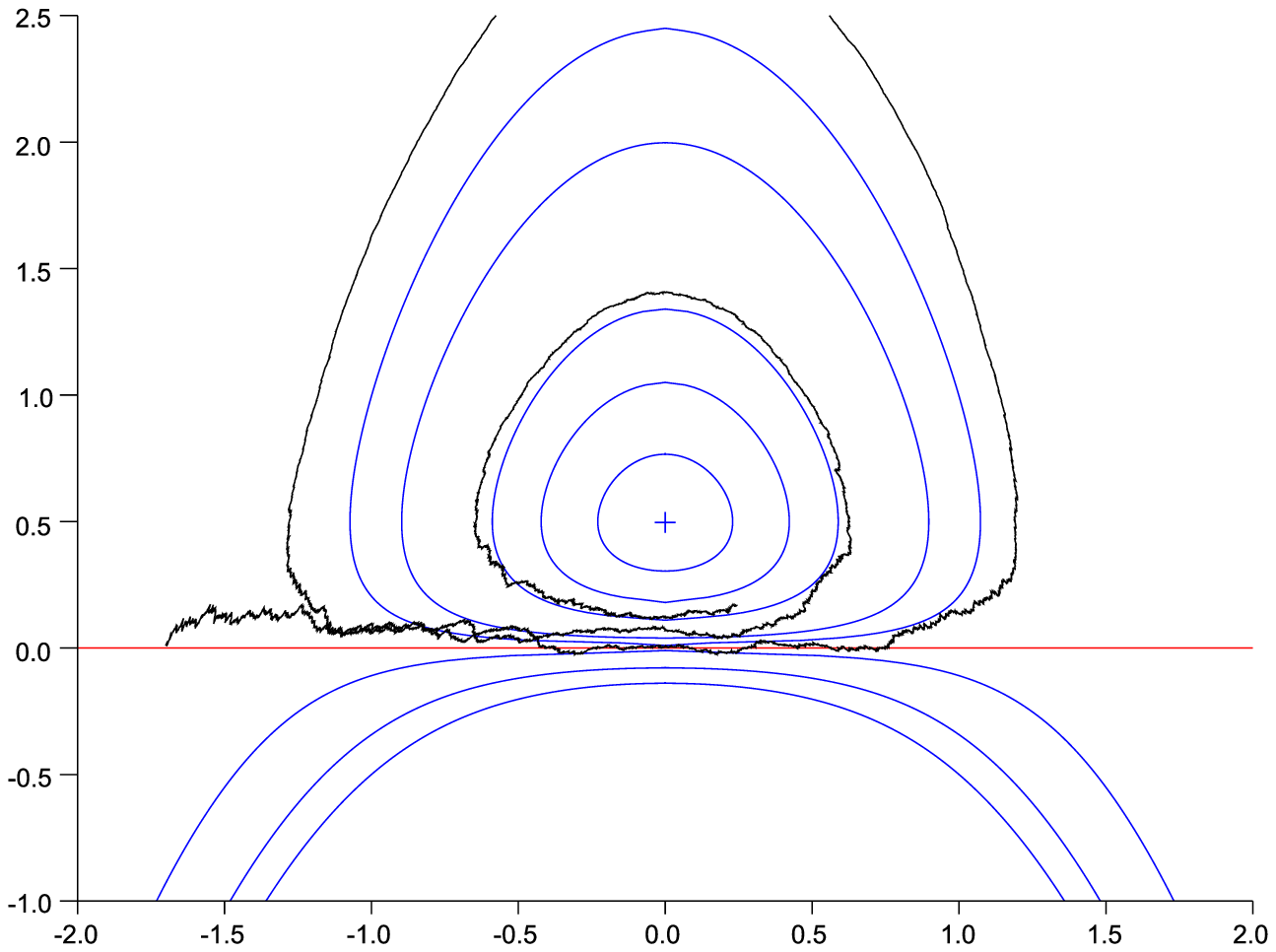}
}
\vspace{3mm}
\centerline{\includegraphics*[clip=true,width=70mm]
{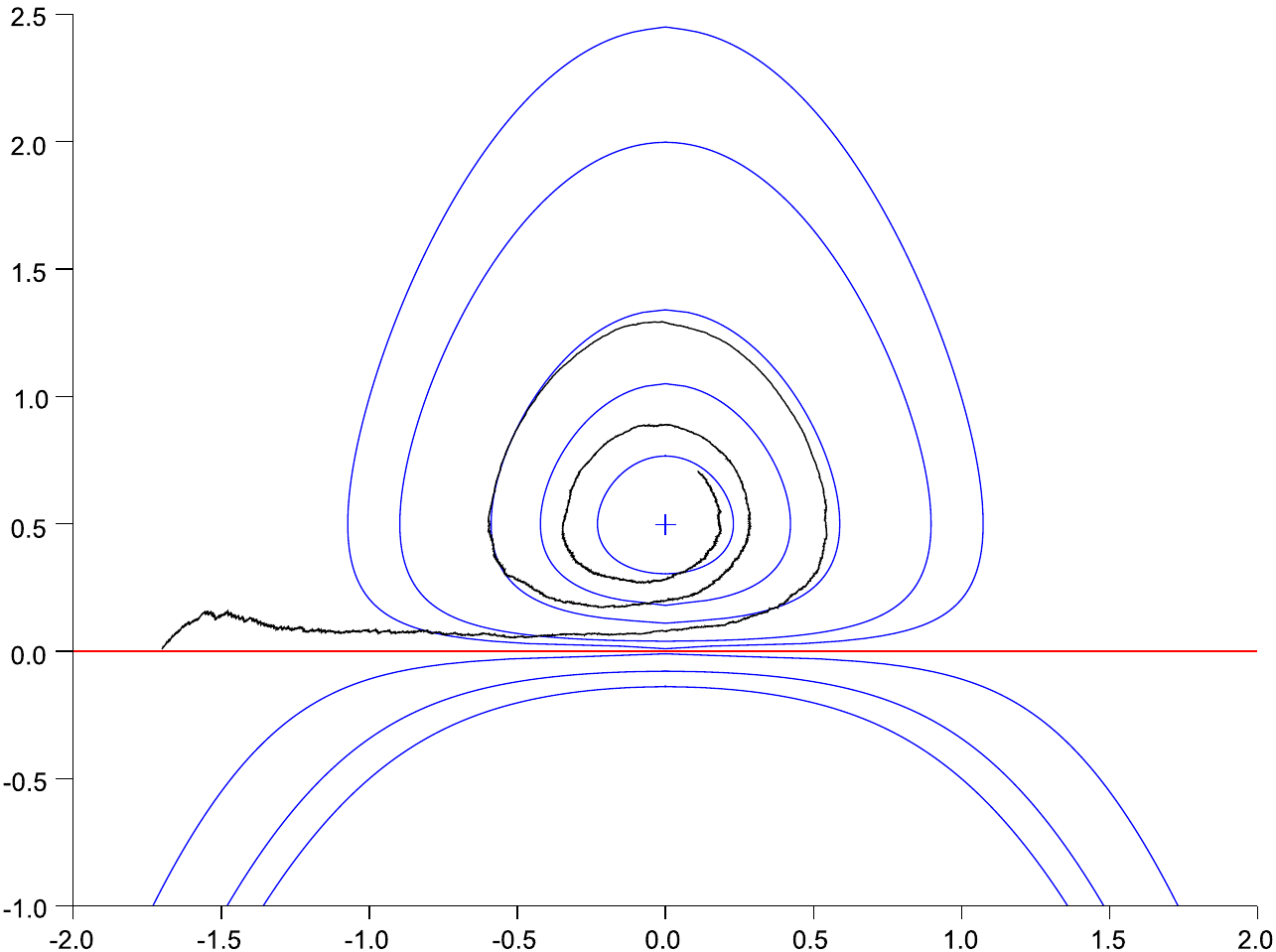}
\hspace{3mm}
\includegraphics*[clip=true,width=70mm]
{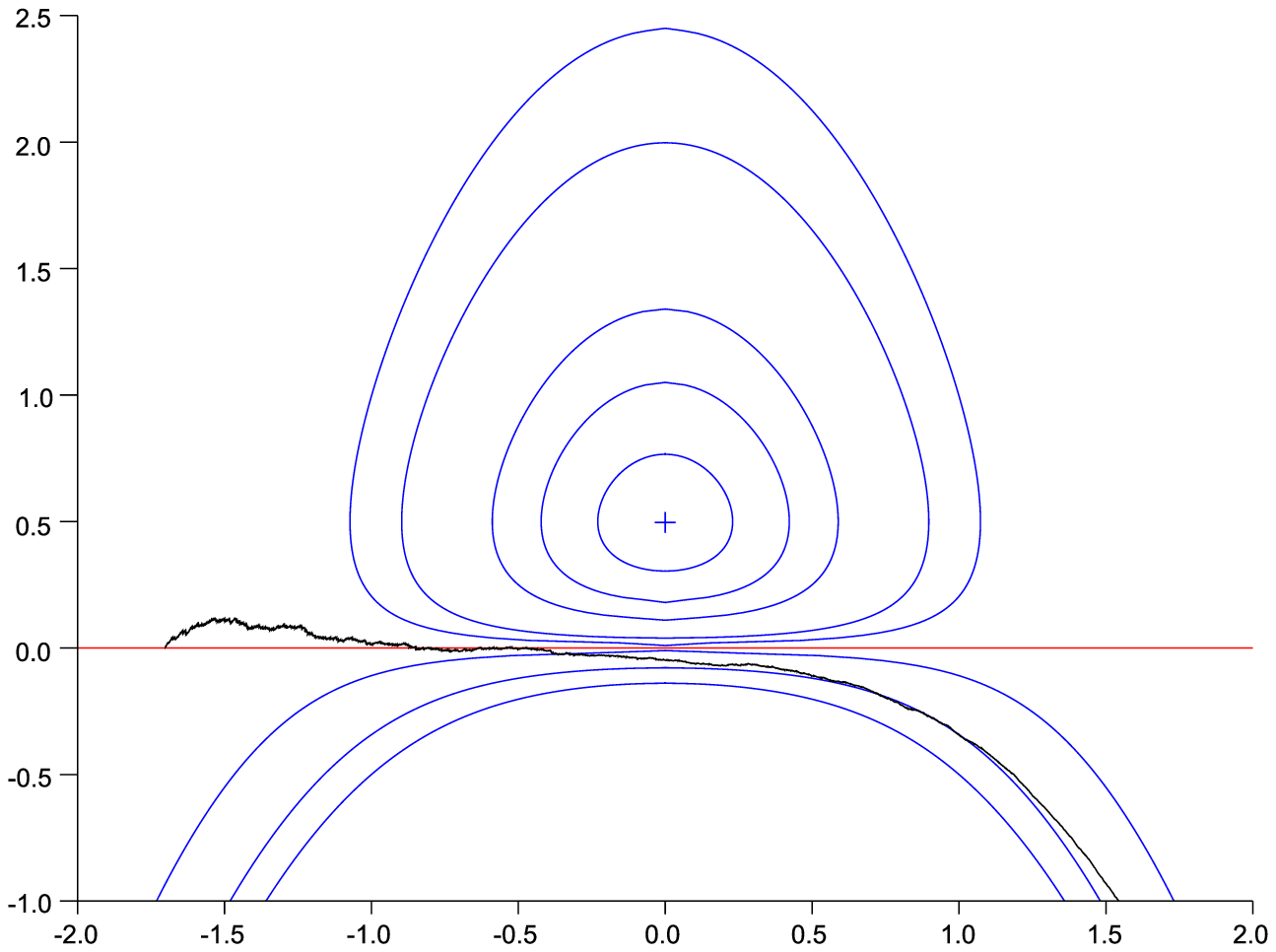}
}
 \figtext{ 
 	\writefig	-0.5	11.0	{\bf (a)}
 	\writefig	7.0	11.0	{\bf (b)}
 	\writefig	-0.5	5.45	{\bf (c)}
 	\writefig	7.0	5.45	{\bf (d)}
 	\writefig	0.6	11.0	$z$
 	\writefig	8.1	11.0	$z$
 	\writefig	0.6	5.45	$z$
 	\writefig	8.1	5.45	$z$
 	\writefig	6.9	6.5	$\xi$
 	\writefig	14.4	6.5	$\xi$
 	\writefig	6.9	0.95	$\xi$
 	\writefig	14.4	0.95	$\xi$
 }
\caption[]{Sample paths of the stochastic equations~\eqref{Rxiz02} (black),
superimposed on the level curves of $\KK$ (blue). The red line is the
separatrix for the case $\eps=\tilde\mu=0$. Parameter values are $c=0$,
$\eps=0.01$ and 
{\bf (a)} $\tilde\mu=0.01$, $\tilde\sigma_1=\tilde\sigma_2 = 0.01$, 
{\bf (b)} $\tilde\mu=0.01$, $\tilde\sigma_1=\tilde\sigma_2 = 0.03$, 
{\bf (c)} $\tilde\mu=0.05$, $\tilde\sigma_1=\tilde\sigma_2 = 0.01$, 
{\bf (d)} $\tilde\mu=-0.05$, $\tilde\sigma_1=\tilde\sigma_2 = 0.013$.
}
\label{fig_path_levelset}
\end{figure}

\figref{fig_path_levelset} shows sample paths of the SDE~\eqref{Rxiz02},
superimposed on the level curves of the first integral $Q$. For sufficiently
weak noise, the sample paths stay close to the level curves. Whether the system
performs a spike or not depends strongly on the dynamics near the separatrix,
which is close to $z=0$ when $\tilde\mu$ and $\eps$ are small. 

To understand better the dynamics close to $z=0$, consider the system 
\begin{equation}
\label{Rxiz02A}
\begin{split}
\6\xi^0_t &= \dfrac{1}{2}\6t  \;,\\
\6z^0_t &= \tilde{ \mu} \6t - 2 \tilde{\sigma}_1  \xi^0_t \6W_t^{(1)} +
\tilde{\sigma}_2  \6W_t^{(2)}\;,
\end{split}
\end{equation}
obtained by neglecting terms of order $z$ and $\sqrt{\eps}$ in~\eqref{Rxiz02}. 
The solution of~\eqref{Rxiz02A} is given by 
\begin{equation}
\label{Rxiz02B}
\begin{split}
\xi^0_t &= \xi^0_0 + \frac12 t  \;,\\
z^0_t &= z^0_0 + \tilde{ \mu} t - 2 \tilde{\sigma}_1  \int_0^t \xi^0_s
\6W_s^{(1)} +
\tilde{\sigma}_2  W_t^{(2)}\;. 
\end{split}
\end{equation}
There is a competition between two terms:  the term $\tilde\mu t$, which pushes
sample paths upwards to the region of SAOs, and the noise terms, whose variance
grows like $(\tilde{\sigma}_1^2+\tilde{\sigma}_2^2)t$. We can thus expect that
if $\tilde\mu>0$ and $\tilde{\sigma}_1^2+\tilde{\sigma}_2^2 \ll \tilde\mu^2$,
then the upwards drift dominates the standard deviation of the noise terms. The
system will be steered to the upper half plane with high probability, and make
many SAOs before ultimately escaping and performing a spike. Going back to
original parameters, the condition $\tilde{\sigma}_1^2+\tilde{\sigma}_2^2 \ll
\tilde\mu^2$ translates into 
\begin{equation}
 \label{weaknoisecondition} 
 \sigma_1^2+\sigma_2^2\ll(\eps^{1/4}\delta)^2\;.
\end{equation} 
If by contrast $\tilde{\sigma}_1^2+\tilde{\sigma}_2^2$ is of the same order as
$\tilde\mu^2$ or larger, or if $\tilde\mu < 0$, then the probability of spiking
will be much larger. 

\begin{figure}
\centerline{\includegraphics*[clip=true,width=80mm]
{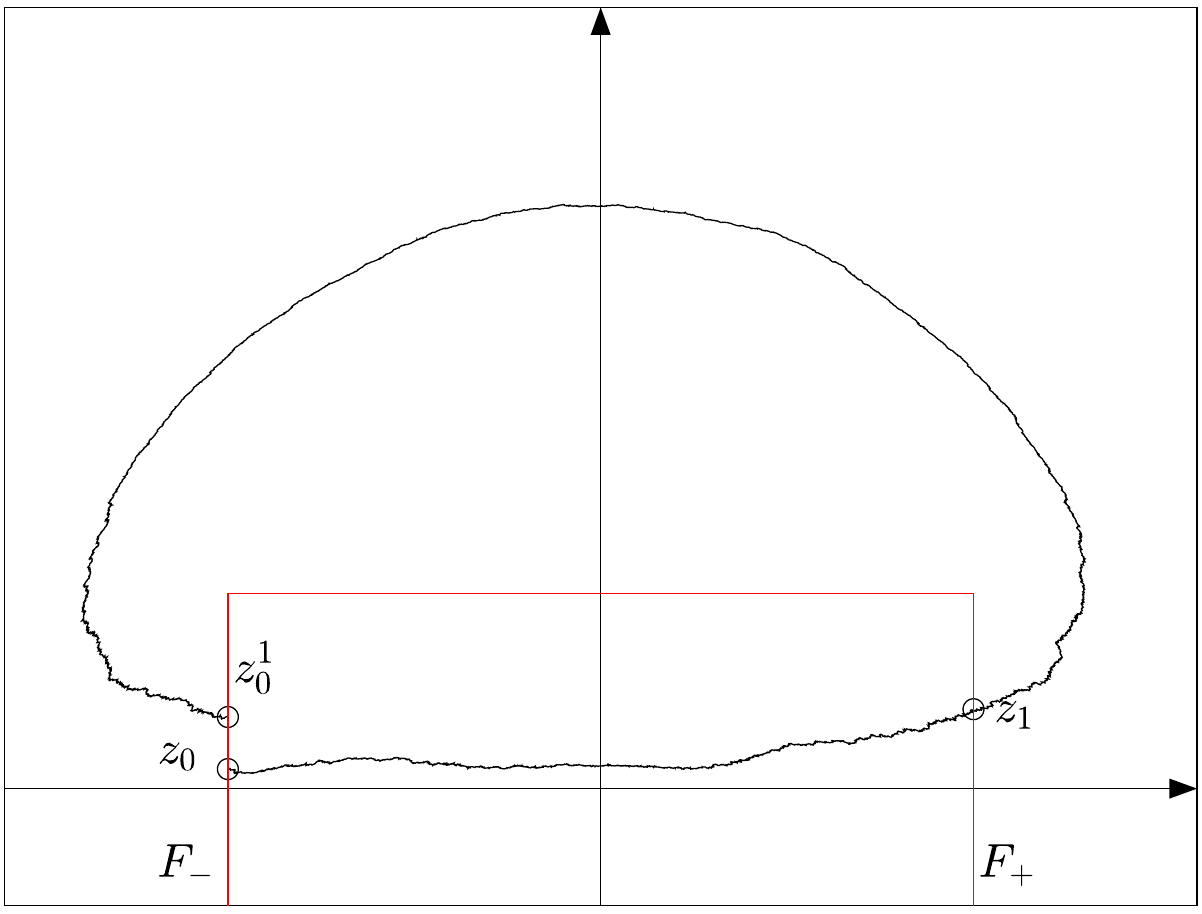}
}
 \figtext{ 
 	\writefig	10.8	0.9	$\xi$
 	\writefig	7.0	6.0	$z$
 	\writefig	7.4	2.7	$1/2$
 	\writefig	9.4	1.0	$L$
 }
\caption[]{Definition of the curves $F_-$ and $F_+$. To approximate the kernel
of the Markov chain, we take an initial condition $z_0$ on $F_-$, and determine
first bounds on the first-hitting point $z_1$ of $F_+$, and then on the
first-hitting point $z_0'$ of $F_-$. 
}
\label{fig_definition_F}
\end{figure}

We now make these ideas rigorous, by proving that spikes are indeed rare in the
regime $\tilde{\sigma}_1^2+\tilde{\sigma}_2^2 \ll \tilde\mu^2$. 
In order to define the Markov chain, we introduce the broken line
\begin{equation}
\label{bounds10}
F_-=\Bigset{ \xi = -L  \text{ and }  z \leqs \frac{1}{2} } \cup \Bigset{-L \leqs
\xi \leqs 0 \text{ and }  z=\frac{1}{2} }\;,
\end{equation}
see \figref{fig_definition_F}. We choose an initial condition $(-L,z_0)$ on
$F_-$, with $z_0$ close to zero. Our aim is to estimate the probability that the
sample path starting in $(-L,z_0)$ returns to $F_-$, after having made one
revolution around $P$, at a point $(-L,z_1)$ with $z_1>z_0$. This will be likely
for weak noise. As a consequence, the set of points in $F_-$ lying above
$(-L,z_0)$ will be mapped into itself by the Markov chain with high probability,
and this fact can be used to estimate the principal eigenvalue $\lambda_0$ and
the expected number of SAOs.

A difficulty of this approach is that the $(\xi,z)$-coordinates are only useful
for small $z$ and bounded $\xi$. For the remaining dynamics, it is in fact much
simpler to work with the first integral $\KK$ (this idea has already been used
in~\cite{MuratovVandeneijnden2007}), and some conjugated angular
variable $\phi$. It would be nice if the $(\KK,\phi)$-coordinates could be used
everywhere, but unfortunately it turns out they are not appropriate when $z$ is
close to $0$ (this is related to the fact that the period of oscillations
diverges near the separatrix). We are thus forced to work with both pairs of
variables, depending on where we are in the phase plane. So we introduce a
second broken line 
\begin{equation}
\label{bounds11}
F_+=\Bigset{ \xi = L  \text{ and }  z \leqs \frac{1}{2} } \cup \Bigset{0 \leqs
\xi \leqs L \text{ and }  z=\frac{1}{2} }\;,
\end{equation}
and use $(\xi,z)$-coordinates to describe the dynamics below $F_-$ and $F_+$,
and $(\KK,\phi)$-coor\-dinates to describe the dynamics above these lines. 

An important point is the choice of the parameter $L$ defining $F_\pm$. On one
hand, it cannot be too large, because we want to treat the terms of order
$\sqrt{\eps}$ in~\eqref{Rxiz02} as small perturbations. On the other hand, the
equations in $(\KK,\phi)$-coordinates contain error terms which
become large in the region of small $z$ and bounded $\xi$. We will work with
$z$ bounded below by a positive constant $c_0$ times $\tilde\mu^{1-\gamma}$ for
some appropriate $\gamma\in(0,1)$. Thus taking 
\begin{equation}
\label{bounds12}
L^2 =\dfrac{\gamma}{2}  \bigpar{-\log (c_- \tilde{\mu})} \;,
\end{equation}
ensures that $\e^{-2L^2} = (c_- \tilde{\mu})^\gamma$, and thus that the
first integral $\KK$ has at least order $\tilde\mu$ when $\abs{\xi}\geqs L$ and
$z \geqs c_0 \tilde\mu^{1-\gamma}$, where we will use $(\KK,\phi)$-coordinates. 

We can now state the main result of this section. 

\begin{theorem}[Weak-noise regime]
\label{thm_weak} 
Assume that $\eps$ and $\delta/\sqrt{\eps}$ are sufficiently small. 
If $c\neq 0$, assume further that $\delta \geqs \abs{c}\eps^\beta$ for some
$\beta<1$. 
Then there exists a constant $\kappa>0$ such that for $\sigma_1^2+\sigma_2^2
\leqs (\eps^{1/4}\delta)^2/\log(\sqrt{\eps}/\delta)$, the principal eigenvalue
$\lambda_0$ satisfies 
\begin{equation}
 \label{weak04}
1 - \lambda_0 \leqs
\exp\biggset{-\kappa\frac{(\eps^{1/4}\delta)^2}{\sigma_1^2+\sigma_2^2}}\;.
\end{equation}  
Furthermore, for any initial distribution $\mu_0$ of incoming sample paths, the
expected number of SAOs satisfies 
\begin{equation}
 \label{weak05}
\bigexpecin{\mu_0}{N} \geqs
C(\mu_0)\exp\biggset{\kappa\frac{(\eps^{1/4}\delta)^2}{\sigma_1^2+\sigma_2^2}}
\;.  
\end{equation} 
Here $C(\mu_0)$ is the probability that the incoming path hits $F_-$ above the 
separatrix. 
\end{theorem}
\begin{proof}
Let us first show that the problem can be reduced to proving the existence of a
subset $A\subset E$ with positive Lebesgue measure that the Markov chain is
unlikely to leave. This set will then be chosen as the set of points in $F_-$
for which $z$ is larger than some $z_0$ of order $\tilde\mu^{1-\gamma}$. 
Let 
\begin{equation}
 \label{bounds01} 
\eps_A = \sup_{x\in A} \bigbrak{1 - K(x,A)}
\end{equation} 
be the maximal probability to leave $A$ when starting in $A$. 
Let us show that  
\begin{equation}
 \label{bounds02}
\lambda_0 \geqs 1-\eps_A\;.
\end{equation}  
Indeed, the relation $\lambda_0\pi_0=\pi_0 K$ yields 
\begin{align}
\nonumber
\lambda_0 \pi_0(A) &= \int_A \pi_0(\6x)K(x,A) + \int_{E\setminus A}
\pi_0(\6x)K(x,A) \\
&\geqs \pi_0(A) (1-\eps_A) + \int_{E\setminus A} \pi_0(\6x)s(x) \, \nu(A)\;.
 \label{bounds04}
\end{align} 
Either $\pi_0(A)=1$, and the result follows immediately. Or $\pi_0(A)<1$, and
thus $\pi_0(E\setminus A)>0$, so that the second term on the right-hand side is
strictly positive. It follows that $\lambda_0\pi_0(A)>0$, and we
obtain~\eqref{bounds02} upon dividing by $\pi_0(A)$.

Next, let us prove that 
\begin{equation}
 \label{bounds05}
\expecin{\mu_0}{N} \geqs \frac{\mu_0(A)}{\eps_A}\;. 
\end{equation} 
For $x\in A$, let $\theta(x) = \expecin{x}{N} = \sum_{n\geqs0}K^n(x,E)$. Then
$\theta(x)=\lim_{n\to\infty}\theta_n(x)$ where 
\begin{equation}
 \label{bounds06}
\theta_n(x) = \sum_{m=0}^n K^m(x,E)\;.  
\end{equation} 
We have 
\begin{equation}
 \label{bounds07}
\theta_{n+1}(x) = 1 + (K\theta_n)(x) \geqs 1 + \int_A K(x,\6y)\theta_n(y)\;. 
\end{equation} 
Now let $m_n=\inf_{x\in A}\theta_n(x)$. Then $m_0=1$ and 
\begin{equation}
 \label{bounds08}
m_{n+1} \geqs 1 + (1-\eps_A)m_n\;.
\end{equation} 
By induction on $n$ we get 
\begin{equation}
 \label{bounds09}
m_n \geqs \frac{1}{\eps_A} - \frac{(1-\eps_A)^{n+1}}{\eps_A}\;,
\end{equation} 
so that $\expecin{x}{N} = \theta(x) \geqs 1/\eps_A$ for all $x\in A$,
and~\eqref{bounds05} follows upon integrating against $\mu_0$ over $A$.

It thus remains to construct a set $A\subset E$ such that $\eps_A$ is
exponentially small in $\tilde\mu^2/(\tilde\sigma_1^2+\tilde\sigma_2^2)$. The
detailed computations being rather involved, we give them in the appendix, and
only summarise the main steps here. 
\begin{enum}
\item	In the first step, we take an initial condition $(-L,z_0)$ on $F_-$,
with $z_0$ of order $\tilde\mu^{1-\gamma}$ (\figref{fig_definition_F}). It is
easy to show that the
deterministic solution starting in $(-L,z_0)$ hits $F_+$ for the first time at a
point $z^0_T\geqs c_0\tilde\mu^{1-\gamma}$ where $c_0>0$. Consider now the 
stochastic sample path starting in
$(-L,z_0)$. Proposition~\ref{prop_xiz_kernel} in Appendix~\ref{sec_separatrix}
shows that there are constants $C, \kappa_1>0$ such that the sample path hits
$F_+$ for the first time at a point $(L,z_1)$ satisfying 
\begin{equation}
 \label{bounds13}
\bigprob{z_1 < c_0\tilde\mu^{1-\gamma} - \tilde\mu} \leqs 
\frac{C}{\tilde\mu^{2\gamma}} \e^{-\kappa_1\tilde\mu^2/\tilde\sigma^2}\;. 
\end{equation} 
This is done by first approximating~\eqref{Rxiz02} by a linear system, and then
showing that the effect of nonlinear terms is small. 
\item	In the second step, we show that a sample path starting in $(L,z_1)\in
F_+$ returns with high probability to $F_-$ at a point $(-L,z_0')$ with
$z_0'\geqs z_1$. 
Using an averaging procedure for the variables $(\KK,\phi)$, we show that
$\KK$ varies little between $F_+$
and $F_-$. Corollary~\ref{cor_Kphi} in Appendix~\ref{sec_SAO} shows that there
is a $\kappa_2>0$ such that 
\begin{equation}
 \label{bounds15}
\bigpcond{z_0' < z_1}
{z_1 \geqs c_0\tilde\mu^{1-\gamma} - \tilde\mu} \leqs
2\e^{-\kappa_2\tilde\mu^2/\tilde\sigma^2}\;. 
\end{equation}
\end{enum}
In the above results, we assume that either $c=0$, or $c\neq0$ and 
$\tilde\mu^{1+\theta}\geqs\sqrt{\eps}$ for some $\theta>0$ (which follows from
the assumption $\delta > \eps^\beta$). The reason is that
if $c=0$, we can draw on the fact that the error terms of order $\sqrt{\eps}$
in~\eqref{Rxiz02} are positive, while if $c\neq0$ we only know their order. 
Choosing $A$ as the set of points in $F_-$ for which
$z\geqs c_0\tilde\mu^{1-\gamma} - \tilde\mu$, we obtain that $\eps_A$ is bounded
by the sum of~\eqref{bounds13} and~\eqref{bounds15}, and the results follow by
returning to original parameters. 
\end{proof}

Relation~\eqref{weak05} shows that the average number of SAOs between two
consecutive spikes is exponentially large in this regime. Note that each SAO
requires a rescaled time of order $1$ (see Section~\ref{ssec_Kphi}), and thus a
time of order $\sqrt{\eps}$ in original units. It follows that the average
interspike interval length is obtained by multiplying~\eqref{weak05} by a
constant times~$\sqrt{\eps}$. 

Relation~\eqref{res06} shows that the distribution of $N$ is asymptotically
geometric with parameter given by~\eqref{weak04}. Hence the interspike
interval distribution will be close to an exponential one, but with a
periodic modulation due to the SAOs.  


\section{The transition from weak to strong noise}
\label{sec_trans}

We now give an approximate description of how the dynamics changes with
increasing noise intensity. Assume that we start~\eqref{Rxiz02} with an initial
condition $(\xi_0,z_0)$ where $\xi_0 = -L$ for some $L>0$ and 
$z_0$ is small. As long as $z_t$ remains small, we may approximate 
$\xi_t$ in the mean by $\xi_0+t/2$, and thus $z_t$ will be close to the
solution of
\begin{equation}
 \label{trans01} 
\6z^1_t = \Bigpar{ \tilde{ \mu}  + t  z^1_t}  \6t - \tilde{\sigma}_1  t
\6W_t^{(1)}+ \tilde{\sigma}_2 \6W_t^{(2)}\;.
\end{equation}
This linear equation can be solved explicitly. In
particular, at time $T=4L$, $\xi_t$ is close to $L$ and we have the following
result.

\begin{prop}
\label{prop_trans}
Let $2L^2 = \gamma\abs{\log(c_-\tilde\mu)}$ for some $\gamma, c_->0$. Then for
any $H$,  
\begin{equation}
 \label{trans06}
 \bigprob{z^1_T \leqs -H} = \Phi\left(-\pi^{1/4}
\frac{\tilde\mu}{\tilde\sigma} \biggbrak{1 +
\BigOrder{(H+z_0)\tilde\mu^{\gamma-1}}}\right)\;, 
\end{equation}
where $\tilde\sigma^2=\tilde\sigma_1^2 + \tilde\sigma_2^2$ and 
$\Phi(x)=\int_{-\infty}^x \e^{-u^2/2}\6u/\sqrt{2\pi}$ is the distribution
function of the standard normal law. 
\end{prop}
\begin{proof}
Solving~\eqref{trans01} by variation of the
constant yields 
\begin{equation}
 \label{trans02}
 z^1_T = z_0 + \e^{T^2/2} \biggbrak{\tilde\mu \int_{t_0}^T \e^{-s^2/2}\6s 
- \tilde\sigma_1 \int_{t_0}^T s \e^{-s^2/2} \6W^{(1)}_s 
+ \tilde\sigma_2 \int_{t_0}^T \e^{-s^2/2} \6W^{(2)}_s} \;. 
\end{equation} 
Note that by the choice of $L$, we have 
$\e^{T^2/2}=\e^{2L^2}=(c_-\tilde\mu)^{-\gamma}$. 
The random variable $z^1_T$ is Gaussian, with expectation 
\begin{equation}
 \label{trans03}
\bigexpec{z^1_T} = z_0 + \tilde\mu \e^{2L^2} \int_{-2L}^{2L} \e^{-s^2/2}\6s
\end{equation} 
and variance 
\begin{equation}
\variance(z^1_T) = \tilde\sigma_1^2 \e^{4L^2}\int_{-2L}^{2L} s^2\e^{-s^2} \6s +
\tilde\sigma_2^2 \e^{4L^2} \int_{-2L}^{2L} \e^{-s^2} \6s\;.
 \label{trans04}
\end{equation} 
Using this in the relation 
\begin{equation}
 \label{trans05}
\bigprob{z^1_T \leqs -H} =
\int_{-\infty}^{-H}
\frac{\e^{-(z-\expec{z^1_T})^2/2\variance(z^1_T)}}{\sqrt{2\pi\variance(z^1_T)}}
\6z = 
\Phi\left(-\frac{H+\expec{z^1_T}}{\sqrt{\variance(z^1_T)}}\right)
\end{equation} 
yields the result. 
\end{proof}

Choosing $\gamma$ large enough, the right-hand side of~\eqref{trans06} is
approximately constant for a large range of values of $z_0$ and $H$. The
probability that the system performs no complete SAO before spiking again
should thus behave as 
\begin{equation}
 \label{trans07}
\bigprobin{\mu_0}{N=1} \simeq  \Phi\left(-\pi^{1/4}
\frac{\tilde\mu}{\tilde\sigma}\right) = 
\Phi\left(-
\frac{(\pi\eps)^{1/4}(\delta-3\alpha_*\sigma_1^2/\eps)}{\sqrt{
\sigma_1^2+\sigma_2^2}}
\right)\;.
\end{equation} 
Since $1-\lambda_0$ is equal to the probability of leaving $\cD$ before
completing the first SAO, when starting in the QSD $\pi_0$, we expect that
$1-\lambda_0$ has a similar behaviour, provided $\pi_0$ is concentrated near
$z=0$. We can identify three regimes, depending on the value of
$\tilde\mu/\tilde\sigma$~:
\begin{enum}
\item	{\bf Weak noise~:} $\tilde\mu \gg \tilde\sigma$, which in original
variables translates into $\sqrt{\sigma_1^2+\sigma_2^2} \ll \eps^{1/4}\delta$.
This is the weak-noise regime already studied in the previous section, in
which $\lambda_0$ is exponentially close to $1$, and thus spikes are separated
by long sequences of SAOs.

\item	{\bf Strong noise~:} $\tilde\mu \ll -\tilde\sigma$, which implies $\mu
\ll \tilde\sigma^2$, and in original variables translates into
$\sqrt{\sigma_1^2+\sigma_2^2} \gg \eps^{3/4}$. Then $\prob{N>1}$ is
exponentially small, of order $\e^{-(\sigma_1^2+\sigma_2^2)/\eps^{3/2}}$. Thus
with high probability, there will be no complete SAO between consecutive spikes,
i.e., the neuron is spiking repeatedly. 

\item	{\bf Intermediate noise~:} $\abs{\tilde\mu} = \Order{\tilde\sigma}$,
which translates into $\eps^{1/4}\delta \leqs \sqrt{\sigma_1^2+\sigma_2^2} \leqs
\eps^{3/4}$. Then the mean number of SAOs is of order $1$. In particular, when
$\sigma_1=\sqrt{\eps\delta}$, $\tilde\mu=0$ and thus $\prob{N=1}$ is close
to $1/2$. 
\end{enum}
 
\begin{figure}
\centerline{\hspace{2mm}\includegraphics*[clip=true,width=70mm]
{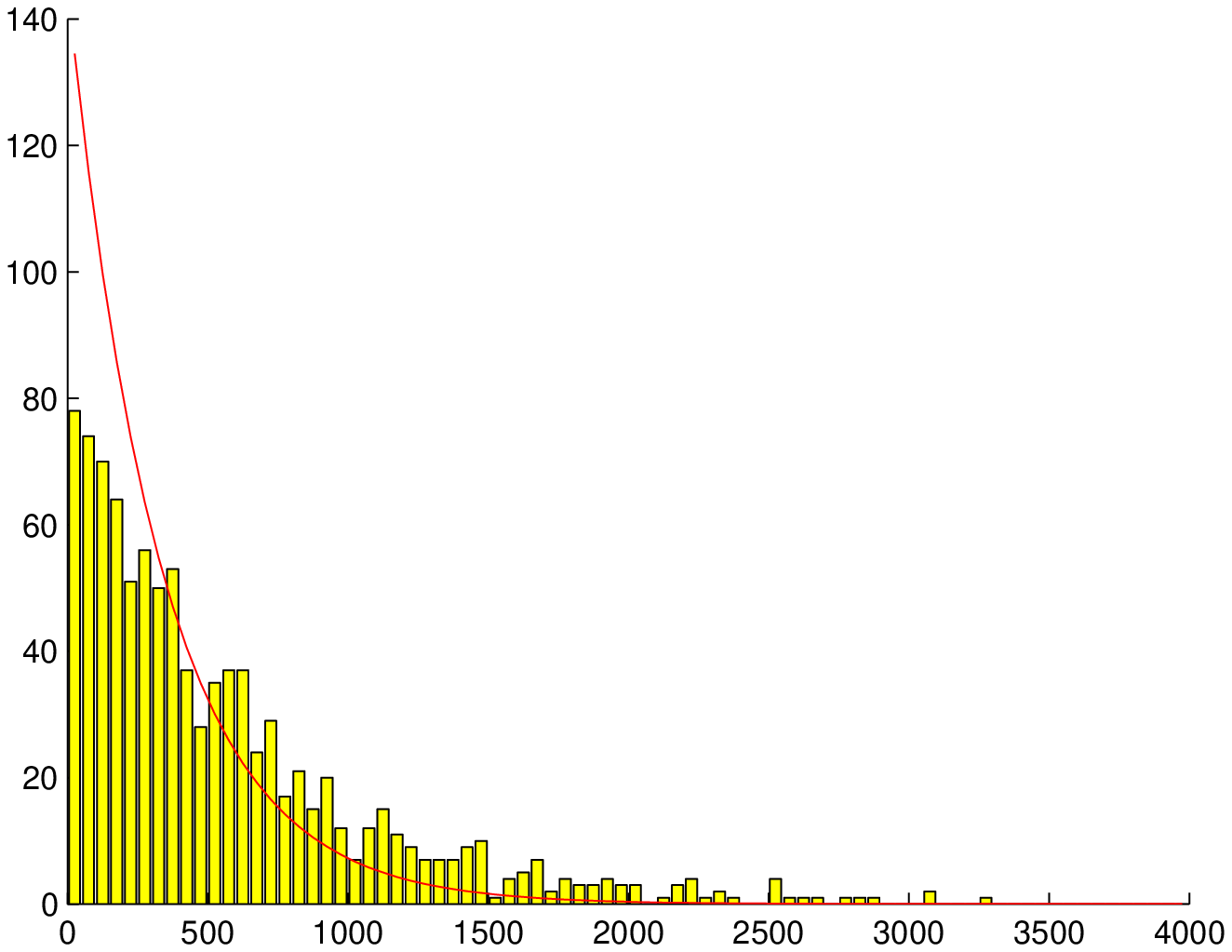}
\hspace{2mm}
\includegraphics*[clip=true,width=70mm]{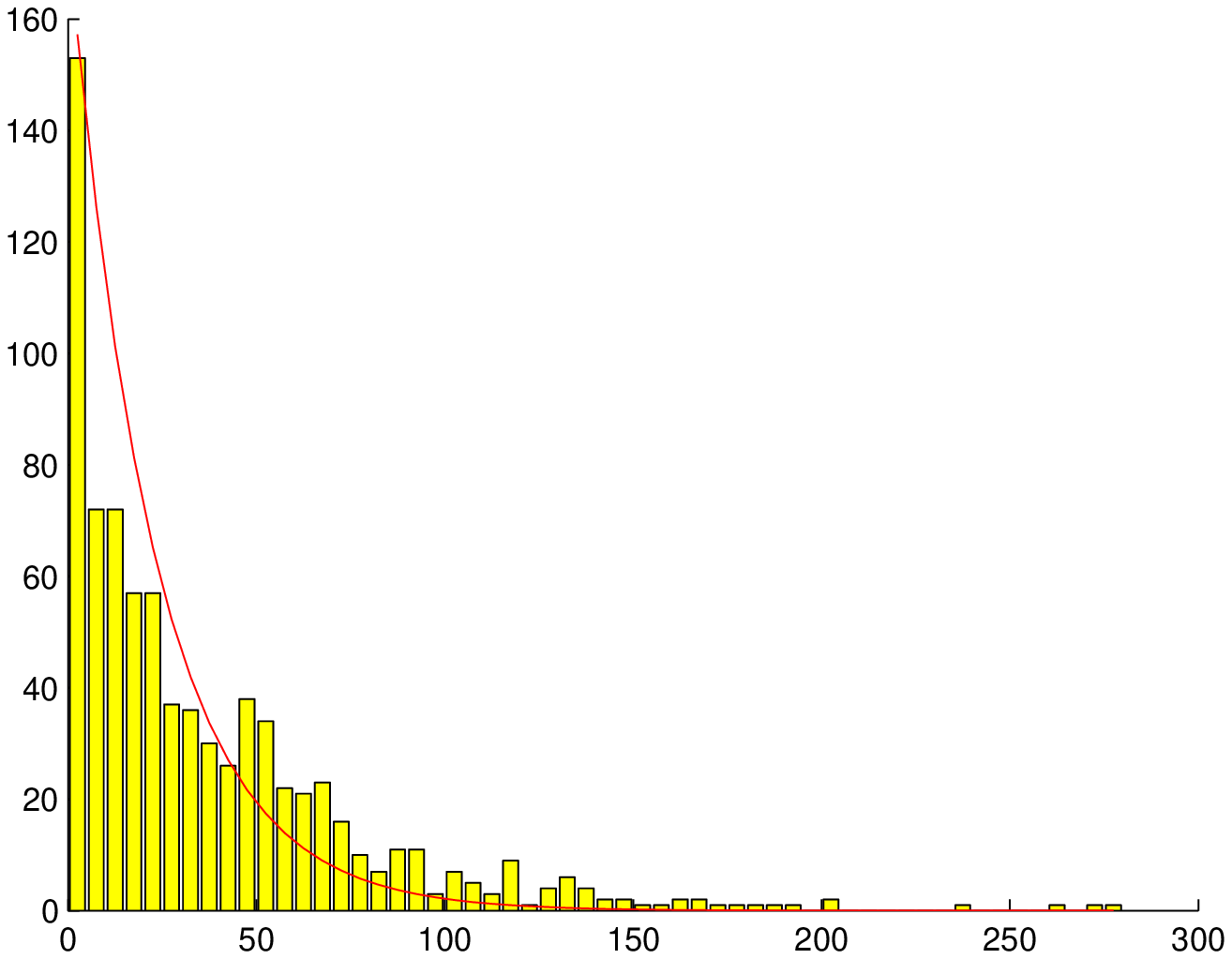}
}
\vspace{2mm}
\centerline{\hspace{2mm}\includegraphics*[clip=true,width=70mm]
{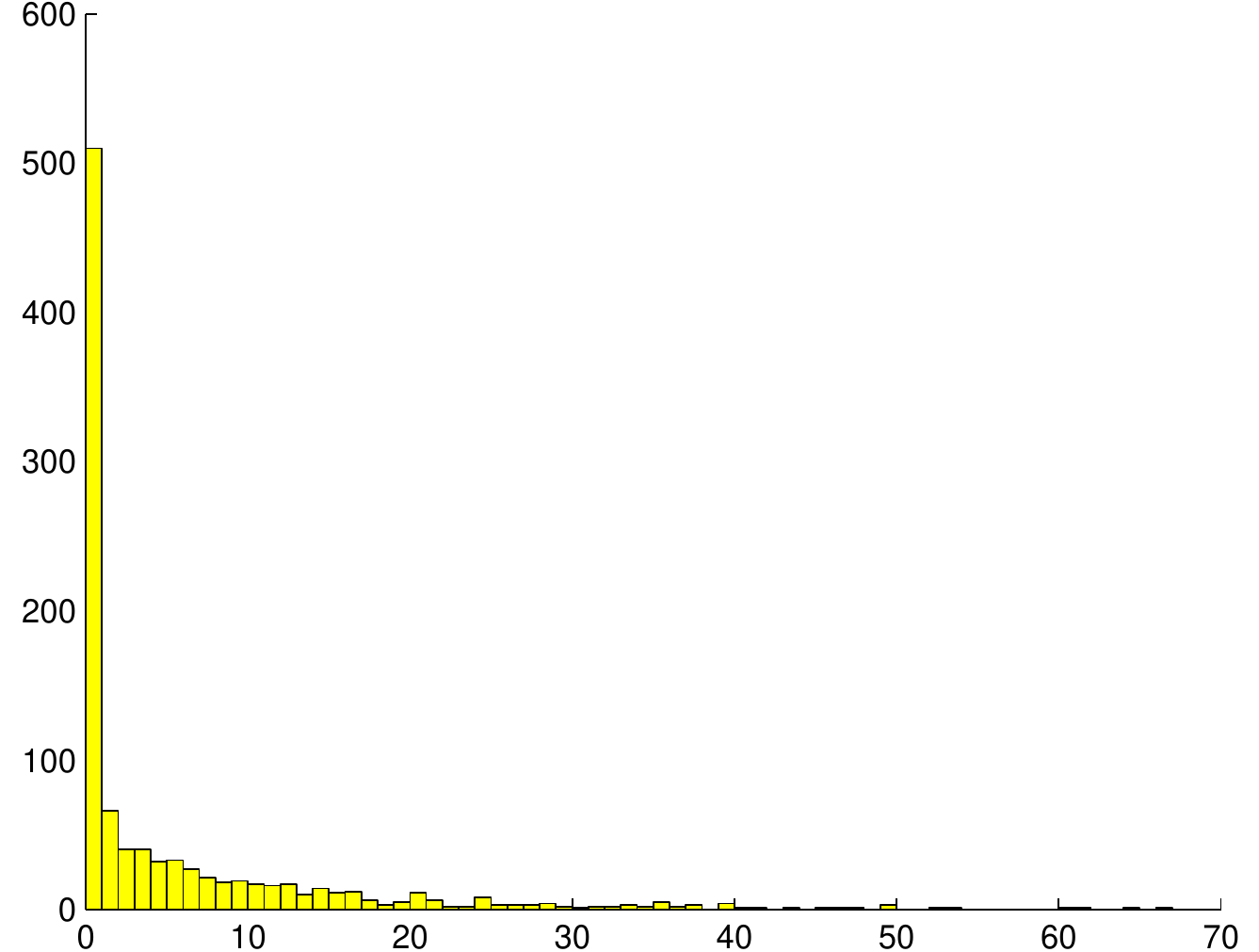}
\hspace{2mm}
\includegraphics*[clip=true,width=70mm]
{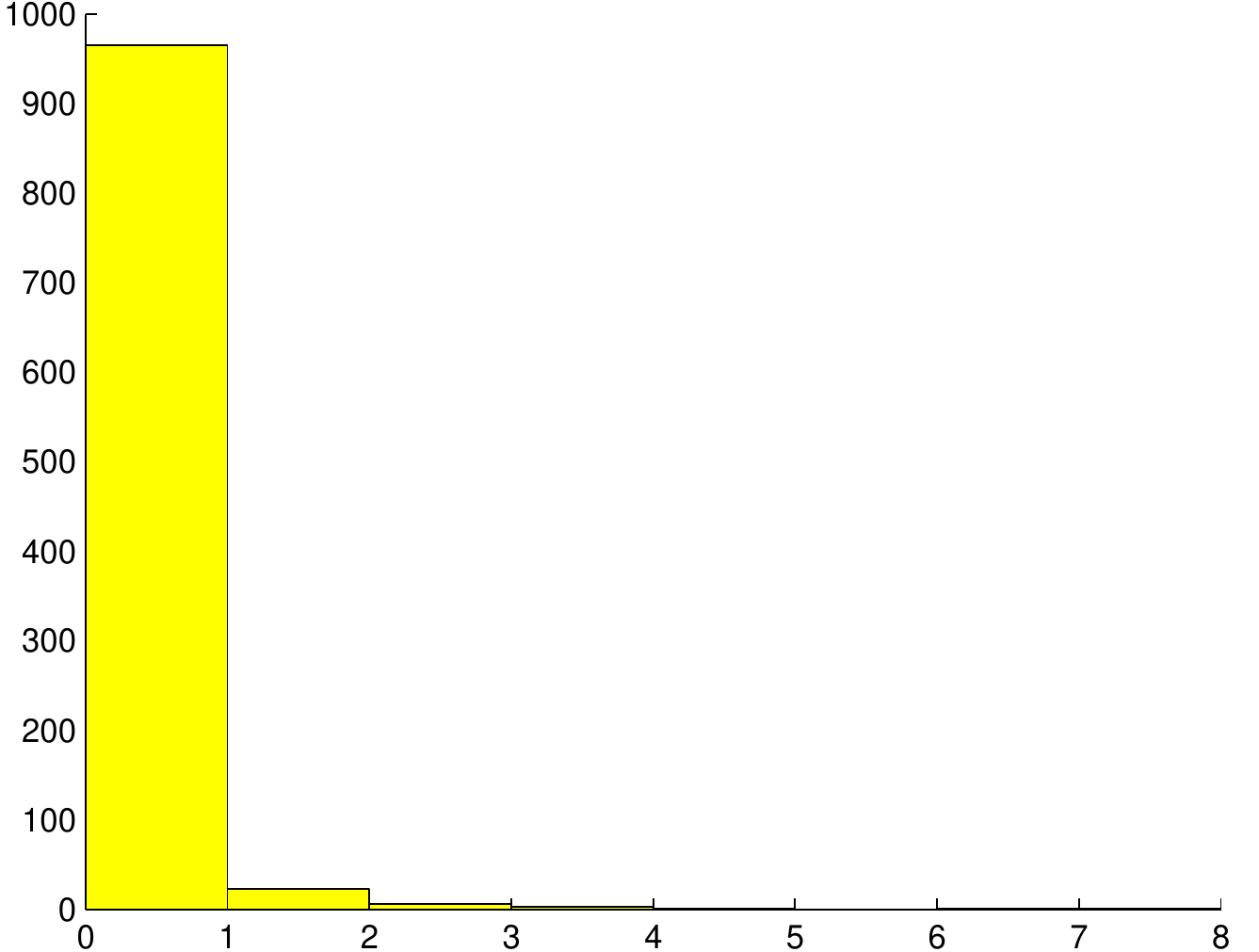}
}
 \figtext{ 
 	\writefig	7.25	6.05	$N$
 	\writefig	-0.25	10.95	{\bf (a)}
 	\writefig	14.7	6.05	$N$
 	\writefig	7.15	10.95	{\bf (b)}
 	\writefig	7.25	0.6	$N$
 	\writefig	-0.25	5.4	{\bf (c)}
 	\writefig	14.7	0.6	$N$
 	\writefig	7.15	5.4	{\bf (d)}
 }
 \vspace{2mm}
\caption[]{Histograms of numerically simulated distributions of the SAO number
$N$, obtained from time series containing $1000$ spikes each. The
superimposed curves show
geometric distributions with parameter $\lambda_0$, where $\lambda_0$ has been
estimated from the expectation of $r^N$, as explained in the text. Parameter
values are $\eps=10^{-4}$ and $\tilde\sigma=0.1$ in all cases, and {\bf (a)}
$\tilde\mu=0.12$, {\bf (b)} $\tilde\mu=0.05$, {\bf (c)} $\tilde\mu=0.01$, and
{\bf (d)} $\tilde\mu=-0.09$ (cf.~\eqref{Rxiz19} for their definition).
}
\label{fig_histograms}
\end{figure}

An interesting point is that the transition from weak to strong noise is
gradual, being characterised by a smooth change of the distribution of $N$ as a
function of the parameters. There is no clear-cut transition at the parameter
value $\sigma_1=\sqrt{\eps\delta}$ obtained in~\cite{MuratovVandeneijnden2007}
(cf.~\figref{fig_bif_diagram}), the only particularity of this parameter value
being that $\prob{N=1}$ is close to $1/2$. In other words, the system decides
between spiking and performing an additional SAO according to the result of a
fair coin flip. The definition of a boundary between the intermediate and
strong-noise regimes mainly depends on how well the SAOs can be resolved in
time. A very good time resolution would put the boundary at noise intensities of
order $\eps^{3/4}$, while a lower time resolution would move in closer to
$\sqrt{\eps\delta}$.


\section{Numerical simulations}
\label{sec_sim}

\figref{fig_histograms} shows numerically simulated distributions of the SAO
number. The geometric decay is clearly visible. In addition, for
decreasing values of $\tilde\mu/\tilde\sigma$, there is an increasing bias
towards the first peak $N=1$, which with our convention corresponds to the
system performing no complete SAO between consecutive spikes. Of course
this does not contradict the asymptotic result~\eqref{res06}, but it shows
that transient effects are important.

\begin{figure}
\centerline{
\includegraphics*[clip=true,width=47mm]
{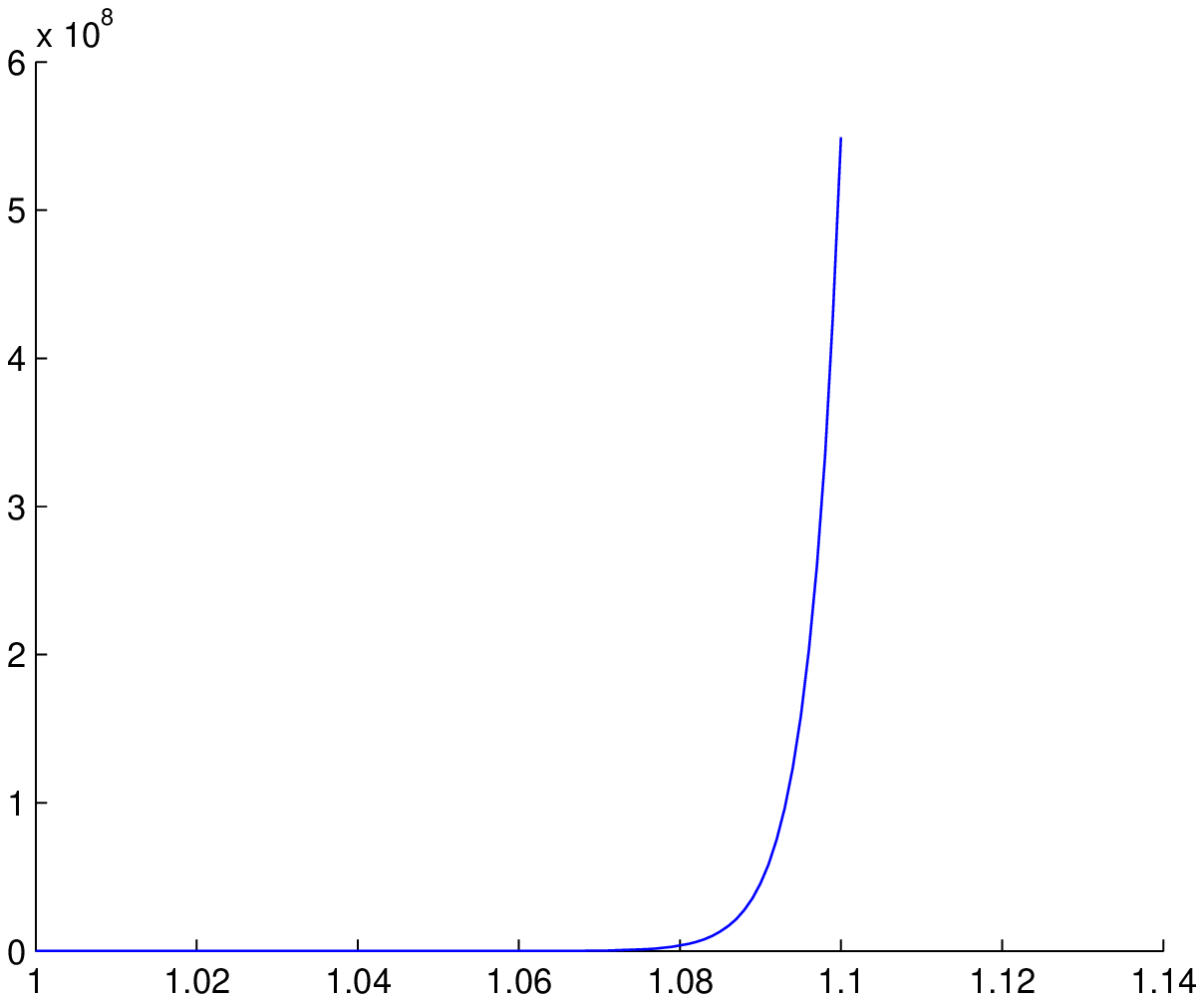}
\hspace{2mm}
\includegraphics*[clip=true,width=47mm]
{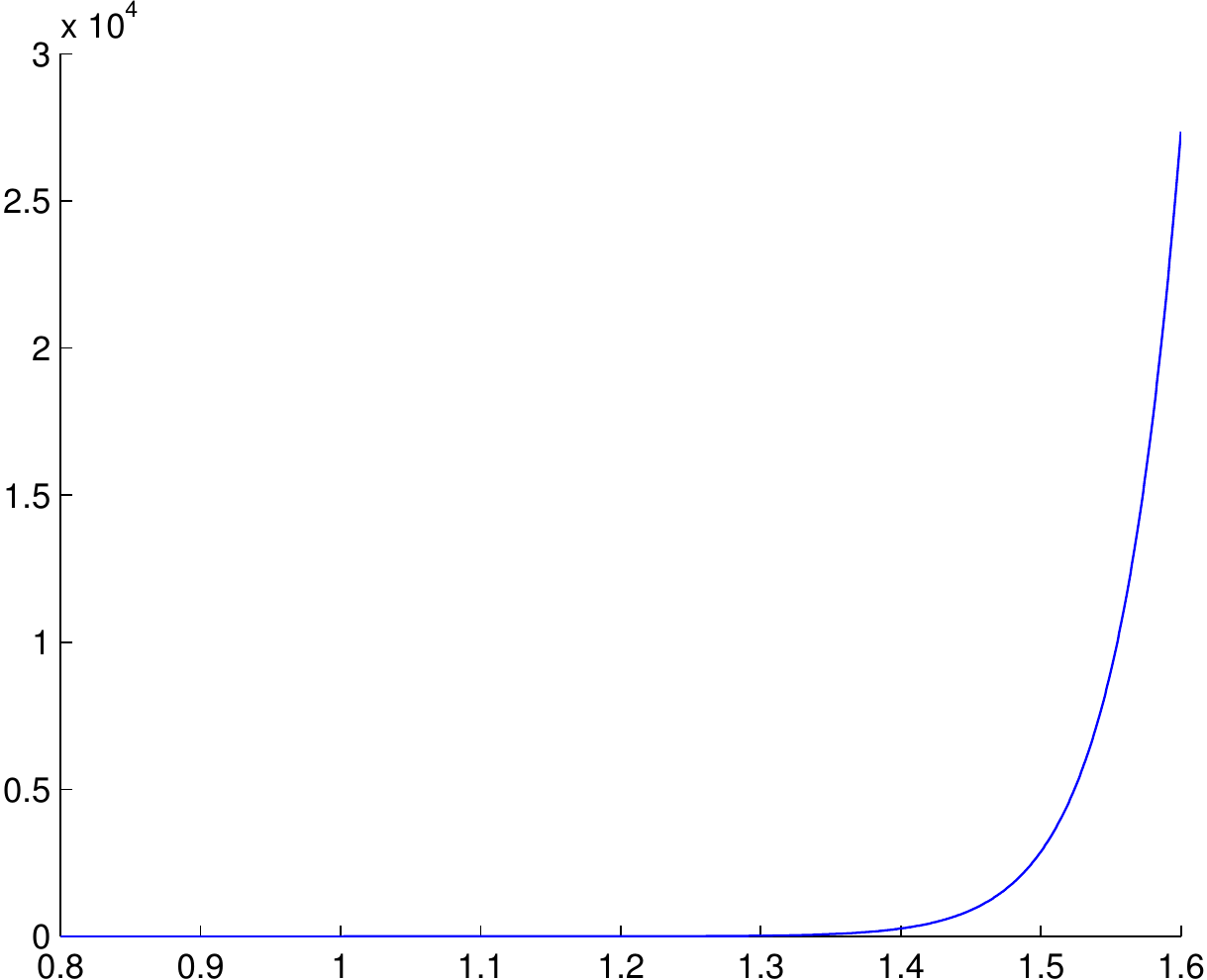}
\hspace{2mm}
\includegraphics*[clip=true,width=47mm]
{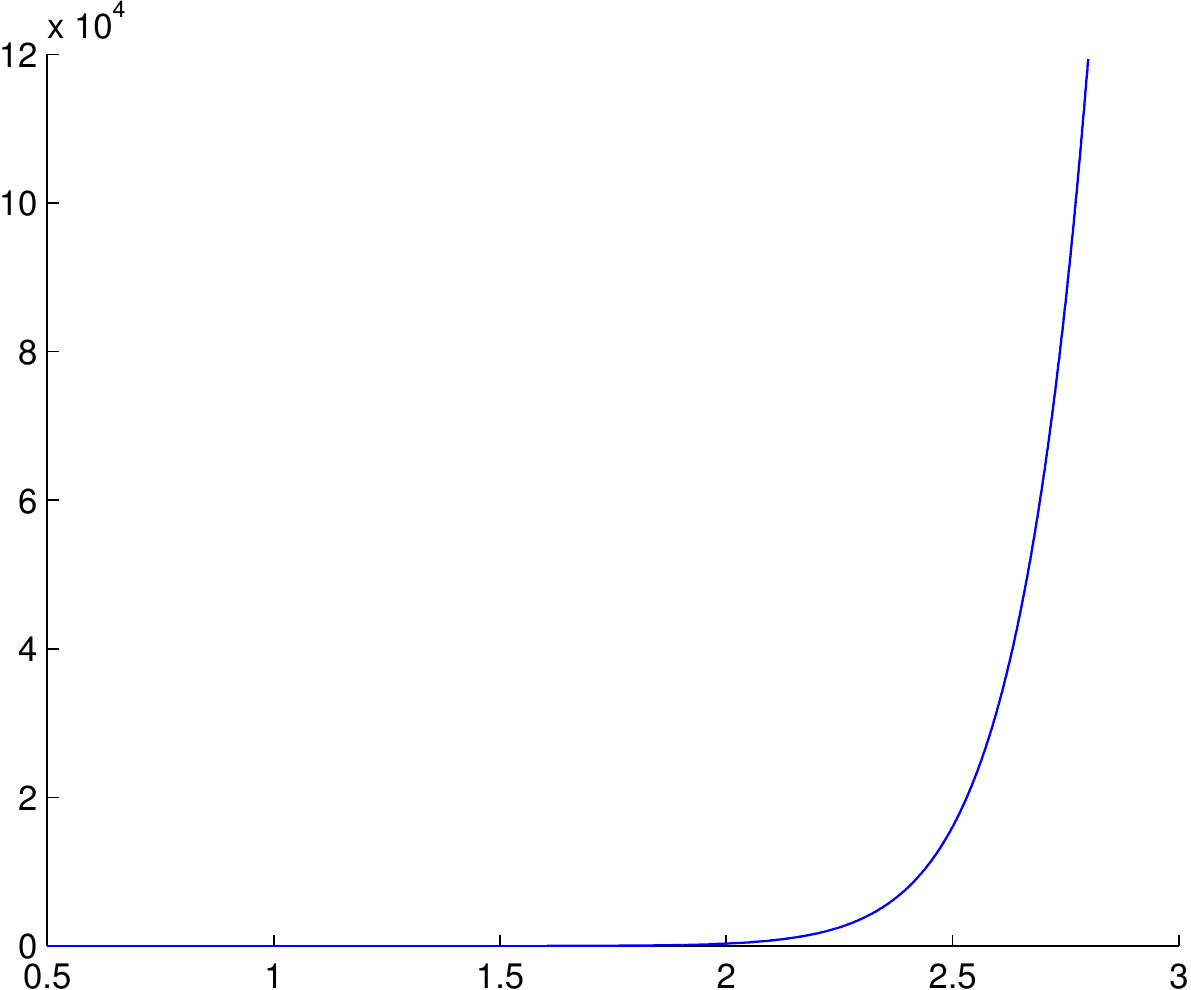}
}
 \figtext{ 
 	\writefig	-0.7	4.0	{\bf (a)}
 	\writefig	4.4	4.0	{\bf (b)}
 	\writefig	9.6	4.0	{\bf (c)}
 	\writefig	0.4	4.0	$\E(r^N)$
 	\writefig	5.5	4.0	$\E(r^N)$
 	\writefig	10.7	4.0	$\E(r^N)$
 	\writefig	4.2	0.8	$r$
 	\writefig	9.4	0.8	$r$
 	\writefig	14.6	0.8	$r$
 }
 \vspace{2mm}
\caption[]{Empirical expectation of $r^N$ as a function of $r$, for
$\tilde\sigma=0.1$, $\eps=10^{-4}$ and $\tilde\mu = 0.05$ (a), $\tilde\mu =
-0.03$ (b) and $\tilde\mu = -0.06$ (c). The respective values of
$-\tilde\mu/\tilde\sigma$ are thus $-0.5$ (a), $0.3$ (b) and $0.6$ (c). 
The location of the pole allows to estimate $1/\lambda_0$.
}
\label{fig_comparison1}
\end{figure}

Due to the finite sample size, the number of events in the tails of the
histograms is too small to allow for a chi-squared adequacy test. We can,
however, estimate the principal eigenvalue $\lambda_0$, by using the fact that
the moment generating function $\expecin{\mu_0}{r^N}$ has a simple pole at
$r=1/\lambda_0$ (see \eqref{markov103} and~\eqref{markov107}).
\figref{fig_comparison1} shows examples of the dependence of the empirical
expectation of $r^N$ on $r$. By detecting when its derivative exceeds a given
threshold, one obtains an estimate of $1/\lambda_0$. Geometric distributions
with parameter $\lambda_0$ have been superimposed on two histograms
in~\figref{fig_histograms}. 

\figref{fig_comparison2} shows, as a function of $-x=-\tilde\mu/\tilde\sigma$,
the curve $x\mapsto \Phi(-\pi^{1/4}x)$, as well as the inverse of the empirical
expectation of $N$, the probability that $N=1$, and $1-\lambda_0$ where the
principal eigenvalue $\lambda_0$ has been estimated from the generating
function. The data points for $\tilde\mu>0$ have been obtained from histograms
containing $1000$ spikes, while those for $\tilde\mu<0$ have been obtained from
histograms containing $500$ spikes separated by $N>1$ SAOs (the number of
spiking events with $N=1$ being much larger). Theorem~\ref{thm_weak} 
applies to the far left of the figure, when $\tilde\mu\gg\tilde\sigma$.

\begin{figure}
\centerline{\hspace{7mm}
\includegraphics*[clip=true,width=90mm]{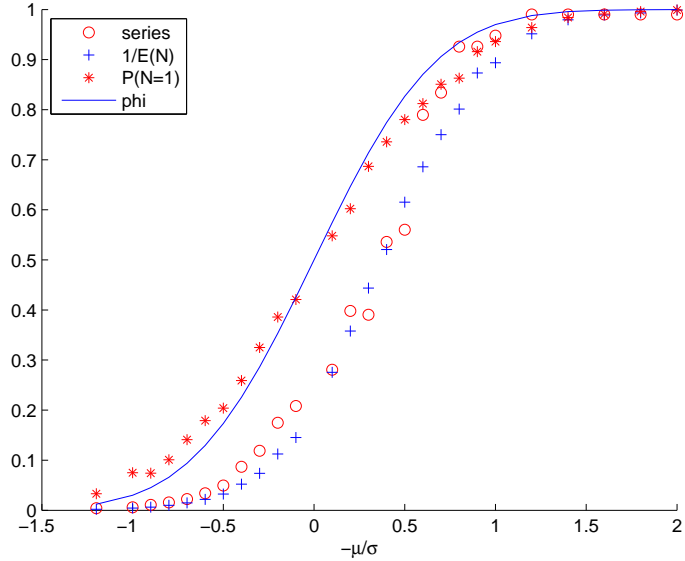}
}
 \figtext{ 
 }
 \vspace{2mm}
\caption[]{Plots of the function $\Phi(-\pi^{1/4}\tilde\mu/\tilde\sigma)$ as a
function of $-\tilde\mu/\tilde\sigma$ (full line), and numerical estimates of
$\prob{N=1}$ (stars), $1/\expec{N}$ (crosses) and $1-\lambda_0$ (circles).
}
\label{fig_comparison2}
\end{figure}

As predicted by~\eqref{trans07}, $\prob{N=1}$ is indeed close to the theoretical
value $\Phi(-\pi^{1/4}\tilde\mu/\tilde\sigma)$. Recall from~\eqref{res07} that
$1/\expecin{\mu_0}{N}$, $\probin{\mu_0}{N=1}$ and $1-\lambda_0$ would be equal
if the initial distribution $\mu_0$ after a spike were equal to the QSD $\pi_0$.
The simulations show that $1/\expec{N}$ and $1-\lambda_0$ are systematically
smaller than $\prob{N=1}$. The difference between $\prob{N=1}$ and
$1-\lambda_0$ is a measure of how far away $\mu_0$ is from the QSD $\pi_0$. The
difference between $1/\expec{N}$ and $1-\lambda_0$ also depends on the spectral
gap between $\lambda_0$ and the remaining spectrum of the Markov kernel. Note
that $1/\expec{N}$ and $1-\lambda_0$ seem to follow a similar curve as
$\prob{N=1}$, but with a shifted value of $\tilde\mu$. We do not have any 
explanation for this at the moment. 


\section{Conclusion and outlook}
\label{sec_conc}

We have shown that in the excitable regime, and when the stationary point $P$ is
a focus, the interspike interval statistics of the stochastic FitzHugh--Nagumo
equations can be characterised in terms of the random number of SAOs $N$. The
distribution of $N$ is asymptotically geometric, with parameter $1-\lambda_0$,
where $\lambda_0$ is the principal eigenvalue of a substochastic Markov chain,
describing a random Poincar\'e map. This result is in fact fairly general,
as it does not depend at all on the details of the system. It only requires the
deterministic system to admit an invariant region where the dynamics involves
(damped) oscillations, so that a Poincar\'e section can be defined in a
meaningful way. Thus Theorem~\ref{thm_geometric} will hold true for a large
class of such systems. 

To be useful for applications, this qualitative result has to be complemented by
quantitative estimates of the relevant parameters. Theorem~\ref{thm_weak}
provides such estimates for $\lambda_0$ and the expected number of SAOs in the
weak-noise regime $\sigma_1^2+\sigma_2^2\ll(\eps^{1/4}\delta)^2$. We have
obtained one-sided estimates on these quantities, which follow from the
construction of an almost invariant region $A$ for the Markov chain. It is
possible to obtain two-sided estimates by deriving more precise properties for
the Markov chain, in particular a lower bound on the probability of leaving the
complement of $A$. We expect the exponent
$(\eps^{1/4}\delta)^2/(\sigma_1^2+\sigma_2^2)$ to be sharp in the case
$\delta\gg\eps$, since this corresponds to the drift $\tilde\mu$ in the
expression~\eqref{Rxiz02} for $\dot{z}$ dominating the error terms of order
$\sqrt{\eps}$ due to higher-order nonlinear terms. For smaller $\delta$,
however, there is a competition between the two terms, the effect of which is
not clear and has to be investigated in more detail. The same problem prevents
us from deriving any bounds for $\delta\leqs\eps$ when the parameter $c$
defining the FitzHugh--Nagumo equations is different from zero. It may be
possible to achieve a better control on the nonlinear terms by additional
changes of variables.  

For intermediate and strong noise, we obtained an approximation~\eqref{trans07}
for the probability $\prob{N=1}$ of spiking immediately, showing that the
transition from rare to frequent spikes is governed by the distribution function
$\Phi$ of the normal law. Though we didn't obtain rigorous bounds on the
principal eigenvalue and expected number of SAOs in this regime, simulations
show a fairly good agreement with the approximation for $\prob{N=1}$. The
results on the Markov kernel contained in the appendix should in fact yield more
precise information on $\lambda_0$ and the law of $N$, via approximations for
the quasistationary distribution $\pi_0$. Generally speaking, however, we need
better tools to approximate QSDs, principal eigenvalues and the spectral gap of
substochastic Markov chains.

Finally, let us note that the approach presented here should be applicable to
other excitable systems involving oscillations. For instance, for some parameter
values, the Morris--Lecar equations~\cite{MorrisLecar81} admit a stable
stationary point surrounded by an unstable and a stable periodic orbit. In a
recent work~\cite{Ditlevsen_Greenwood_11}, Ditlevsen and Greenwood have combined
this fact and results on linear oscillatory systems with
noise~\cite{Baxendale_Greenwood_11} to relate the
spike statistics to those of an integrate-and-fire model. It would be
interesting to implement the Markov-chain approach in this situation as well.


\appendix


\section{Dynamics near the separatrix}
\label{sec_separatrix}


The appendix contains some of the more technical computations required for the
proof of Theorem~\ref{thm_weak}. We treat separately the dynamics near the
separatrix, and during the remainder of an SAO. 

In this section, we use the equations in $(\xi,z)$-variables given
by~\eqref{Rxiz02} to describe the dynamics in a neighbourhood of the separatrix.
To be more specific, we will assume that $z$ is small, of the order of some
power of $\mu$, and that $\xi$ varies in an interval $[-L,L]$, where the
parameter $L$ is given by~\eqref{bounds12}. Let $F_\pm$ be the two broken lines
defined in~\eqref{bounds10} and~\eqref{bounds11}. Given an initial condition
$(-L,z_0)\in F_-$, our goal is to estimate where the sample path starting in
$(-L,z_0)$ hits $F_+$ for the first time. This will characterise the first part
of the Markov kernel $K$. 


\subsection{The linearised process}
\label{ssec_ptrans} 


Before analysing the full dynamics of~\eqref{Rxiz02} we consider some
approximations of the system. The fact that $\xi_t\simeq\xi_0+t/2$ for small
$z$ motivates the change of variable 
\begin{equation}
\label{preuve_ecart01}
\xi = \dfrac{t}{2} + u\;,
\end{equation}
which transforms the system \eqref{Rxiz02} into
\begin{equation}
\label{preuve_ecart02}
\begin{split}
\6u_t &= \bigpar{ - z_t +\Order{\tilde\eps} }
\6t + \tilde{\sigma}_1 \6 W_t^{(1)}\;, \\
\6z_t &= \bigpar{\tilde{ \mu}  + t  z_t + 2u_t  z_t + \Order{\tilde\eps}
} \6t -  \tilde{\sigma}_1 t \6W_t^{(1)} -  2\tilde{\sigma}_1 
u_t \6W_t^{(1)}+ \tilde{\sigma}_2  \6W_t^{(2)}\;,
\end{split}
\end{equation}
where we write $\tilde\eps = \sqrt{\eps} (L^4+cL^2)$. We choose an initial
condition $(0,z_0)$ at time $t_0=-2L$. 
As a first approximation, consider the deterministic system 
\begin{equation}
\label{preuve_ecart03}
\begin{split}
\6u^0_t &= -z^0_t \6t\;,  \\
\6z^0_t &= \bigpar{\tilde{\mu} + t z^0_t} \6t\;.
\end{split}
\end{equation}
The solution of the second equation is given by 
\begin{equation}
 \label{preuve_ecart03A}
z^0_t = \e^{t^2/2} \biggbrak{z_0 \e^{-t_0^2/2} + \tilde\mu \int_{t_0}^t
\e^{-s^2/2}\6s}\;. 
\end{equation} 
In particular, at time $T=2L$, we have $\xi_T=L+u_t\simeq L$ and the location of
the first-hitting point of $F_+$ is approximated by  
\begin{equation}
 \label{preuve_ecart03A:1}
z^0_T = z_0 + \tilde\mu \e^{T^2/2}  \int_{t_0}^T
\e^{-s^2/2}\6s = z_0 + \Order{\tilde\mu^{1-\gamma}}\;.
\end{equation} 
As a second approximation, we incorporate the noise terms and consider the
linear SDE 
\begin{equation}
\label{preuve_ecart03B}
\begin{split}
\6u^1_t &= -z^1_t \6t + \tilde{\sigma}_1 \6W_t^{(1)}\;, \\
\6z^1_t &= \bigpar{ \tilde{\mu} + t z^1_t } \6t - \tilde{\sigma}_1 t
\6W_t^{(1)} + \tilde{\sigma}_2 \6W_t^{(2)}\;.
\end{split}
\end{equation}
Let us now quantify the deviation between $(u^1_t,z^1_t)$ and $(u^0_t,z^0_t)$.

\begin{prop}
\label{prop_xiz_linear}
Let 
\begin{equation}
\label{prop_ecart03}
\zeta(s) = \e^{s^2} \biggbrak{ \e^{-t_0^2} + \int_{t_0}^s \e^{-u^2} \6u}\;. 
\end{equation}
Then there exists a constant $M>0$ such that for all $t\geqs t_0$, 
all $h,h_1,h_2>0$ and all $\rho\in(0,\tilde\mu^{2\gamma}/M)$, 
\begin{equation}
\label{preuve_ecart18}
\biggprob{ \sup_{t_0 \leqs s \leqs t} \dfrac{|z^1_s-z^0_s|}{\sqrt{\zeta(s)}}
\geqs h
} \leqs \dfrac{2(t-t_0)}{\rho} \exp \biggset{ -\dfrac{1}{8}
\dfrac{h^2}{\tilde{\sigma}^2}
\bigpar{ 1-M \rho\, \tilde\mu^{-2\gamma} } }
\end{equation} 
and 
\begin{multline}
 \label{preuve_ecart21}
\qquad
\biggprob{ \sup_{t_0 \leqs s \leqs t} |u^1_s-u^0_s | \geqs h_1 + 
h_2 \int_{t_0}^t\sqrt{\zeta(s)}\6s} \\
\leqs 2\exp \biggset{ {-\dfrac{h_1^2}{2(t-t_0)
\tilde{\sigma}_1^2}} } 
+ \dfrac{2(t-t_0)}{\rho} \exp \biggset{ -\dfrac{1}{8}
\dfrac{h_2^2}{\tilde{\sigma}^2}
\bigpar{ 1-M \rho\, \tilde\mu^{-2\gamma} } }\;.
\qquad
\end{multline}
\end{prop}

\begin{proof}
The difference $(x^1,y^1)=(u^1-u^0,z^1-z^0)$ satisfies the system 
\begin{equation}
\label{preuve_ecart06}
\begin{split}
\6x_t^1 &= -y_t^1 \6t + \tilde{\sigma}_1 \6W_t^{(1)}\;, \\
\6y_t^1 &= t\, y_t^1 \6t - \tilde{\sigma}_1 t \6W_t^{(1)} + 
\tilde{\sigma}_2 \6W_t^{(2)} \;.
\end{split}
\end{equation}
The second equation admits the solution 
\begin{equation}
\label{preuve_ecart07}
y_t^1 =  \tilde{\sigma}_2 \e^{t^2/2} \int_{t_0}^t \e^{-s^2/2}
\6W_s^{(2)} - \tilde{\sigma}_1 \e^{t^2/2} \int_{t_0}^t s
e^{-s^2/2} \6W_s^{(1)} 
=: y_t^{1,1} + y_t^{1,2}\;.
\end{equation}
We first estimate $y_t^{1,1}$. Let $u_0=t_0<u_1<\dots<u_K=t$ be a partition of
$[t_0,t]$. The Bernstein-like estimate \cite[Lemma~3.2]{BG1} yields the bound 
\begin{equation}
\label{preuve_ecart09}
\biggprob{ \sup_{t_0 \leqs s \leqs t} \dfrac{1}{\sqrt{\zeta(s)}}
\tilde{\sigma}_2 \biggabs{\int_{t_0}^s
\e^{( s^2 - u^2)/2} \6W_u} \geqs H_0} \leqs 2
\sum_{k=1}^K
P_k
\end{equation}
for any $H_0>0$, 
where 
\begin{equation}
\label{preuve_ecart11}
P_k \leqs \exp \biggset{ -\dfrac{1}{2} \dfrac{H_0^2}{\tilde{\sigma}_2^2}
\inf_{u_{k-1}
\leqs s \leqs u_k} \dfrac{\zeta(s)}{\zeta(u_k)} \e^{u_k^2-s^2} }\;.
\end{equation}
The definition of $\zeta(s)$ implies 
\begin{equation}
\label{preuve_ecart12}
\dfrac{\zeta(s)}{\zeta(u_k)} \e^{u_k^2-s^2} = 1 - \dfrac{1}{\zeta(u_k)}
\int_s^{u_k} \e^{u_k^2-u^2} \6u
\geqs 1 - \int_s^{u_k} \e^{t_0^2-u^2}\6u\;.
\end{equation}
Note that $\e^{t_0^2}=\e^{4L^2}=\Order{\tilde\mu^{-2\gamma}}$. 
For a uniform partition given by $u_{k}-u_{k-1} = \rho$ with $\rho \ll
\tilde\mu^{2\gamma}$, we can bound this last expression below by 
\begin{equation}
\label{preuve_ecart13}
1 - M \rho\, \tilde\mu^{-2\gamma}
\end{equation}
for some constant $M$. This yields 
\begin{equation}
\label{preuve_ecart14}
\biggprob{ \sup_{0 \leqs s \leqs t} \dfrac{|y_s^{1,1}|}{\sqrt{\zeta(s)}} \geqs
H_0
} \leqs \dfrac{2(t-t_0)}{\rho} \exp \biggset{ -\dfrac{1}{2}
\dfrac{H_0^2}{\tilde{\sigma}_2^2}
\bigpar{ 1-M \rho\, \tilde\mu^{-2\gamma} } }\;.
\end{equation}
Doing the same for $y_s^{1,2}$ we obtain 
\begin{equation}
\label{preuve_ecart15}
\biggprob{ \sup_{0 \leqs s \leqs t} \dfrac{|y_s^{1,2}|}{\sqrt{\zeta(s)}} \geqs
H_1
} \leqs \dfrac{2(t-t_0)}{\rho} \exp \biggset{ -\dfrac{1}{2}
\dfrac{H_1^2}{\tilde{\sigma}_1^2}
\bigpar{ 1-M \rho\, \tilde\mu^{-2\gamma} } }
\end{equation}
for any $H_1>0$. 
Letting $h=H_0 + H_1$ with $H_0 = H_1=h/2$, we obtain~\eqref{preuve_ecart18}. 
Now we can express $x_t^1$ in terms of $y_t^1$ by
\begin{equation}
\label{preuve_ecart19}
x_t^1 = -\int_{t_0}^t y_s^1 \6s + \tilde{\sigma_1} \int_{t_0}^t  \6W_s^{(1)}\;.
\end{equation}
Then the Bernstein inequality 
\begin{equation}
\label{preuve_ecart20}
\biggprob{ \sup_{0 \leqs s \leqs t} \biggabs{\tilde{\sigma_1} \int_{t_0}^t 
\6W_s^{(1)}}
\geqs h_1  }  \leqs
2\exp\biggset{-\frac{h_1^2}{2(t-t_0) \tilde{\sigma}_1^2}}
\end{equation}
yields~\eqref{preuve_ecart21}. 
\end{proof}


\subsection{The nonlinear equation}
\label{ssec_xiz_nonlin} 


We now turn to the analysis of the full system~\eqref{Rxiz02}, or,
equivalently, 
\eqref{preuve_ecart02}. Before that, we state a generalised Bernstein inequality
that we will need several times in the sequel. 
Let $W_t$ be an $n$-dimensional standard Brownian motion, and 
consider the martingale 
\begin{equation}
 \label{eq:teclem01}
M_t = \int_{t_0}^t g(X_s,s) \,\6W_s
= \sum_{i=1}^{n}  \int_{t_0}^t g_{i}(X_t,t) \,\6W^{(i)}_t \;,
\end{equation}  
where $g=(g_{1},\dots,g_{n})$ takes values in $\R^{n}$ and the process $X_t$ is
assumed to be adapted to the filtration generated by $W_t$. Then we have the
following result (for the proof, see~\cite[Lemma D.8]{BGK12}):

\begin{lemma}
\label{lem_Bernstein}
Assume that the integrand satisfies 
\begin{equation}
 \label{eq:teclem02}
g(X_t,t) g(X_t,t)^T \leqs G(t)^2
\end{equation} 
almost surely, for a deterministic function $G(t)$, and that the integral 
\begin{equation}
 \label{eq:teclem03}
V(t) = \int_{t_0}^t G(s)^2\,\6s
\end{equation} 
is finite. Then 
\begin{equation}
 \label{eq:teclem04}
\biggprob{ \sup_{{t_0}\leqs s\leqs t} M_s > x}
\leqs \e^{-x^2/2V(t)}
\end{equation} 
for any $x>0$. 
\end{lemma}

\begin{prop}
Assume $z_0=\Order{\tilde\mu^{1-\gamma}}$. There exist constants $C,\kappa,M>0$
such that for $t_0 \leqs t \leqs T + \Order{\abs{\log\tilde\mu}^{-1/2}}$, all 
$\tilde\sigma\leqs\tilde\mu$ and $H>0$, 
\begin{equation}
\biggprob{ \sup_{t_0 \leqs s \leqs t} \dfrac{|z_s - z_s^0|}{\sqrt{\zeta(s)}}
\geqs H } \leqs \frac{CT}{\tilde\mu^{2\gamma}} \biggpar{
\exp \biggset{ - \kappa 
\frac{\bigbrak{H-M(T^2\tilde\mu^{2-4\gamma} +
T\tilde\eps\tilde\mu^{-2\gamma})}^2}{\tilde\sigma^2} } 
+  \e^{-\kappa \tilde\mu^2/\tilde\sigma^2} }
\label{prop_ecart04}
\end{equation}
and for all $H'>0$, 
\begin{equation}
\biggprob{ \sup_{t_0 \leqs s \leqs t} |u_s - u_s^0|
\geqs H' } \leqs \frac{CT}{\tilde\mu^{2\gamma}} \biggpar{
\exp \biggset{ - \kappa 
\frac{\bigbrak{H'-M(T^2\tilde\mu^{2-4\gamma} +
T\tilde\eps\tilde\mu^{-2\gamma})}^2}{\tilde\sigma^2\tilde\mu^{-2\gamma}} } 
+ \e^{-\kappa \tilde\mu^2/\tilde\sigma^2}}\;.
\label{prop_ecart06}
\end{equation}
\end{prop}

\begin{proof}
The upper bound on $t$ implies that $\e^{t^2/2}=\Order{\tilde\mu^{-\gamma}}$.
Thus it follows from~\eqref{preuve_ecart03} and~\eqref{preuve_ecart03A} that 
\begin{equation}
 \label{preuve_ecart_101}
z^0_s = \Order{\tilde\mu^{1-\gamma}} 
\qquad\text{and}\qquad
u^0_s = \Order{T\tilde\mu^{1-\gamma}} 
\end{equation} 
for $t_0\leqs s\leqs t$. Given $h, h_1, h_2>0$, 
we introduce the stopping times 
\begin{align}
\nonumber
\tau_1 &= \inf\Bigsetsuch{s\geqs t_0}
{\abs{z^1_s - z^0_s} \geqs h\sqrt{\zeta(s)}}\;, \\
\tau_2 &= \inf\biggsetsuch{s\geqs t_0}
{\abs{u^1_s - u^0_s} \geqs h_1 + h_2 \int_{t_0}^t \sqrt{\zeta(s)}\6s}\;.
 \label{preuve_ecart_102} 
\end{align} 
The integral of $\sqrt{\zeta(s)}$ is of order $T\tilde\mu^{-\gamma}$ at most. 
Thus choosing $h=h_1=h_2=\tilde\mu$ guarantees
that 
\begin{equation}
 \label{preuve_ecart_103}
z^1_s = \Order{\tilde\mu^{1-\gamma}} 
\qquad\text{and}\qquad
u^1_s = \Order{T\tilde\mu^{1-\gamma}} 
\end{equation} 
for $t_0\leqs s\leqs t\wedge \tau_1\wedge\tau_2$.
For these values of $h$, $h_1$ and $h_2$, Proposition~\ref{prop_xiz_linear}
implies that 
\begin{align}
\nonumber
\bigprob{\tau_1 < t} &\leqs cT\tilde\mu^{-2\gamma} \e^{-\kappa
\tilde\mu^2/\tilde\sigma^2}\;, \\
\bigprob{\tau_2 < t} &\leqs  
cT\tilde\mu^{-2\gamma} \e^{-\kappa
\tilde\mu^2/\tilde\sigma^2} 
 \label{preuve_ecart_104}
\end{align} 
for some constants $\kappa, c>0$. 
We consider the difference $(x_t^2,y_t^2)=(u_t,z_t)-(u^1_t,z^1_t)$,
which satisfies the system of SDEs
\begin{equation}
 \label{preuve_ecart23}
\begin{split}
\6x_t^2 &= \bigpar{- y_t^2 + \Order{\tilde\eps}} \6t\;, \\
\6y_t^2 &= \bigbrak{t y_t^2 + 2 (u_t^1 + x_t^2 ) ( z_t^1 +
y_t^2)+ \Order{\tilde\eps} } \6t - 2 \tilde{\sigma}_1 ( u_t^1 + x_t^2 ) 
\6W_t^{(1)} \;.
\end{split}
\end{equation}
We introduce a Lyapunov function $U_t>0$ defined by 
\begin{equation}
 \label{preuve_ecart24}
(U_t-C_0)^2 = \dfrac{\left(x_t^2\right)^2 + \left(y_t^2\right)^2 }{2}\;.
\end{equation}
The constant $C_0$ will be chosen in order to kill the second-order terms
arising from It\^o's formula. 
Let 
\begin{equation}
 \label{preuve_ecart25}
\tau^* = \inf \setsuch{ t\geqs t_0}{U_t=1}\;. 
\end{equation}
Applying It\^o's formula and choosing $C_0$ of order
$\tilde\sigma_1^2\tilde\mu^{-(1-\gamma)}$ yields 
\begin{equation}
 \label{preuve_ecart26}
\6U_t \leqs \bigbrak{C_1 + C_2(t) U_t}\6t
+ \tilde{\sigma}_1 g(t) \6W_t^1 \;,
\end{equation}
where (using the fact that $\tilde\sigma \leqs \tilde\mu$)
\begin{align}
\nonumber 
C_1 &= \bigOrder{T\tilde\mu^{2-2\gamma}} + \Order{\tilde\eps}\;, \\
C_2(t) &= t\vee0 + \Order{T\tilde\mu^{1-\gamma}}\;,
\end{align}
and $g(t)$ is at most of order $1$ 
for $t\leqs \tau_1\wedge\tau_2\wedge\tau^*$. Hence
\begin{equation}
 \label{preuve_ecart27}
U_{t\wedge\tau_1\wedge\tau_2\wedge\tau^*} \leqs U_{t_0} + C_1(t-t_0) +
\int_{t_0}^{t\wedge\tau_1\wedge\tau_2\wedge\tau^*} C_2(s)U_s\6s + 
 \tilde{\sigma}_1 \int_{t_0}^{t\wedge\tau_1\wedge\tau_2\wedge\tau^*} g(s)
\6W_s^1\;.
\end{equation}
We introduce a last stopping time  
\begin{equation}
 \label{preuve_ecart27B}
\tau_3 = \inf\biggsetsuch{t\geqs t_0}{\biggabs{\tilde{\sigma}_1
\int_0^{t\wedge\tau_1\wedge\tau_2\wedge\tau^*}
g(s) \6W_s^1 } \geqs h_3} \;.
\end{equation}
Then Lemma~\ref{lem_Bernstein} implies 
\begin{equation}
 \label{preuve_ecart28}
\bigprob{ \tau_3 < t  }  \leqs
\e^{-\kappa_3h_3^2/\tilde\sigma_1^2}
\end{equation}
for a $\kappa_3>0$. 
Applying Gronwall's lemma to~\eqref{preuve_ecart27} we get 
\begin{align}
\nonumber
U_{t\wedge\tau_1\wedge\tau_2\wedge\tau_3\wedge\tau^*} &\leqs \bigbrak{U_{t_0} +
C_1(t-t_0) + h_3}
\exp\biggset{\int_{t_0}^{t\wedge\tau_1\wedge\tau_2\wedge\tau_3\wedge\tau^*}
C_2(u)\6u}\\
&= \Order{T^2\tilde\mu^{2-3\gamma}}
+ \Order{\tilde\sigma^2 T^2\tilde\mu^{1-2\gamma}}
+ \Order{\tilde\eps T \tilde\mu^{-\gamma}}\;.
 \label{preuve_ecart29}
\end{align}
This shows in particular that $\tau^*>t$, provided we take $\gamma$ small
enough. Now~\eqref{prop_ecart04} follows from the decomposition 
\begin{align}
\nonumber
\biggprob{ \sup_{t_0 \leqs s \leqs t} \dfrac{|z_s - z_s^0|}{\sqrt{\zeta(s)}}
\geqs H } \leqs{}& \biggprob{ \sup_{t_0 \leqs s \leqs
t\wedge\tau_1\wedge\tau_2\wedge\tau_3} \dfrac{|z^1_s - z_s^0|}{\sqrt{\zeta(s)}}
\geqs H - \sup_{t_0 \leqs s \leqs
t\wedge\tau_1\wedge\tau_2\wedge\tau_3} \frac{U_s}{\sqrt{\zeta(s)}}} \\
&{}+ \bigprob{\tau_1>t} + \bigprob{\tau_2>t} + \bigprob{\tau_3>t} \;,
 \label{preuve_ecart30} 
\end{align} 
and~\eqref{prop_ecart06} is obtained in a similar way.
\end{proof}


We can now derive bounds for the contribution of the motion near the separatrix
to the Markov kernel.

\begin{prop}
\label{prop_xiz_kernel}
Fix some $\gamma\in(0,1/4)$ and an initial condition $(\xi_0,z_0)=(-L,z_0)\in
F_-$ with $\abs{z_0}=\Order{\tilde\mu^{1-\gamma}}$.
\begin{enum}
\item	Assume $c=0$. Then there exist constants
$C,\kappa_1, h_0>0$ such that the sample path starting in $(\xi_0,z_0)$ will hit
$F_+$ for the first time at a point $(\xi_1,z_1)$ such that 
\begin{equation}
 \label{xiz_kernel_02}
\Bigprob{z_1 \leqs z^0_T-\tilde\mu} \leqs \frac{C}{\tilde\mu^{2\gamma}} 
\exp\biggset{-\kappa_1\frac{\tilde\mu^2}{\tilde\sigma^2}}\;.
\end{equation}

\item	If $c\neq0$, but $\sqrt{\eps} \leqs \tilde\mu^{1+2\gamma+\theta}$ for
some $\theta>0$, then the first-hitting point of $F_+$ always satisfies 
\begin{equation}
 \label{xiz_kernel_01}
\Bigprob{\abs{z_1-z^0_T} \geqs \tilde\mu} 
\leqs \frac{C}{\tilde\mu^{2\gamma}} 
\exp\biggset{-\kappa_1\frac{\tilde\mu^2}{\tilde\sigma^2}}\;.
\end{equation}  
\end{enum}
\end{prop}
\begin{proof}
Consider first the case $\sqrt{\eps} \leqs \tilde\mu^{1+2\gamma+\theta}$. 
For any $h>0$ we can write 
\begin{equation}
 \label{xiz_kernel_04}
\bigprob{\abs{z_\tau-z^0_T} \geqs \tilde\mu}
\leqs  \bigprob{\abs{z_\tau-z^0_\tau} \geqs \tilde\mu - h\tilde\mu^{2-2\gamma}} 
+ \bigprob{\abs{z^0_\tau-z^0_T} \geqs h\tilde\mu^{2-2\gamma}}\;.
\end{equation} 
The first term on the right-hand side can be bounded, 
using~\eqref{prop_ecart04}, by a term of order
$\tilde\mu^{-2\gamma}\e^{-\kappa_1\tilde\mu^2/\tilde\sigma^2}$. The conditions
on $\gamma$ and $\sqrt{\eps}$ ensure that the error terms in the exponent
in~\eqref{prop_ecart04} are negligible.
 
To bound the second term on the right-hand side, we note
that~\eqref{preuve_ecart03} implies that $\abs{z^0_\tau-z^0_T}$ has order 
$\tilde\mu^{1-\gamma}\abs{\tau-T}$. Furthermore, the definitions of $\tau$ and
$u^0$ imply that $\abs{\tau-T}=2\abs{u_t-u^0_t}+\Order{\tilde\mu^{1-\gamma}}$. 
This shows that 
\begin{equation}
 \label{xiz_kernel_03}
\bigprob{\abs{z^0_\tau-z^0_T} \geqs h\tilde\mu^{2-2\gamma}}
\leqs \bigprob{\abs{u_t-u^0_t} \geqs hc_1\tilde\mu^{1-\gamma} -c_2
\mu^{1-\gamma}}
\end{equation} 
for some constants $c_1,c_2>0$. Taking $h=2c_2/c_1$ and
using~\eqref{prop_ecart06} yields a similar bound as for the first term. 

In the case $c=0$, we can conclude in the
same way by observing that $z_t$ is bounded below by its value for $\eps=0$,
the $\eps$-dependent term of $\6z_t$ in~\eqref{Rxiz02} being positive. 
Thus we need no condition on $\sqrt{\eps}$ for the error terms in the exponent
to be negligible.
\end{proof}


\section{Dynamics during an SAO}
\label{sec_SAO}


\subsection{Action--angle-type variables}
\label{ssec_Kphi}

In this section, we construct another set of coordinates allowing to describe
the dynamics during a small-amplitude oscillation.
Recall that in the limit  $\eps \to 0$ and $\tilde{\mu} \to 0$, the
deterministic
system~\eqref{weak004} admits a first integral
\begin{equation}
\label{Kphi-02}
\KK = 2z \e^{-2z - 2\xi^2+1}\;.
\end{equation}
The separatrix is given in this limit by $\KK=0$, while $\KK=1$
corresponds to the stationary point $P$. When
$\eps$ and $\tilde\mu$ are positive, we obtain 
\begin{equation}
 \label{Kphi-03}
\dot{\KK} = \biggbrak{2\tilde\mu(1-2z) + \sqrt{\eps}
\biggpar{\frac4{9\alpha_*^2}\xi^4 + c \bigbrak{4z\xi^2-6\xi^2+8z^2-4z+1}}} 
\e^{-2z - 2\xi^2+1}\;. 
\end{equation} 
Observe that if $c=0$, the term of order $\sqrt{\eps}$ is strictly positive. 

In order to analyse the dynamics in more detail, it is useful to introduce an
angle variable $\phi$. We define a coordinate transformation
from $(0,1]\times\fS^1$ to $\R\times\R_+$ by 
\begin{equation}
\label{Kphi11}
\begin{split}
\xi &= -\sqrt{\dfrac{-\log \KK}{2}} \sin \phi \\
z &= \dfrac{1}{2} \left( 1 + f \biggpar{ \sqrt{\dfrac{-\log \KK}{2}} \cos \phi
}
\right) \;.
\end{split}
\end{equation}
Here $f: \R \to (-1,+\infty)$ is defined as the solution of
\begin{equation}
\label{Kphi01}
\log(1+f(u))-f(u) = -2u^2
\end{equation}
such that
\begin{equation}
\label{Kphi02}
\sign f(u) = \sign u \;.
\end{equation}
The graph of $f$ is plotted in~\figref{fig_f}. 

\begin{figure}
\centerline{\includegraphics*[clip=true,width=100mm]
{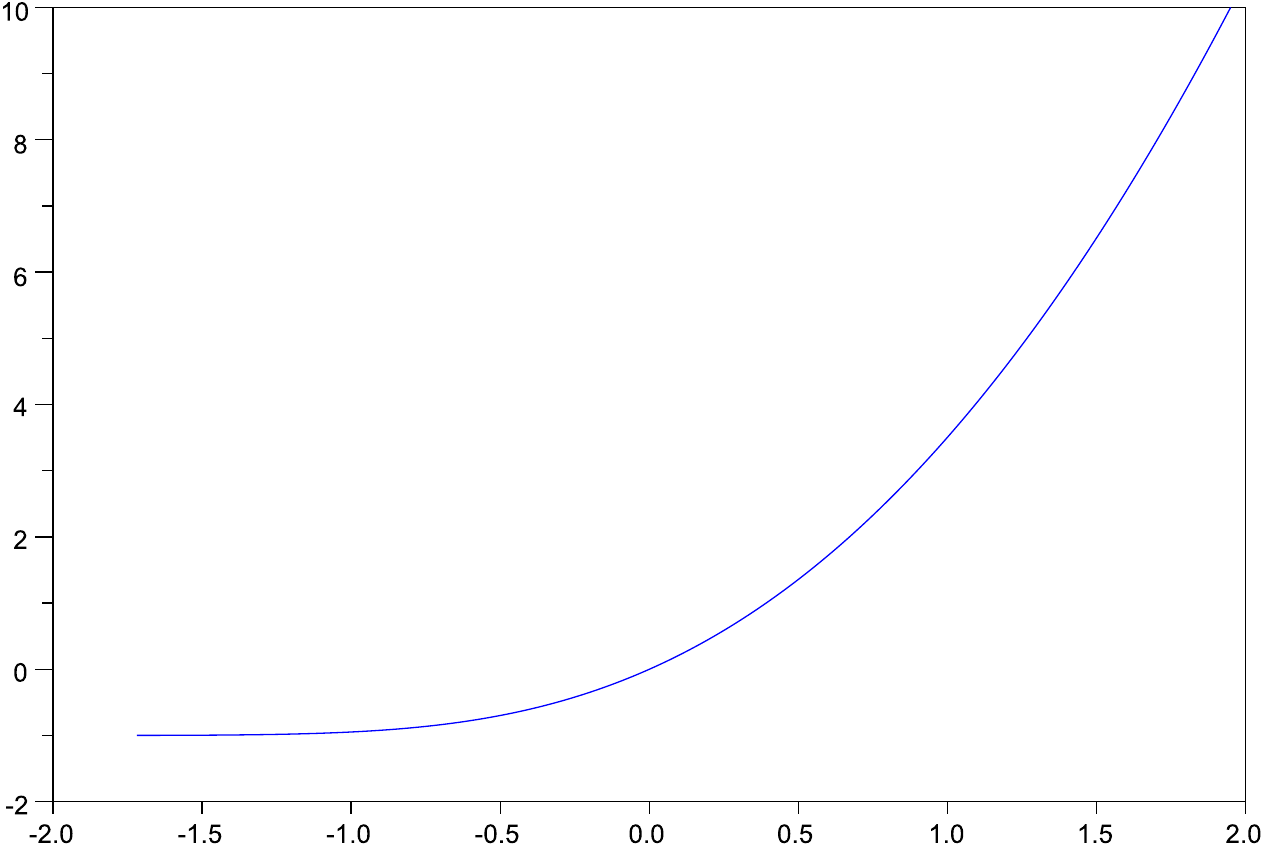}
}
 \figtext{ 
 	\writefig	12.7	0.6	$u$
 	\writefig	1.4	6.8	$f(u)$
 }
\caption[]{Graph of $u\mapsto f(u)$.}
\label{fig_f}
\end{figure}

\begin{lemma}
\label{lem_Kphi_f}
The function $f$ has the following properties:
\begin{itemiz}
\item 	Lower bounds: 
\begin{equation}
 \label{Kphi03}
f(u) > -1 
\qquad\text{and}\qquad
 f(u) \geqs 2u
\qquad \forall u\in\R\;.
\end{equation} 
\item	Upper bounds:
There exist constants $C_1,C_2>0$ and a function $r:\R_-\to\R$, with 
$0\leqs r(u)\leqs C_1 \e^{-1-2u^2}$, such that 
\begin{align}
 \label{Kphi04A} 
f(u) &\leqs C_2 u+2u^2 &&\forall u\geqs 0\;, \\
f(u) &= -1 + \e^{-1-2u^2}[1+r(u)] &&\forall u\leqs 0\;.
 \label{Kphi04B} 
\end{align} 
\item Derivatives: $f\in \cC^{\infty}$ and 
\begin{align}
\label{Kpi07}
f'(u) &= 4u \dfrac{1+f(u)}{f(u)}\;, \\
\label{Kpi08}
f''(u) &= 4 \dfrac{1+f(u)}{f(u)} \biggpar{ 1 - 4 \dfrac{u^2}{f(u)^2} }\;.
\end{align}
\item	There exists a constant $M>0$ such that 
\begin{equation}
\label{Kpi09}
0 < f''(u) \leqs M \quad \forall u \in \R\;.
\end{equation}
\end{itemiz}
\end{lemma}
\begin{proof}
The results follow directly from the implicit function theorem and elementary
calculus. 
\end{proof}

We can now derive an expression for the SDE in coordinates $(\KK,\phi)$. To
ease notation, we introduce the function 
\begin{equation}
\label{Kphi-10}
X = X(\KK,\phi) = \sqrt{\frac{-\log \KK}{2}} \cos \phi\;,
\end{equation}
a parameter $\tilde{\sigma}>0$ defined by 
\begin{equation}
 \label{Kphi-10b}
 \tilde{\sigma}^2 = \tilde{\sigma}_1^2 + \tilde{\sigma}_2^2\;,
\end{equation} 
and the two-dimensional Brownian motion 
$\6W_t = (\6\tilde{W}_t^{(1)},\6\tilde{W}_t^{(2)})^T$.

\begin{prop}
\label{prop_Kphi_EDS}
For $z>0$, the system of SDEs~\eqref{Rxiz02} is equivalent to the system
\begin{equation}
\label{Kphi-21}
\begin{split}
\6\KK_t &=  \tilde{\mu} f_1(\KK_t,\phi_t) \, \6t + \tilde{\sigma}
\psi_1(\KK_t,\phi_t)
\6W_t \\
\6\phi_t &=   f_2(\KK_t,\phi_t) \, \6t  +  \tilde{\sigma} \psi_2(\KK_t,\phi_t)
\6W_t \;,
\end{split}
\end{equation}
where we introduced the following notations.
\begin{itemiz}
\item 	The new drift terms are of the form 
\begin{align}
\label{Kphi-11}
f_1(\KK,\phi) &= -2\KK\dfrac{f(X)}{1+f(X)} 
\biggbrak{1 + \frac{\sqrt{\eps}}{\tilde\mu} R_{\KK,\eps}(\KK,\phi)
+ \frac{\tilde\sigma^2}{\tilde\mu} R_{\KK,\sigma}(\KK,\phi)}\;, \\
\label{Kphi-12}
f_2(\KK,\phi) &= \dfrac{f(X)}{2 X}  
\biggbrak{1 + \frac{2\tilde\mu\tan\phi}{\log \KK(1+f(X))}
+ \sqrt{\eps} R_{\phi,\eps}(\KK,\phi)
+ \tilde\sigma^2 R_{\phi,\sigma}(\KK,\phi)}\;.
\end{align}

\item	The remainders in the drift terms are bounded as follows. 
Let 
\begin{equation}
 \label{Kphi-13}
\rho(\KK,\phi) = 
\begin{cases}
\sqrt{\abs{\log \KK}} & \text{if $\cos\phi\geqs0$\;,} \\
\KK^{-\cos^2\phi} & \text{if $\cos\phi<0$\;. }
\end{cases} 
\end{equation} 
Then there exists a constant $M_1>0$ such that for all $\KK\in(0,1)$ and all
$\phi\in\fS^1$, 
\begin{align}
\nonumber
\abs{ R_{\KK,\eps}(\KK,\phi)} &\leqs M_1 \abs{\log \KK}^2\;,
&
\abs{ R_{\KK,\sigma}(\KK,\phi)} &\leqs M_1 \rho(\KK,\phi) \;, \\
\abs{R_{\phi,\eps}(\KK,\phi)} &\leqs M_1 \abs{\log \KK}^{3/2} \rho(\KK,\phi)\;, 
&
\abs{ R_{\phi,\sigma}(\KK,\phi)} &\leqs M_1 \rho(\KK,\phi)^2/\abs{\log \KK}\;.
 \label{Kphi-11B}
\end{align} 
Furthermore, if $c=0$ then $-f(X)R_{\KK,\eps}(\KK,\phi) \geqs 0$. 

\item	The diffusion coefficients are given by 
\begin{align}
\nonumber
    \psi_1(\KK,\phi)&= \biggpar{2\sqrt{2} \frac{\tilde\sigma_1}{\tilde\sigma}
\KK 
\biggbrak{\sqrt{-\log \KK} - \frac{f(X)}{1+f(X)}} 
\sin \phi  ,
     -2\frac{\tilde\sigma_2}{\tilde\sigma} \KK \dfrac{f(X) }{1+f(X)} }\;,
 \\
    \psi_{2}(\KK,\phi)&= \biggpar{
-\frac{\tilde\sigma_1}{\tilde\sigma}\sqrt{\frac{2}{-\log \KK}}
\frac{1+f(X)\cos\phi}{[1+f(X)]\cos\phi}, \frac{\tilde\sigma_2}{\tilde\sigma}
\dfrac{1}{\log \KK}
\dfrac{f(X)}{1+f(X)} \tan \phi} \;.
\label{Kphi17}
\end{align}

\item	There exists a constant $M_2>0$ such that for all $\KK\in(0,1)$ and all
$\phi\in\fS^1$, 
\begin{equation}
 \label{Kphi17B}
\norm{\psi_1(\KK,\phi)}^2 \leqs M_2 \KK^2 \rho(\KK,\phi)^2 \;, 
\qquad
\norm{\psi_2(\KK,\phi)}^2 \leqs M_2 \frac{\rho(\KK,\phi)^2}{\abs{\log \KK}^2}\;.
\end{equation} 
\end{itemiz}
\end{prop}
\begin{proof}
The result follows from It\^o's formula, by a straightforward though lengthy
computation. The difference between the bounds obtained for $\cos\phi\geqs0$ and
$\cos\phi<0$ is due to the fact that terms such as $f(X)\tan(\phi)/(1+f(X))$ can
be bounded by a constant times $\sqrt{-\log \KK}$ in the first case, and by a
constant times $\KK^{-\cos^2\phi}$ in the second one, as a consequence of
Lemma~\ref{lem_Kphi_f}. The fact that $-f(X)R_{\KK,\eps}$ is positive if
$c=0$ follows from the positivity of the term of order $\sqrt{\eps}$
in~\eqref{Kphi-03}.
\end{proof}


\subsection{Averaging}
\label{ssec_averaging}

In System~\eqref{Kphi-21}, the variable $\KK$ changes more slowly than the
variable $\phi$, which is a consequence of the fact that $\KK$ is a first
integral
when $\tilde\mu=\eps=\tilde\sigma=0$. This suggests to use an averaging approach
to analyse the dynamics. However, since the behaviour near $\phi=\pi$ has
already been considered in the previous section, using $(\xi,z)$-coordinates, we
only need to consider $\phi\in[\phi_0,\phi_1]$, where
$-\pi<\phi_0<0<\phi_1<\pi$. 

We look for a change of variables of the form 
\begin{equation} 
\label{Kphi-22}
\Kbar=\KK+\tilde{\mu}w(\KK,\phi)
\end{equation}
which eliminates the term of order $\tilde\mu$ in $\6\KK_t$. 
It\^o's formula yields
\begin{equation}
\label{Kphi-23}
\6\Kbar_t = \6\KK_t + \tilde{\mu} \dfrac{\partial w}{\partial \phi} \6\phi_t +
\tilde{\mu} \dfrac{\partial w}{\partial \KK} \, \6\KK_t + 
\dfrac{1}{2} \tilde{\mu} \left( \dfrac{\partial^2w}{\partial \KK^2} \6\KK_t^2 +
2\dfrac{\partial^2w}{\partial \KK \, \partial \phi} \6\KK_t \, d\phi_t +
\dfrac{\partial^2w}{\partial \phi^2} \6\phi_t^2 \right)\;.
\end{equation}
Replacing $d\KK_t$ et $d\phi_t$ by their expressions in \eqref{Kphi-21}, we get
\begin{equation}
\label{Kphi-24}
\6\Kbar_t =  \tilde{\mu} \biggpar{ f_1 + \dfrac{\partial w}{\partial \phi}
f_2 + \Order{\tilde{\mu}}+\Order{\tilde{\sigma}^2} }   \6t 
+  \tilde{\sigma} \left( \psi_1 + \tilde{\mu} \left( \dfrac{\partial w}{\partial
\phi} \psi_2 + \dfrac{\partial w}{\partial \KK} \psi_1 \right) \right)  \6W_t\;.
\end{equation}    
Thus choosing the function $w$ in such a way that
\begin{equation}
\label{Kphi-25}
f_1 + \dfrac{\partial w}{\partial \phi} f_2 = 0
\end{equation} 
will decrease the order of the drift term in~\eqref{Kphi-24}. 
We thus define the function $w$ by the integral
\begin{equation}
\label{Kphi-26}
w(\KK,\phi) = - \int_{\phi_0}^{\phi} \dfrac{f_1(\KK,\theta)}{f_2(\KK,\theta)} 
\6\theta\;, 
\end{equation}
which is well-defined (i.e., there are no resonances), since~\eqref{Kphi03}
shows that $f_2(\KK,\phi)$ is bounded below by a positive constant, for
sufficiently small $\tilde\mu$, $\eps$ and $\tilde\sigma$.

\begin{lemma}
\label{lem_Kphi_w} 
Let $\phi_0\in(-\pi,-\pi/2)$ and $\phi_1\in(\pi/2,\pi)$ be such that
$\cos^2(\phi_0), \cos^2(\phi_1) \leqs b$ for some $b\in(0,1)$. Then 
\begin{equation}
 \label{Kphi-w01}
w(\KK,\phi_1) = -\frac{\sqrt2\e}{\sqrt{-\log \KK}}
\biggbrak{\frac{\KK^{\sin^2\phi_0}}{-\sin\phi_0} +
\frac{\KK^{\sin^2\phi_1}}{\sin\phi_1} + r_1(\KK)}
\Bigpar{1 + r_2(\KK) + r_\eps(\KK)} \;,
\end{equation} 
where the remainder terms satisfy 
\begin{align}
\nonumber
r_1(\KK) &= \bigOrder{\KK \log(\abs{\log \KK})} \;, \\
\nonumber 
r_2(\KK) &= \biggOrder{\frac{1}{\abs{\log \KK}} + \tilde\mu \KK^{-b} +
\tilde\sigma^2
\biggpar{\frac{\KK^{-b}}{\tilde\mu} + \frac{\KK^{-2b}}{\abs{\log \KK}}}}\;, \\
r_\eps(\KK) &\leqs \sqrt{\eps}\,\biggOrder{\frac{\abs{\log
\KK}^2}{\tilde\mu} + \KK^{-b}\abs{\log \KK}^{3/2}}\;,
\label{Kphi-w02} 
\end{align}
and $r_\eps(\KK)\geqs0$ is $c=0$. 
Furthermore, the derivatives of $w$ satisfy the bounds 
\begin{equation}
 \label{Kphi-w02B}
\dpar{w}{\KK}(\KK,\phi) = \biggOrder{\frac{\KK^{-b}}{\sqrt{\abs{\log \KK}}}}\;, 
\qquad
\dpar{w}{\phi}(\KK,\phi) = \bigOrder{\KK^{1-b}\sqrt{\abs{\log \KK}}}\;,
\end{equation} 
and 
\begin{equation}
 \label{Kphi-w02C}
\dpar{^2w}{\KK^2} =  \biggOrder{\frac{\KK^{-1-b}}{\sqrt{\abs{\log \KK}}}}\;,
\quad
\dpar{^2w}{\KK\partial\phi} = \biggOrder{\frac{\KK^{-b}}{\sqrt{\abs{\log
K}}}}\;, 
\quad
\dpar{^2w}{\phi^2} = \bigOrder{\KK^{1-b}\sqrt{\abs{\log \KK}}}\;.
\end{equation} 
\end{lemma}
\begin{proof}
We split the integral into three parts. Using the change of variables $t = \sin
\phi$ and a partial fraction decomposition, we find that the leading part of the
integral on $[-\pi/2,\pi/2]$ satisfies 
\begin{equation}
 \label{Kphi-w03}
\int_{-\pi/2}^{\pi/2} \frac{4\KK X(\KK,\phi)}{1+f(X(\KK,\phi))}\6\phi = 
\biggOrder{\KK\frac{\log(\abs{\log \KK})}{\sqrt{\abs{\log \KK}}}}\;. 
\end{equation} 
Next we consider the integral on $[\phi_0,-\pi/2]$. The change of variables
$u=\sqrt{-2\log \KK}\sin\phi$, \eqref{Kphi04B} and asymptotic properties of
the error function imply 
\begin{align}
\nonumber
\int_{\phi_0}^{-\pi/2} \frac{4\KK X}{1+f(X)}\6\phi 
&= 2\KK \int_{\sin\phi_0\sqrt{-2\log \KK}}^{-\sqrt{-2\log \KK}}
\frac{\6u}{1+f\Bigpar{-\sqrt{\frac{-\log \KK}{2} - \frac{u^2}4}}}\\
\label{KKphi-w04}
&= 2\e \int_{\sin\phi_0\sqrt{-2\log \KK}}^{-\sqrt{-2\log \KK}}
\e^{-u^2/2} \bigbrak{1+\Order{\KK\e^{-u^2/2}}} \6u \\
\nonumber
&= -2\e \frac{\KK^{\sin^2\phi_0}}{\sqrt{-2\log \KK}(-\sin\phi_0)}
\biggbrak{1+\biggOrder{\frac{1}{\abs{\log \KK}}} +
\Order{\KK^{\cos^2\phi_0}}}\;.
\end{align}
The integral on $[\pi/2,\phi_1]$ can be computed in a similar way. This yields
the leading term in~\eqref{Kphi-w01}, and the form of the remainders follows
from~\eqref{Kphi-11B} with $\rho=\KK^{-b}$. 
The bound on $\tdpar w\phi$ follows directly from~\eqref{Kphi-25}, while the
bound on $\tdpar w\KK$ is obtained by computing the derivative of $f_1/f_2$. 
The bounds on second derivatives follow by similar computations.
\end{proof}

Notice that for the remainder $r_2(\KK)$ to be small, we need that
$\KK^b\gg\tilde\mu$ and $\KK^b\gg\tilde\sigma^2/\tilde\mu$. Then the term
$r_\eps(\KK)$ is of order $\sqrt{\eps}\abs{\log\tilde\mu}^2/\tilde\mu$, which is
small for $\tilde\mu/\abs{\log\tilde\mu}^2\gg\sqrt{\eps}$. If that is the case,
then $w(\KK,\phi_1)$ has order $\KK^{1-b}/\sqrt{\abs{\log \KK}}$. Otherwise,
$w(\KK,\phi_1)$ has order $\sqrt{\eps}\KK^{1-b}\abs{\log \KK}^{3/2}/\tilde\mu$. 
In the sequel, we will sometimes bound $1/\sqrt{\abs{\log \KK}}$ by $1$ to get
simpler expressions. 


\subsection{Computation of the kernel}
\label{Kphi_kernel} 

We can now proceed to the computation of the rotational part of the
kernel of the Markov chain. Recall the broken lines $F_\pm$ introduced
in~\eqref{bounds10} and~\eqref{bounds11}. For an initial condition $(L,z_0)\in
F_+$, we want to compute the coordinates of the point $(\xi_\tau,z_\tau)$ at the
first time 
\begin{equation}
\label{Kphi-29}
\tau = \inf \bigsetsuch{ t>0 }{(\xi_t,z_t) \in F_- }
\end{equation}
that the path starting in $(L,z_0)$ hits $F_-$. 

We will assume that there is a $\beta\in(0,1]$ such that 
\begin{equation}
\label{Kphi-30}
(c_- \tilde{\mu})^\beta \leqs z_0 \leqs z_{\max} < \dfrac{1}{2}\;.
\end{equation}
The $(\KK,\phi)$-coordinates of the initial condition are given by 
\begin{align}
\nonumber
\KK_0 &= 2z_0\e^{1-2z_0}\e^{-2L^2} \geqs
2(c_-\tilde\mu)^{\beta+\gamma}\;, \\
\sin^2\phi_0 &= \frac{2L^2}{-\log \KK_0} \geqs \frac{\gamma}{\beta+\gamma}\;,
\label{Kphi-30A}
\end{align}
with $\phi_0\in(-\pi,-\pi/2)$. Thus Lemma~\ref{lem_Kphi_w} applies with
$b=\beta/(\beta+\gamma)<1$. Notice that 
\begin{equation}
 \label{Kphi-30B}
\KK_0^{\cos^2\phi_0} \geqs 2^b(c_-\tilde\mu)^\beta\;. 
\end{equation} 

\begin{prop}
\label{prop_Kphi_kernel} 
Assume $z_0$ satisfies~\eqref{Kphi-30} for a $\beta<1$. Then there exists a
constant $\kappa>0$ such that the following holds for sufficiently small
$\tilde\mu$ and $\tilde\sigma$. 
\begin{enum}
\item	If $\sqrt{\eps}\leqs\tilde\mu/{\abs{\log\tilde\mu}^2}$, then 
with probability greater or equal than 
\begin{equation}
 \label{Kphi-31A}
1 - \e^{-\kappa\tilde\mu^2/\tilde\sigma^2}\;, 
\end{equation} 
$(\xi_t,z_t)$ hits $F_-$ for the first time at a point $(-L,z_1)$ such that
\begin{equation}
\label{Kphi-31}
z_1 = z_0 + \tilde\mu A(z_0) 
+ \frac{z_0}{1-2z_0} \Bigbrak{\tilde\sigma V(z_0) +
\bigOrder{\tilde\mu^{2(1-\beta)} + \tilde\sigma^2\tilde\mu^{-2\beta}} }\;.
\end{equation}
The function $A(z_0)$ is given by 
\begin{equation}
\label{Kphi-32}
A(z_0) = \frac{\e^{2z_0}}{L(1-2z_0)}
\biggbrak{1 + \bigOrder{z_0\log\abs{\log\tilde\mu}} +
\biggOrder{\frac{1}{\abs{\log z_0}}}}\;,
\end{equation}
and $V(z_0)$ is a random variable satisfying
\begin{equation}
\label{Kphi-35}
\bigprob{\tilde\sigma \abs{V(z_0)} \geqs h} \leqs 2
\exp\biggset{-\frac{\kappa h^2\tilde\mu^{2\beta}}{\tilde\sigma^2}}
\qquad \forall h>0\;.
\end{equation}

\item	If $c=0$ and $\sqrt{\eps}>\tilde\mu/{\abs{\log\tilde\mu}^2}$, then
$(\xi_t,z_t)$ hits $F_-$ for the first time either at a point $(-L,z_1)$ such
that $z_1$ is greater or equal than the right-hand side of~\eqref{Kphi-31}, or
at a point $(\xi_1,1/2)$ with $-L\leqs\xi_1\leqs 0$, again with a probability
bounded below by~\eqref{Kphi-31A}.
\end{enum}
\end{prop}

\begin{proof} 
We first consider the case
$\sqrt{\eps}\leqs\tilde\mu/{\abs{\log\tilde\mu}^2}$. 
\begin{itemiz}
\item{{\bf Step 1~:}}
To be able to bound various error terms, we need to assume that $\KK_t$ stays
bounded below. We thus introduce a second stopping time 
\begin{equation}
 \label{Kphi-35:1}
\tau_1 = \inf\bigsetsuch{t>0}{\cos\phi_t<0, \KK_t^{\cos^2\phi_t} <
(c_-\tilde\mu)^{\beta}}\;. 
\end{equation}
We start by showing that $\tau\wedge\tau_1$ is bounded with high probability.
Proposition~\ref{prop_Kphi_EDS} implies the existence of a
constant $C>0$ such that 
\begin{equation}
\label{Kphi-36}
\6\phi_t \geqs C \6t + \tilde{\sigma} \psi_2(\KK_t,\phi_t) \6W_t\;.
\end{equation}
Integrating this relation between $0$ and $t$, we get 
\begin{equation}
\label{Kphi-37}
\phi_{t} \geqs \phi_0 + Ct + \tilde{\sigma} \int_0^t \psi_2(\KK_s,\phi_s)
\6W_s\;. 
\end{equation}
Lemma~\ref{lem_Bernstein} and~\eqref{Kphi17B} provide the bound
\begin{equation}
\label{Kphi-38}
\biggprob{ \biggabs{ \tilde\sigma\int_0^{t\wedge\tau_1} \psi_2(\KK_s,\phi_s)
\6W_s
} \geqs h } 
\leqs \exp\biggset{-\frac{\kappa h^2 \tilde\mu^{2\beta}}{\tilde\sigma^2}}
\end{equation}
for some $\kappa>0$. 
Since by definition, $\phi_{\tau\wedge\tau_1}-\phi_0 < 2\pi$, we get 
\begin{equation}
\label{Kphi-39}
\biggprob{\tau\wedge\tau_1 > \frac{2\pi+h}{C}} \leqs \exp\biggset{-\frac{\kappa
h^2 \tilde\mu^{2\beta}}{\tilde\sigma^2}}\;.
\end{equation}
From now on, we work on the set $\Omega_1=\set{\tau\wedge\tau_1 \leqs
(2\pi+1)/C}$, which has probability greater or equal $1-\e^{-\kappa
\tilde\mu^{2\beta}/\tilde\sigma^2}$. 

\item{{\bf Step 2~:}}
The SDE~\eqref{Kphi-23} for $\Kbar_t$ can be written 
\begin{equation}
 \label{Kphi-39B}
\6\Kbar_t = \Kbar_t\bar f(\Kbar_t,\phi_t)\6t + \tilde\sigma\Kbar_t\bar\psi
(\Kbar_t,\phi_t)\6W_t\;,
\end{equation} 
where the bounds in Proposition~\ref{prop_Kphi_EDS} and
Lemma~\ref{lem_Kphi_w} yield 
\begin{align}
\nonumber
\bar f(\Kbar,\phi) &= \bigOrder{\tilde\mu^{2(1-\beta)} +
\tilde\sigma^2 \tilde\mu^{1-3\beta}}\;, \\
\norm{\bar\psi(\Kbar,\phi)}^2 &= \bigOrder{\tilde\mu^{-2\beta}}\;.
 \label{Kphi-39C}
\end{align}
By It\^o's formula, the variable $\QQ_t=\log\Kbar_t$ satisfies 
\begin{equation}
 \label{Kphi-39D} 
\6\QQ_t = \tilde f(\QQ_t,\phi_t)\6t + \tilde\sigma\tilde\psi
(\QQ_t,\phi_t)\6W_t\;,
\end{equation} 
where $\tilde f(\QQ,\phi) = \bar f(\e^\QQ,\phi) +
\Order{\tilde\sigma^2\tilde\mu^{-2\beta}}$ and
$\tilde\psi(\QQ,\phi)=\bar\psi(\e^\QQ,\phi)$. 
Setting 
\begin{equation}
 \label{Kphi-39E}
V_t = \int_0^t \tilde\psi(\QQ_s,\phi_s)\6W_s\;,
\end{equation} 
we obtain, integrating~\eqref{Kphi-39D} and using the fact that
$\tilde\mu^{1-3\beta}\leqs\tilde\mu^{-2\beta}$, 
\begin{equation}
\label{Kphi-40}
\QQ_t = \QQ_0 + \tilde{\sigma} V + \bigOrder{
\tilde\mu^{2(1-\beta)} +
\tilde\sigma^2 \tilde\mu^{-2\beta}}\;.
\end{equation}
Another application of Lemma~\ref{lem_Bernstein} yields 
\begin{equation}
\label{Kphi-41}
\bigprob{ \tilde\sigma \abs{V_{t\wedge\tau_1}} \geqs h_1 } 
\leqs 2\exp\biggset{-\frac{\kappa_1 h_1^2 \tilde\mu^{2\beta}}{\tilde\sigma^2}}
\end{equation}
for some $\kappa_1>0$. 
A convenient choice is $h_1=\tilde\mu^{1-\beta}$. 
From now on, we work on the set
$\Omega_1\cap\Omega_2$, where $\Omega_2=\set{\tilde\sigma
V_{t\wedge\tau_1}<\tilde\mu^{1-\beta}}$ satisfies 
$\fP(\Omega_2) \geqs 1-\e^{-\kappa_1\tilde\mu^2/\tilde\sigma^2}$.

\item{{\bf Step 3~:}}
Returning to the variable $\Kbar$, we get 
\begin{equation}
 \label{Kphi-41B}
\Kbar_t = \KK_0 \e^{\tilde\sigma V_t} \Bigbrak{1+\bigOrder{
\tilde\mu^{2(1-\beta)} + \tilde\sigma^2 \tilde\mu^{-2\beta}}}\;,
\end{equation} 
and thus 
\begin{equation}
 \label{Kphi-41C}
\KK_t =  \KK_0 \e^{\tilde\sigma V_t} \Bigbrak{1+\bigOrder{
\tilde\mu^{2(1-\beta)} + \tilde\sigma^2 \tilde\mu^{-2\beta}}} 
- \tilde\mu w(\KK_t,\phi_t)\;.
\end{equation} 
Using the implicit function theorem and the upper bound on $w$, we get the a
priori bound
\begin{equation}
 \label{Kphi-41D}
\frac{\abs{\KK_t-\KK_0}}{\KK_0}  = 
\bigOrder{\tilde\mu^{1-\beta}+\tilde\sigma^2\tilde\mu^{-2\beta}}\;. 
\end{equation} 

\item{{\bf Step 4~:}}
The a priori estimate~\eqref{Kphi-41D} implies that on $\Omega_1\cap\Omega_2$, 
the sample path cannot hit
$F_-$ on the part $\set{-L\leqs\xi\leqs0, z=1/2}$. 
Indeed, this would imply that $\KK_\tau\geqs(c_-\tilde\mu)^\gamma$, while
$\KK_0\leqs a(c_-\tilde\mu)^\gamma$ with $a=2z_{\max}\e^{1-2z_{\max}}<1$. 
As a consequence, we would have $(\KK_\tau-\KK_0)/\KK_0 > (1-a)/a$,
contradicting~\eqref{Kphi-41D}.

Let us now show that we also have $\tau_1 \geqs \tau$ on $\Omega_1\cap\Omega_2$.
Assume by contradiction that $\tau_1<\tau$. Then we have 
$\KK_{\tau_1}=(c_-\tilde\mu)^\beta$ and
$\cos^2\phi_{\tau_1}<\beta/(\beta+\gamma)$, so that
$\KK_{\tau_1}=\order{\tilde\mu^{\beta+\gamma}}$. 
Thus $\tilde\mu w(\KK_{\tau_1},\phi_{\tau_1}) =
\Order{\KK_{\tau_1}\tilde\mu^{1-\beta}} = \order{\tilde\mu^{1+\gamma}}$. 
Together with the lower bound~\eqref{Kphi-30A} on $\KK_0$, this implies that the
right-hand side of~\eqref{Kphi-41C} is larger than a
constant times $\tilde\mu^{\beta+\gamma}$ at time $t=\tau_1$. But this
contradicts the fact that $\KK_{\tau_1}=\order{\tilde\mu^{\beta+\gamma}}$. 

\item{{\bf Step 5~:}} 
The previous step implies that $\xi_\tau=-L$ on $\Omega_1\cap\Omega_2$. We can
thus write 
\begin{equation}
 \label{Kphi-42A}
\phi_\tau = g(\KK_\tau) 
\qquad\text{where}\qquad
\sin(g(\KK)) = \sqrt{\frac{2}{-\log \KK}} \, L\;, 
\end{equation} 
with $g(\KK)\in(\pi/2,\pi)$. Notice that $g(\KK_0)=-\phi_0$. 
Furthermore, we have 
\begin{equation}
 \label{Kphi-42B}
g'(\KK) = \frac{L}{\sqrt{2}\KK(-\log \KK)^{3/2} \cos(g(\KK))}\;, 
\end{equation} 
and thus $\KK_\tau g'(\KK_\tau) = \Order{1/\abs{\log\tilde\mu}}$.
Using this in the Taylor expansion 
\begin{equation}
 \label{Kphi-42C}
w(\KK_\tau,\phi_\tau) = w(\KK_0,-\phi_0) + (\KK_\tau-\KK_0) 
\biggbrak{\dpar w\KK(\KK_\theta,g(\KK_\theta)) + 
\dpar w\phi(\KK_\theta,g(\KK_\theta))g'(\KK_\theta)} \;,
\end{equation} 
which holds for some $\KK_\theta\in(\KK_0,\KK_\tau)$, yields the estimate 
\begin{equation}
 \label{Kphi-42D}
\frac{w(\KK_\tau,\phi_\tau)}{\KK_0}
= \frac{w(\KK_0,-\phi_0)}{\KK_0} + 
\bigOrder{\tilde\mu^{1-2\beta} + \tilde\sigma^2\tilde\mu^{-3\beta}}\;. 
\end{equation} 
Substitution in~\eqref{Kphi-41C} yields the more precise estimate 
\begin{equation}
 \label{Kphi-42E}
\KK_1 = \KK_0 \biggbrak{\e^{\tilde\sigma V_\tau} - \tilde\mu
\frac{w(\KK_0,-\phi_0)}{\KK_0} 
+ \bigOrder{\tilde\mu^{2(1-\beta)} + \tilde\sigma^2\tilde\mu^{1-3\beta}}}\;. 
\end{equation} 

\item{{\bf Step 6~:}} 
Finally, we return to the variable $z_1=z_\tau$. Eliminating $\phi$ from the
equations~\eqref{Kphi11}, it can be expressed in terms of $\KK_\tau$ as
\begin{equation}
\label{Kphi-45}
z_1 = G(\KK_\tau) 
\defby \dfrac{1}{2} \Biggbrak{ 1 + f \biggpar{ -\sqrt{\frac{- \log
\KK_\tau}{2}-L^2}\, }}\;.
\end{equation}
Note that $G(\KK_0)=z_0$, while 
\begin{equation}
 \label{Kphi-45B}
G'(\KK_0) = -\frac{1}{2\KK_0}
\frac{1+f\Bigpar{- \sqrt{\frac{-\log \KK_0}{2} - L^2}\,}}
{f\Bigpar{-\sqrt{\frac{-\log \KK_0}{2} - L^2}\,}}  
= \frac{z_0}{\KK_0(1-2z_0)}\;,
\end{equation} 
and
\begin{equation}
 \label{Kphi-45C}
G''(\KK) = \frac{1}{2\KK^2}  \frac{1+f\Bigpar{- \sqrt{\frac{-\log \KK}{2} -
L^2}\,}}{f\Bigpar{-\sqrt{\frac{-\log \KK}{2} - L^2}\,}} 
\left[1-\frac{1}{f\Bigpar{-\sqrt{\frac{-\log \KK}{2} - L^2}\,}^2}\right]\;,
\end{equation} 
which has order $z_0^2/\KK_0^2$. 
The Taylor expansion 
\begin{equation}
\label{Kphi-46}
z_1 = G(\KK_0) + (\KK_1 - \KK_0) G'(\KK_0) + \dfrac{(\KK_1 - \KK_0) ^2}{2}
G''(\KK_{\theta})
\end{equation}
thus becomes 
\begin{equation}
 \label{Kphi-47} 
z_1 = z_0 + \frac{\KK_1-\KK_0}{\KK_0} \frac{z_0}{1-2z_0} 
+ \biggOrder{\biggbrak{\frac{\KK_1-\KK_0}{\KK_0}z_0}^2}\;.
\end{equation} 
By~\eqref{Kphi-42E}, we have 
\begin{equation}
 \label{Kphi-48}
 \frac{\KK_1-\KK_0}{\KK_0} = \tilde\sigma V_\tau 
- \tilde\mu \frac{w(\KK_0,-\phi_0)}{\KK_0} 
+ \bigOrder{\tilde\mu^{2(1-\beta)} + \tilde\sigma^2\tilde\mu^{-2\beta}}\;,
\end{equation} 
and Lemma~\ref{lem_Kphi_w} yields 
\begin{equation}
 \label{Kphi-49}
- \mu \frac{w(\KK_0,-\phi_0)}{\KK_0} = 
\frac{2\e}{L} \biggbrak{\KK_0^{-\cos^2\phi_0} + \Order{\log\abs{\log \KK_0}}}
\biggbrak{1+\biggOrder{\frac{1}{\abs{\log \KK_0}}}}\;.
\end{equation} 
Now~\eqref{Kphi-30A} implies $\KK_0^{-\cos^2\phi_0}=\e^{2z_0}(2\e z_0)^{-1}$ and
$c_1\abs{\log z_0} \leqs \abs{\log \KK_0} \leqs c_2 \abs{\log\tilde\mu}$.
This completes the proof of the case 
$\sqrt{\eps}\leqs\tilde\mu/{\abs{\log\tilde\mu}^2}$.
\end{itemiz}
In the case $\sqrt{\eps}>\tilde\mu/{\abs{\log\tilde\mu}^2}$, we just use the
fact that $\KK_t$ is bounded below by its value in the previous case, as a
consequence of \eqref{Kphi-w02B}.
\end{proof}

\begin{cor}
\label{cor_Kphi} 
Assume that either $c=0$ or $\sqrt{\eps}\leqs\tilde\mu/\abs{\log\tilde\mu}^2$. 
There exists a $\kappa_2>0$ such that for an initial condition $(L,z_0)\in F_-$
with $z_0\geqs(c_-\tilde\mu)^{1-\gamma}$, the first hitting of $F_+$ occurs at
a height $z_1\geqs z_0$ with probability larger than
$1-\e^{-\kappa_2\tilde\mu^2/\tilde\sigma^2}$.  
\end{cor}
\begin{proof}
It suffices to apply the previous result with $\beta=1-\gamma$ and
$h$ of order $\tilde\mu^{\gamma}$. 
\end{proof}

\small

\bibliography{../BL}

\newcommand{\etalchar}[1]{$^{#1}$}
\def\cprime{$'$}
\providecommand{\bysame}{\leavevmode\hbox to3em{\hrulefill}\thinspace}
\providecommand{\MR}{\relax\ifhmode\unskip\space\fi MR }
\providecommand{\MRhref}[2]{%
  \href{http://www.ams.org/mathscinet-getitem?mr=#1}{#2}
}
\providecommand{\href}[2]{#2}
\begin{thebibliography}{BKLLC11}

\bibitem[BAKS84]{BenArous_Kusuoka_Stroock_1984}
G{\'e}rard Ben~Arous, Shigeo Kusuoka, and Daniel~W. Stroock, \emph{The
  {P}oisson kernel for certain degenerate elliptic operators}, J. Funct. Anal.
  \textbf{56} (1984), no.~2, 171--209.

\bibitem[BE86]{BaerErneuxI}
S.M. Baer and T.~Erneux, \emph{{Singular Hopf bifurcation to relaxation
  oscillations I}}, SIAM J. Appl. Math. \textbf{46} (1986), no.~5, 721--739.

\bibitem[BE92]{BaerErneuxII}
\bysame, \emph{{Singular Hopf bifurcation to relaxation oscillations II}}, SIAM
  J. Appl. Math. \textbf{52} (1992), no.~6, 1651--1664.

\bibitem[BG02]{BG1}
Nils Berglund and Barbara Gentz, \emph{Pathwise description of dynamic
  pitchfork bifurcations with additive noise}, {Probab.} {Theory} {Related}
  {Fields} \textbf{122} (2002), no.~3, 341--388.

\bibitem[BG09]{BG_neuro09}
\bysame, \emph{Stochastic dynamic bifurcations and excitability}, Stochastic
  Methods in Neuroscience (Carlo Laing and Gabriel Lord, eds.), Oxford
  University Press, 2009, pp.~64--93.

\bibitem[BG11]{Baxendale_Greenwood_11}
Peter~H. Baxendale and Priscilla~E. Greenwood, \emph{Sustained oscillations for
  density dependent {M}arkov processes}, J. Math. Biol. \textbf{63} (2011),
  no.~3, 433--457.

\bibitem[BGK12]{BGK12}
Nils Berglund, Barbara Gentz, and Christian Kuehn, \emph{Hunting {F}rench ducks
  in a noisy environment}, J. Differential Equations \textbf{252} (2012),
  4786--4841.

\bibitem[Bir57]{Birkhoff1957}
Garrett Birkhoff, \emph{Extensions of {J}entzsch's theorem}, Trans. Amer. Math.
  Soc. \textbf{85} (1957), 219--227.

\bibitem[BKLLC11]{Borowski_Kuske_etal_2011}
Peter Borowski, Rachel Kuske, Yue-Xian Li, and Juan Luis~Cabrera,
  \emph{Characterizing mixed mode oscillations shaped by noise and bifurcation
  structure}, Chaos \textbf{20} (2011), no.~4, 043117.

\bibitem[Bra98]{Braaksma}
B.~Braaksma, \emph{{Singular Hopf bifurcation in systems with fast and slow
  variables}}, Journal of Nonlinear Science \textbf{8} (1998), no.~5, 457--490.

\bibitem[CR71]{CapocelliRicciardi71}
R.~M. Capocelli and L.~M. Ricciardi, \emph{Diffusion approximation and first
  passage time problem for a model neuron}, Kybernetik (Berlin) \textbf{8}
  (1971), no.~6, 214--223.

\bibitem[Dah77]{Dahlberg1977}
Bj{\"o}rn E.~J. Dahlberg, \emph{Estimates of harmonic measure}, Arch. Rational
  Mech. Anal. \textbf{65} (1977), no.~3, 275--288.

\bibitem[DG11]{Ditlevsen_Greenwood_11}
Susanne Ditlevsen and Priscilla Greenwood, \emph{The {M}orris--{L}ecar neuron
  model embeds a leaky integrate-and-fire model}, Preprint {\tt
  arXiv:1108.0073}, 2011.

\bibitem[DGK{\etalchar{+}}11]{KuehnMMO}
M.~Desroches, J.~Guckenheimer, C.~Kuehn, B.~Krauskopf, H.~Osinga, and
  M.~Wechselberger, \emph{Mixed-mode oscillations with multiple time scales},
  SIAM Review, in press (2011).

\bibitem[DMS{\etalchar{+}}00]{Dicksonetal1}
C.T. Dickson, J.~Magistretti, M.H. Shalisnky, E.~Fransen, M.E. Hasselmo, and
  A.~Alonso, \emph{{Properties and role of $I_h$ in the pacing of subtreshold
  oscillations in entorhinal cortex layer II neurons}}, J. Neurophysiol.
  \textbf{83} (2000), 2562--2579.

\bibitem[DOP79]{DegnOlsenPerram}
H.~Degn, L.F. Olsen, and J.W. Perram, \emph{Bistability, oscillation, and chaos
  in an enzyme reaction}, Annals of the New York Academy of Sciences
  \textbf{316} (1979), no.~1, 623--637.

\bibitem[DT09]{DossThieullen2009}
Catherine Doss and Mich\`ele Thieullen, \emph{Oscillations and random
  perturbations of a {FitzHugh-Nagumo} system}, Preprint hal-00395284 (2009),
  2009.

\bibitem[Fit55]{Fitzhugh}
R.~FitzHugh, \emph{Mathematical models of threshold phenomena in the nerve
  membrane}, Bull. Math. Biophysics \textbf{17} (1955), 257--269.

\bibitem[Fit61]{Fitzhugh61}
R.~FitzHugh, \emph{Impulses and physiological states in models of nerve
  membrane}, Biophys. J. \textbf{1} (1961), 445--466.

\bibitem[GM64]{GersteinMandelbrot64}
George~L. Gerstein and Beno\^{i}t~E. Mandelbrot, \emph{Random walk models for
  the spike activity of a single neuron}, Biophys. J. \textbf{4} (1964),
  41--68.

\bibitem[HH52]{HodgkinHuxley52}
A.~L. Hodgkin and A.~F. Huxley, \emph{A quantitative description of ion
  currents and its applications to conduction and excitation in nerve
  membranes}, J. Physiol. (Lond.) \textbf{117} (1952), 500--544.

\bibitem[HHM79]{HudsonHartMarinko}
J.L. Hudson, M.~Hart, and D.~Marinko, \emph{An experimental study of multiple
  peak periodic and nonperiodic oscillations in the {Belousov-Zhabotinskii}
  reaction}, J. Chem. Phys. \textbf{71} (1979), no.~4, 1601--1606.

\bibitem[HM09]{HiczenkoMedvedev2009}
Pawel Hitczenko and Georgi~S. Medvedev, \emph{Bursting oscillations induced by
  small noise}, SIAM J. Appl. Math. \textbf{69} (2009), no.~5, 1359--1392.

\bibitem[Izh00]{Izhikevich00}
Eugene~M. Izhikevich, \emph{Neural excitability, spiking and bursting},
  Internat. J. Bifur. Chaos Appl. Sci. Engrg. \textbf{10} (2000), no.~6,
  1171--1266.

\bibitem[Jen12]{Jentzsch1912}
Robert Jentzsch, \emph{{\"U}ber {I}ntegralgleichungen mit positivem {K}ern},
  {J}. f. d. reine und angew. {M}ath. \textbf{141} (1912), 235--244.

\bibitem[KP03]{KosmidisPakdaman}
Efstratios~K. Kosmidis and K.~Pakdaman, \emph{An analysis of the reliability
  phenomenon in the {F}itz{H}ugh--{N}agumo model}, J. Comput. Neuroscience
  \textbf{14} (2003), 5--22.

\bibitem[KP06]{KosmidisPakdaman2006}
\bysame, \emph{Stochastic chaos in a neuronal model}, Internat. J. Bifur. Chaos
  \textbf{16} (2006), no.~2, 395--410.

\bibitem[KR50]{KreinRutman1950}
M.~G. Kre{\u\i}n and M.~A. Rutman, \emph{Linear operators leaving invariant a
  cone in a {B}anach space}, Amer. Math. Soc. Translation \textbf{1950} (1950),
  no.~26, 128.

\bibitem[Lon93]{Longtin}
Andr\'e Longtin, \emph{Stochastic resonance in neuron models}, J.~Stat.\ Phys.
  \textbf{70} (1993), 309--327.

\bibitem[Lon00]{Longtin2000}
\bysame, \emph{Effect of noise on the tuning properties of excitable systems},
  Chaos, Solitons and Fractals \textbf{11} (2000), 1835--1848.

\bibitem[LSG99]{Lindner_Schimansky_1999}
Benjamin Lindner and Lutz Schimansky-Geier, \emph{Analytical approach to the
  stochastic {F}itz{H}ugh-{N}agumo system and coherence resonance}, Physical
  Review E \textbf{60} (1999), no.~6, 7270--7276.

\bibitem[ML81]{MorrisLecar81}
C.~Morris and H.~Lecar, \emph{Voltage oscillations in the barnacle giant muscle
  fiber}, Biophys. J. (1981), 193--213.

\bibitem[MVE08]{MuratovVandeneijnden2007}
Cyrill~B. Muratov and Eric Vanden-Eijnden, \emph{Noise-induced mixed-mode
  oscillations in a relaxation oscillator near the onset of a limit cycle},
  Chaos \textbf{18} (2008), 015111.

\bibitem[MVEE05]{MuratovVanden-EijndenE}
C.B. Muratov, E.~Vanden-Eijnden, and W.~E, \emph{Self-induced stochastic
  resonance in excitable systems}, Physica D \textbf{210} (2005), 227--240.

\bibitem[NAY62]{Nagumo62}
J.~S. Nagumo, S.~Arimoto, and S.~Yoshizawa, \emph{An active pulse transmission
  line simulating nerve axon}, Proc. IRE \textbf{50} (1962), 2061--2070.

\bibitem[Num84]{Nummelin84}
Esa Nummelin, \emph{General irreducible {M}arkov chains and nonnegative
  operators}, Cambridge Tracts in Mathematics, vol.~83, Cambridge University
  Press, Cambridge, 1984.

\bibitem[Ore71]{OreyBook}
Steven Orey, \emph{Lecture notes on limit theorems for {M}arkov chain
  transition probabilities}, Van Nostrand Reinhold Co., London, 1971, Van
  Nostrand Reinhold Mathematical Studies, No. 34.

\bibitem[PSS92]{PetrovScottShowalter}
V.~Petrov, S.K. Scott, and K.~Showalter, \emph{Mixed-mode oscillations in
  chemical systems}, J. Chem. Phys. \textbf{97} (1992), no.~9, 6191--6198.

\bibitem[Row07]{Rowat_2007}
Peter Rowat, \emph{Interspike interval statistics in the stochastic
  {H}odgkin-{H}uxley model: Coexistence of gamma frequency bursts and highly
  irregular firing}, Neural Computation \textbf{19} (2007), 1215–1250.

\bibitem[RS80]{RicciardiSacerdote80}
L.~M. Ricciardi and L.~Sacerdote, \emph{The first passage time problem with
  applications to neuronal modeling}, Second {I}nternational {C}onference on
  {I}nformation {S}ciences and {S}ystems ({U}niv. {P}atras, {P}atras, 1979),
  {V}ol. {III}, Reidel, Dordrecht, 1980, pp.~226--236.

\bibitem[SK11]{SimpsonKuske_2011}
D.~W.~J. Simpson and R.~Kuske, \emph{Mixed-mode oscillations in a stochastic,
  piecewise-linear system}, Physica D \textbf{240} (2011), 1189--1198.

\bibitem[Sow08]{Sowers08}
Richard~B. Sowers, \emph{Random perturbations of canards}, J. Theoret. Probab.
  \textbf{21} (2008), no.~4, 824--889.

\bibitem[SVJ66]{Seneta_VereJones_1966}
E.~Seneta and D.~Vere-Jones, \emph{On quasi-stationary distributions in
  discrete-time {M}arkov chains with a denumerable infinity of states}, J.
  Appl. Probability \textbf{3} (1966), 403--434.

\bibitem[TGOS08]{Turcotte_2008}
Marc Turcotte, Jordi Garcia-Ojalvo, and {Gurol M.} ̈~̈ Suel, \emph{A genetic
  timer through noise-induced stabilization of an unstable state}, PNAS
  \textbf{105} (2008), no.~41, 15732--15737.

\bibitem[TP01a]{TanabePakdaman_PRE2001}
Seiji Tanabe and K.~Pakdaman, \emph{Dynamics of moments of
  {F}itz{H}ugh--{N}agumo neuronal models and stochastic bifurcations}, Phys.
  Rev. E \textbf{63} (2001), 031911.

\bibitem[TP01b]{TanabePakdaman2001}
Seiji Tanabe and K.~Pakdaman, \emph{Noise-induced transition in excitable
  neuron models}, Biol. Cybern. \textbf{85} (2001), 269--280.

\bibitem[TRW03]{Tuckwell_etal_2003}
Henry~C. Tuckwell, Roger Rodriguez, and Frederic Y.~M. Wan, \emph{Determination
  of firing times for the stochastic {F}itzhugh-{N}agumo neuronal model},
  Neural Computation \textbf{15} (2003), 143--159.

\bibitem[TTP02]{TakahataTanabePakdaman2002}
Takayuki Takahata, Seiji Tanabe, and K.~Pakdaman, \emph{White-noise stimulation
  of the {H}odgkin–{H}uxley model}, Biol. Cybern. \textbf{86} (2002),
  403–417.

\bibitem[Tuc75]{Tuckwell75}
Henry~C. Tuckwell, \emph{Determination of the inter-spike times of neurons
  receiving randomly arriving post-synaptik potentials}, Biol. Cybernetics
  \textbf{18} (1975), 225--237.

\bibitem[Tuc77]{Tuckwell77}
\bysame, \emph{On stochastic models of the activity of single neurons}, J.
  Theor. Biol. \textbf{65} (1977), 783--785.

\bibitem[Tuc89]{Tuckwell}
\bysame, \emph{Stochastic processes in the neurosciences}, SIAM, Philadelphia,
  PA, 1989.

\end{thebibliography}
\bibliographystyle{amsalpha}

\bigskip


\tableofcontents

\vfill
\noindent
{\small
Nils Berglund and Damien Landon \\ 
Universit\'e d'Orl\'eans, Laboratoire {\sc Mapmo} \\
{\sc CNRS, UMR 7349} \\
F\'ed\'eration Denis Poisson, FR 2964 \\
B\^atiment de Math\'ematiques, B.P. 6759\\
45067~Orl\'eans Cedex 2, France \\
{\it E-mail address: }{\tt nils.berglund@univ-orleans.fr, 
damien.landon@univ-orleans.fr}

}


\end{document}